\documentclass[12pt]{amsart}
\usepackage{amscd,amssymb,amsthm,amsmath,amssymb,mathrsfs,enumerate, longtable,diagbox}
\usepackage[matrix,arrow,curve]{xy}
\usepackage[margin=2cm]{geometry}
\usepackage{hyperref}
\usepackage{tikz}
\usetikzlibrary{automata, positioning}

\newtheorem{theorem}{Theorem}[section]

\newtheorem{lemma}[theorem]{Lemma}
\newtheorem*{lemma*}{Lemma}

\newtheorem*{maintheorem*}{Main Theorem}
\newtheorem*{corollary*}{Corollary}
\newtheorem*{conjecture*}{Conjecture}

\theoremstyle{definition}

\newtheorem*{example*}{Example}

\theoremstyle{remark}
\newtheorem{remark}[theorem]{Remark}

\makeatletter\@addtoreset{equation}{section} \makeatother

\renewcommand\arraystretch{1.44}

\makeatletter
\renewcommand*\env@matrix[1][\arraystretch]{%
  \edef\arraystretch{#1}%
  \hskip -\arraycolsep
  \let\@ifnextchar\new@ifnextchar
  \array{*\c@MaxMatrixCols c}}
\makeatother

\author{Ivan Cheltsov, Kento Fujita, Takashi Kishimoto, Jihun Park}

\title{K-stable Fano 3-folds in the~families \textnumero 2.18 and \textnumero 3.4}

\begin{document}

\begin{abstract}
We prove that smooth Fano 3-folds in the~families \textnumero 2.18 and \textnumero 3.4 are K-stable.
\end{abstract}

\subjclass[2010]{14J45, 32Q20.}

\address{ \emph{Ivan Cheltsov}\newline
\textnormal{University of Edinburgh, Edinburgh, Scotland
\newline
\texttt{i.cheltsov@ed.ac.uk}}}

\address{\emph{Kento Fujita}
\newline
\textnormal{Osaka University, Osaka, Japan}
\newline
\textnormal{\texttt{fujita@math.sci.osaka-u.ac.jp}}}

\address{ \emph{Takashi Kishimoto}\newline \textnormal{Saitama University, Saitama, Japan
\newline
\texttt{kisimoto.takasi@gmail.com}}}

\address{ \emph{Jihun Park}\newline \textnormal{Institute for Basic Science, Pohang, Korea \newline POSTECH, Pohang, Korea \newline
\texttt{wlog@postech.ac.kr}}}

\maketitle

\tableofcontents

Throughout this paper, all varieties are assumed to be projective and defined over~$\mathbb{C}$.

\section{Introduction}
\label{section:intro}

Smooth Fano threefolds are classified into $105$  families labeled as \textnumero 1.1, \textnumero 1.2, \textnumero 1.3,   $\ldots$, \textnumero 10.1.
For the~description of these families, see \cite{IskovskikhProkhorov}.
It has been proved in~\cite{Book,Fujita2019,Fujita2021} that the~families
\begin{center}
\textnumero 2.23, \textnumero 2.26, \textnumero 2.28, \textnumero 2.30, \textnumero 2.31, \textnumero 2.33, \textnumero 2.35, \textnumero 2.36, \textnumero 3.14, \\
\textnumero 3.16, \textnumero 3.18, \textnumero 3.21, \textnumero 3.22, \textnumero 3.23, \textnumero 3.24, \textnumero 3.26, \textnumero 3.28, \textnumero 3.29, \\
\textnumero 3.30, \textnumero 3.31, \textnumero 4.5,  \textnumero 4.8,  \textnumero 4.9,  \textnumero 4.10, \textnumero 4.11, \textnumero 4.12, \textnumero 5.2
\end{center}
do not have smooth K-polystable members, and general members of other families are K-polystable.
For $56$ families, K-polystable smooth Fano 3-folds are described in \cite{AbbanZhuangSeshadri,Book,BelousovLoginov,CheltsovDenisovaFujita,CheltsovPark,Denisova,LTN,Liu,Malbon,XuLiu,CheltsovFujitaKishimotoOkada}.
The remaining $21$ deformation families are:
\begin{center}
\textnumero 1.9, \textnumero 1.10, \textnumero 2.5, \textnumero 2.9, \textnumero 2.10, \textnumero 2.11, \textnumero 2.12, \textnumero 2.13, \textnumero 2.14, \textnumero 2.16, \\
\textnumero 2.17, \textnumero 2.18, \textnumero 2.19, \textnumero 2.20, \textnumero 3.2, \textnumero 3.4, \textnumero 3.5,    \textnumero 3.6, \textnumero 3.7, \textnumero 3.8, \textnumero 3.11.
\end{center}
The families \textnumero 1.10, \textnumero 2.20, \textnumero 3.5, \textnumero 3.8 contain both K-polystable and non-K-polystable members,
and all smooth Fano threefolds in the~families
\begin{center}
\textnumero 1.9, \textnumero 2.5, \textnumero 2.9, \textnumero 2.10, \textnumero 2.11, \textnumero 2.12, \textnumero 2.13, \textnumero 2.14, \\
\textnumero 2.16, \textnumero 2.17, \textnumero 2.18, \textnumero 2.19, \textnumero 3.2, \textnumero 3.4, \textnumero 3.6, \textnumero 3.7, \textnumero 3.11
\end{center}
are conjectured to be K-stable \cite{Book}.
In this paper, we verify this conjecture for two families:

\begin{maintheorem*}
All smooth Fano 3-folds in the~families \textnumero 2.18 and \textnumero 3.4 are K-stable.
\end{maintheorem*}

Hence, to find all smooth K-polystable Fano 3-folds, one have to deal with $19$ families
\begin{center}
\textnumero 1.9, \textnumero 1.10, \textnumero 2.5, \textnumero 2.9, \textnumero 2.10, \textnumero 2.11, \textnumero 2.12, \textnumero 2.13, \textnumero 2.14, \\
\textnumero 2.16, \textnumero 2.17, \textnumero 2.19, \textnumero 2.20, \textnumero 3.2, \textnumero 3.5,    \textnumero 3.6, \textnumero 3.7, \textnumero 3.8, \textnumero 3.11.
\end{center}

To describe smooth Fano 3-folds in the~families \textnumero 2.18 and \textnumero 3.4,
let $V\to\mathbb{P}^1\times\mathbb{P}^2$ be a~double cover branched along a~smooth surface of degree $(2,2)$,
let $V\to\mathbb{P}^2$ be the~composition of this double cover and the~projection $\mathbb{P}^1\times\mathbb{P}^2\to\mathbb{P}^2$,
and let $X\to V$ be the~blow up of a~smooth fiber of this composition morphism.
Then we have the~following commutative diagram:
$$
\xymatrix{
X\ar@{->}[rr]\ar@{->}[d]&&V\ar@{->}[d]\\
\mathbb{P}^1\times\mathbb{F}_1\ar@{->}[rr]&&\mathbb{P}^1\times\mathbb{P}^2}
$$
where $\mathbb{F}_1$ is the~first Hirzebruch surface,
the~morphism $X\to\mathbb{P}^1\times\mathbb{F}_1$ is a~double cover ramified over the~proper transform on $\mathbb{P}^1\times\mathbb{F}_1$ of the~ramification surface of the~double cover $V\to\mathbb{P}^1\times\mathbb{P}^2$,
and~$\mathbb{P}^1\times\mathbb{F}_1\to\mathbb{P}^1\times\mathbb{P}^2$ is a birational morphism induced by the~blow up $\mathbb{F}_1\to\mathbb{P}^2$.
Then
\begin{itemize}
\item $V$ is a~smooth Fano 3-fold in the~deformation family \textnumero 2.18,
\item $X$ is a~smooth Fano 3-fold in the~deformation family \textnumero 3.4.
\end{itemize}
Furthermore, all smooth Fano 3-folds in these deformation families can be obtained in this way.

Let us say few words about the~proof of Main Theorem.
To prove that $V$ is K-stable, we recall from \cite{Dervan,Fujita2019b,LiuZhu,Zhuang} that
\begin{center}
the~Fano 3-fold $V$ is K-stable $\iff$ the~log Fano pair $\Big(\mathbb{P}^1\times\mathbb{P}^2,\frac{1}{2}R\Big)$ is K-stable,
\end{center}
where $R$ is the~ramification surface of the~double cover $V\to\mathbb{P}^1\times\mathbb{P}^2$.
In Section~\ref{section:2-18}, we prove that the log Fano pair $(\mathbb{P}^1\times\mathbb{P}^2,cR)$ is K-stable for every $c\in(0,1)\cap\mathbb{Q}$
using Abban--Zhuang~theory and the~technique developed in \cite{Book,Fujita2021}.
We refer the~reader to \cite[\S~1.7]{Book} and \cite[\S~4]{Fujita2021} for details.
Similarly, to prove that $X$ is K-stable, we prove that
the~log Fano pair $(\mathbb{P}^1\times\mathbb{F}_1,\frac{1}{2}R)$ is K-stable,
where now $R$ is the~ramification surface of the~double cover $X\to\mathbb{P}^1\times\mathbb{F}_1$.
The~proof is much more involved in this case,
because we have to resolve two deadlocks arising when $R$ is quite~special.
To~overcome~these difficulties, we apply Abban--Zhuang theory to exceptional surfaces of toric weighted blow ups of the~3-fold $\mathbb{P}^1\times\mathbb{F}_1$,
and use toric geometry to compute Zariski decompositions.
This is a new approach, which can resolve deadlocks in similar problems.

The structure of this paper is simple: we prove Main Theorem for the~family \textnumero 2.18 in Section~\ref{section:2-18},
and we prove Main Theorem for the~family \textnumero 3.4 in~Section~\ref{section:3-4}.
In Appendix~\ref{section:tables}, we put all the~tables necessary for the~Zariski decompositions discussed in Section~\ref{section:3-4}.

\smallskip

\textbf{Acknowledgements.}
The authors thank IBS Center for Geometry and Physics in Pohang,
Saitama University, and the~University of Edinburgh for their hospitality.

Ivan Cheltsov was supported by JSPS Invitational Fellowships for Research in Japan (S22040),
EPSRC Grant \textnumero\ EP/V054597/1, and Institut des Hautes Etudes Scientifiques (\mbox{Bures-sur-Yvette}).

Kento Fujita was supported by JSPS KAKENHI Grant Number 22K03269.

Takashi Kishimoto was supported by JSPS KAKENHI Grants Number 19K03395 and 23K03047.

Jihun Park was supported by  the~IBS project IBS-R003-D1, Institute for Basic Science in Korea.

\section{Smooth Fano 3-folds in the~family \textnumero 2.18}
\label{section:2-18}

Let $Y=\mathbb{P}^1\times\mathbb{P}^2$, let $R$ be a~smooth surface in $Y$ of degree $(2,2)$, and let $V\to Y$ be the~double cover branched over $R$.
Then $\mathrm{Aut}(V)$ is finite \cite{CheltsovShramovPrzyjalkowski}, so $V$ is K-stable if and only if $V$ is K-polystable.
On the~other hand, it follows from \cite{Dervan,Fujita2019b,LiuZhu,Zhuang} that
\begin{center}
$V$ is K-polystable $\iff$ $(Y,\frac{1}{2}R)$ is K-polystable.
\end{center}
Let $\Delta_Y=cR$ for $c\in[0,1)\cap\mathbb{Q}$. Then $(Y,\Delta_Y)$ is a log Fano pair for every $c\in[0,1)\cap\mathbb{Q}$.

\begin{theorem}
\label{theorem:2-18-K-stable}
The log Fano pair $(Y,\Delta_Y)$ is K-stable for every $c\in(0,1)\cap\mathbb{Q}$.
\end{theorem}

Let us prove Theorem~\ref{theorem:2-18-K-stable}. Set $L=-K_Y-\Delta_Y$. Then $L$ is a~divisor of degree $(2-2c,3-2c)$, so
$$
L^3=6(1-c)(3-2c)^2.
$$
Fix $c\in(0,1)\in\mathbb{Q}$. Let $P$ be a~point in $Y$. Recall that
$$
\delta_P(Y,\Delta_Y)=
\inf_{\substack{\mathbf{E}/Y\\ P\in C_Y(\mathbf{E})}}\frac{A_{Y,\Delta_Y}(\mathbf{E})}{S_{L}(\mathbf{E})},
$$
where the~infimum is taken over all prime divisors $\mathbf{E}$ over $Y$ whose centers on $Y$ contain~$P$, and
$$
S_{L}(\mathbf{E})=\frac{1}{L^3}\int\limits_0^\infty\mathrm{vol}\big(L-u\mathbf{E}\big)du.
$$
By \cite{Li,Fujita2019}, to prove that $(Y,\Delta_Y)$ is K-stable, it is enough to show that $\delta_P(Y,\Delta_Y)>1$.

\begin{lemma}
\label{lemma:2-18-easy-case}
Suppose that $P\not\in R$. Then $\delta_P(Y,\Delta_Y)>1$.
\end{lemma}

\begin{proof}
Let $S$ be the~surface in $Y$ of degree $(1,0)$ that contains $P$, let $R_S=R\vert_{S}$, and~let~\mbox{$\Delta_{S}=cR_S$}.
Take $u\in\mathbb{R}_{\geqslant 0}$. Then $L-uS$ is pseudoeffective $\iff$ $L-uS$ is nef $\iff$ $u\in[0,2-2c]$. This gives
$$
S_{L}(S)=\frac{1}{L^3}\int\limits_0^{2-2c}\big(L-uS\big)^3du=\frac{1}{L^3}\int\limits_0^{2-2c}3(3-2c)^2(2-2c-u)du=1-c<1.
$$
Note that $S\cong\mathbb{P}^2$.  Let $\ell$ be a~general line in $S$ that passes through $P$, and let $v$ be a~non-negative real number.
Then  $(L-uS)\vert_{S}-v\ell$ is a~divisor of degree $3-2c-v$. Thus, we have
\begin{center}
$(L-uS)\vert_{S}-v\ell$ is pseudoeffective $\iff$ $(L-uS)\vert_{S}-v\ell$ is nef $\iff$ $v\in[0,3-2c]$.
\end{center}
Now, following \cite{AbbanZhuang,Book,Fujita2021}, we set
$$
S_L\big(W^{S}_{\bullet,\bullet};\ell\big)=\frac{3}{L^3}\int\limits_0^{2-2c}\int\limits_0^{3-2c}\big((L-uS)\big\vert_{S}-v\ell\big)^2dvdu
$$
and
$$
S_L\big(W_{\bullet, \bullet,\bullet}^{S,\ell};P\big)=\frac{3}{L^3}\int\limits_0^{2-2c}\int\limits_0^{3-2c}\Big(\big(L-uS)\big\vert_{S}-v\ell\big)\cdot\ell\Big)^2dvdu.
$$
Integrating, we get $S_L(W^{S}_{\bullet,\bullet};\ell)=S_L(W_{\bullet, \bullet,\bullet}^{S,\ell};P)=\frac{3-2c}{3}$.
Thus, it follows from \cite{AbbanZhuang,Book,Fujita2021} that
$$
\delta_P(Y,\Delta_Y)\geqslant\min\left\{\frac{1-\mathrm{ord}_P\big(\Delta_{S}\big\vert_{\ell}\big)}{S_L(W_{\bullet, \bullet,\bullet}^{S,\ell}; P)}, \frac{1}{S_L(W_{\bullet,\bullet}^S;\ell)},\frac{1}{S_{L}(S)}\right\}=\frac{3}{3-2c}>1,
$$
since $\mathrm{ord}_P(\Delta_{S}\vert_{\ell})=0$, because $P\not\in R$ by assumption.
\end{proof}

Thus, to prove Theorem~\ref{theorem:2-18-K-stable}, we may assume that $P\in R$.

\begin{lemma}
\label{lemma:2-18-OK-case}
Let $\mathbf{f}$ be the~fiber of the~projection $Y\to\mathbb{P}^2$ such that $P\in \mathbf{f}$.
Suppose that $\mathbf{f}\not\subset R$. Then $\delta_P(Y,\Delta_Y)>1$.
\end{lemma}

\begin{proof}
Let $S$ be a~general surface in $Y$ of degree $(0,1)$ that contains $\mathbf{f}$, let $R_S=R\vert_{S}$, let $\Delta_{S}=cR_S$.
Then $S\cong\mathbb{P}^1\times\mathbb{P}^1$, and $R_S$ is a~smooth curve such that $R_S\sim 2\mathbf{s}+2\mathbf{f}$,
where $\mathbf{s}$ is the~smooth curve in the~surface $S$ such that $\mathbf{s}^2=0$, $\mathbf{s}\cdot\mathbf{f}=1$ and $P\in \mathbf{s}$.
Note that $L\vert_{S}\sim_{\mathbb{R}}(2-2c)\mathbf{s}+(3-2c)\mathbf{f}$.

Take $u\in\mathbb{R}_{\geqslant 0}$. Then $L-uS$ is pseudoeffective $\iff$ $L-uS$ is nef $\iff$ $u\in[0,3-2c]$, so
$$
S_{L}(S)=\frac{1}{L^3}\int\limits_{0}^{3-2c}\big(L-uS\big)^3du=\frac{1}{L^3}\int\limits_0^{3-2c}6(1-c)(3-2c-u)^2du=\frac{3-2c}{3}<1.
$$
Note that $(L-uS)\vert_{S}\sim_{\mathbb{R}}(2-2c)\mathbf{s}+(3-2c-u)\mathbf{f}$.

Now, let $\alpha\colon\widetilde{S}\to S$ be the~blow up of the~point $P$, let $\mathbf{e}$ be the~exceptional curve of the~blow~up~$\alpha$,
let $\widetilde{\mathbf{s}}$, $\widetilde{\mathbf{f}}$ and $R_{\widetilde{S}}$
be the~proper transforms on $\widetilde{S}$ of the~curves $\mathbf{s}$, $\mathbf{f}$ and $R_S$, respectively.
Set $\Delta_{\widetilde{S}}=cR_{\widetilde{S}}$.
Then $\widetilde{S}$ is the~smooth del Pezzo surface of degree~$7$,
$\widetilde{\mathbf{s}}\cap\widetilde{\mathbf{f}}=\varnothing$,
and $\widetilde{\mathbf{s}}$, $\widetilde{\mathbf{f}}$, $\mathbf{e}$ are $(-1)$-curves in~$\widetilde{S}$.
Let $v$ be a~non-negative real number. Then
$$
\alpha^*\big((L-uS)\big\vert_{S}\big)-v\mathbf{e}\sim_{\mathbb{R}}(2-2c)\widetilde{\mathbf{s}}+(3-2c-u)\widetilde{\mathbf{f}}+(5-4c-u-v)\mathbf{e},
$$
and it is pseudoeffective $\iff$ $v\leqslant 5-4c-u$.
For $v\in [0, 5-4c-u]$, we let $\widetilde{P}(u,v)$ be the~positive part of the~Zariski decomposition of $\alpha^*((L-uS)\big\vert_{S})-v\mathbf{e}$,
and we let $\widetilde{N}(u,v)$ be its negative part.
As in the~proof of Lemma~\ref{lemma:2-18-easy-case}, we set
$$
S_L\big(W^{S}_{\bullet,\bullet};\mathbf{e}\big)=\frac{3}{L^3}\int\limits_0^{3-2c}\int\limits_0^{5-4c-u}\big(\widetilde{P}(u,v)\big)^2dvdu.
$$
Similarly, for every point $O\in\mathbf{e}$, we set
$$
S\big(W_{\bullet, \bullet,\bullet}^{\widetilde{S},\mathbf{e}};O\big)=\frac{3}{L^3}
\int\limits_0^{3-2c}\int\limits_0^{5-4c-u}\big(\widetilde{P}(u,v)\cdot\mathbf{e}\big)^2dvdu+F_O\big(W_{\bullet, \bullet,\bullet}^{\widetilde{S},\mathbf{e}}\big),
$$
where
$$
F_O\big(W_{\bullet, \bullet,\bullet}^{\widetilde{S},\mathbf{e}}\big)=\frac{6}{L^3}\int\limits_0^{3-2c}\int\limits_0^{5-4c-u}\big(\widetilde{P}(u,v)\cdot\mathbf{e}\big)\cdot \mathrm{ord}_O\big(\widetilde{N}(u,v)\big|_\mathbf{e}\big)dvdu.
$$
Then it follows from \cite{AbbanZhuang,Book,Fujita2021} that
\begin{equation}
\label{equation:2-18-OK-case}
\delta_P(Y,\Delta_Y)\geqslant\min\left\{
\min_{O\in\mathbf{e}}\frac{1-\mathrm{ord}_O\big(\Delta_{\widetilde{S}}\big\vert_{\mathbf{e}}\big)}{S_L(W_{\bullet, \bullet,\bullet}^{\widetilde{S},\mathbf{e}};O)},
\frac{A_{S,\Delta_S}(\mathbf{e})}{S_L(W_{\bullet,\bullet}^S;\mathbf{e})},\frac{1}{S_{L}(S)}\right\},
\end{equation}
where $A_{S,\Delta_S}(\mathbf{e})=2-c$. On the~other hand, if $0\leqslant u\leqslant 1$, then
$$
\widetilde{P}(u,v)\sim_{\mathbb{R}}\left\{\aligned
&(2-2c)\widetilde{\mathbf{s}}+(3-2c-u)\widetilde{\mathbf{f}}+(5-4c-u-v)\mathbf{e}\ \text{ if } 0\leqslant v\leqslant 2-2c, \\
&(2-2c)\widetilde{\mathbf{s}}+(5-4c-u-v)\big(\widetilde{\mathbf{f}}+\mathbf{e}\big)\ \text{ if } 2-2c\leqslant v\leqslant 3-2c-u,\\
&(5-4c-u-v)\big(\widetilde{\mathbf{s}}+\widetilde{\mathbf{f}}+\mathbf{e}\big)\ \text{ if } 3-2c-u\leqslant v\leqslant 5-4c-u,
\endaligned
\right.
$$
and
$$
\widetilde{N}(u,v)=\left\{\aligned
&0\ \text{ if } 0\leqslant v\leqslant 2-2c, \\
&(v+2c-2)\widetilde{\mathbf{f}}\ \text{ if } 2-2c\leqslant v\leqslant 3-2c-u,\\
&(v+2c-2)\widetilde{\mathbf{f}}+(v+u-3+2c)\widetilde{\mathbf{s}}\ \text{ if } 3-2c-u\leqslant v\leqslant 5-4c-u,
\endaligned
\right.
$$
which gives
$$
\big(\widetilde{P}(u,v)\big)^2=\left\{\aligned
&8c^2+4cu-v^2-20c-4u+12\ \text{ if } 0\leqslant v\leqslant 2-2c, \\
&4(1-c)(4-3c-u-v)\ \text{ if } 2-2c\leqslant v\leqslant 3-2c-u,\\
&(5-4c-u-v)^2\ \text{ if } 3-2c-u\leqslant v\leqslant 5-4c-u,
\endaligned
\right.
$$
and
$$
\widetilde{P}(u,v)\cdot\mathbf{e}=\left\{\aligned
&v\ \text{ if } 0\leqslant v\leqslant 2-2c, \\
&2-2c\ \text{ if } 2-2c\leqslant v\leqslant 3-2c-u,\\
&5-4c-u-v\ \text{ if } 3-2c-u\leqslant v\leqslant 5-4c-u.
\endaligned
\right.
$$
Similarly, if $1\leqslant u\leqslant 3-2c$, then
$$
\widetilde{P}(u,v)\sim_{\mathbb{R}}\left\{\aligned
&(2-2c)\widetilde{\mathbf{s}}+(3-2c-u)\widetilde{\mathbf{f}}+(5-4c-u-v)\mathbf{e}\ \text{ if } 0\leqslant v\leqslant 3-2c-u, \\
&(5-4c-u-v)\big(\widetilde{\mathbf{s}}+\mathbf{e}\big)+(3-2c-u)\widetilde{\mathbf{f}}\ \text{ if } 3-2c-u\leqslant v\leqslant 2-2c,\\
&(5-4c-u-v)\big(\widetilde{\mathbf{s}}+\widetilde{\mathbf{f}}+\mathbf{e}\big)\ \text{ if } 2-2c\leqslant v\leqslant 5-4c-u,
\endaligned
\right.
$$
and
$$
\widetilde{N}(u,v)=\left\{\aligned
&0\ \text{ if } 0\leqslant v\leqslant 3-2c-u, \\
&(v+u-3+2c)\widetilde{\mathbf{s}}\ \text{ if } 3-2c-u\leqslant v\leqslant 2-2c,\\
&(v+2c-2)\widetilde{\mathbf{f}}+(v+u-3+2c)\widetilde{\mathbf{s}}\ \text{ if } 2-2c\leqslant v\leqslant 5-4c-u,
\endaligned
\right.
$$
which gives
$$
\big(\widetilde{P}(u,v)\big)^2=\left\{\aligned
&8c^2+4cu-v^2-20c-4u+12\ \text{ if } 0\leqslant v\leqslant 3-2c-u, \\
&(3-2c-u)(7-6c-u-2v)\ \text{ if } 3-2c-u\leqslant v\leqslant 2-2c,\\
&(5-4c-u-v)^2\ \text{ if } 2-2c\leqslant v\leqslant 5-4c-u,
\endaligned
\right.
$$
and
$$
\widetilde{P}(u,v)\cdot\mathbf{e}=\left\{\aligned
&v\ \text{ if } 0\leqslant v\leqslant 3-2c-u, \\
&3-2c-u\ \text{ if } 3-2c-u\leqslant v\leqslant 2-2c,\\
&5-4c-u-v\ \text{ if } 2-2c\leqslant v\leqslant 5-4c-u.
\endaligned
\right.
$$
Thus, integrating, we get $S_L(W_{\bullet,\bullet}^S;\mathbf{e})=\frac{6-5c}{3}$ and
$$
S_L\big(W_{\bullet, \bullet,\bullet}^{\widetilde{S},\mathbf{e}};O\big)=
\left\{\aligned
&1-c-\frac{2(5-3c)(1-c)^2}{3(3-2c)^2}\ \text{ if } O\not\in\widetilde{\mathbf{f}}\cup\widetilde{\mathbf{s}}, \\
&1-c\ \text{ if } O\in\widetilde{\mathbf{s}},\\
&\frac{3-2c}{3}\ \text{ if } O\in\widetilde{\mathbf{f}}.
\endaligned
\right.
$$
Therefore, if $\widetilde{\mathbf{s}}\cap R_{\widetilde{S}}\cap\mathbf{e}=\varnothing$ and $\widetilde{\mathbf{f}}\cap R_{\widetilde{S}}\cap\mathbf{e}=\varnothing$, then \eqref{equation:2-18-OK-case} gives $\delta_P(Y,\Delta_Y)>1$.

Thus, to complete the~proof, we may assume that either
$\widetilde{\mathbf{s}}\cap R_{\widetilde{S}}\cap\mathbf{e}\ne\varnothing$ or $\widetilde{\mathbf{f}}\cap R_{\widetilde{S}}\cap\mathbf{e}\ne\varnothing$.
Then exactly one of the~following two (mutually excluding) cases holds:
\begin{itemize}
\item[($\heartsuit$)] the~curve $\widetilde{\mathbf{s}}$ contains the~point $R_{\widetilde{S}}\cap\mathbf{e}$, i.e. the~curves $\mathbf{s}$ and $R_S$ are tangent at $P$,
\item[($\diamondsuit$)] the~curve $\widetilde{\mathbf{f}}$ contains the~point $R_{\widetilde{S}}\cap\mathbf{e}$, i.e. the~curves $\mathbf{f}$ and $R_S$ are tangent at $P$.
\end{itemize}
In both cases, we consider the~following commutative diagram:
$$
\xymatrix{
\widetilde{S}\ar@{->}[d]_\alpha&&\overline{S}\ar@{->}[ll]_\beta\ar@{->}[d]^\gamma\\
S&&\widehat{S}\ar@{->}[ll]^\rho}
$$
where $\beta$ is the~blow up of the~intersection point $R_{\widetilde{S}}\cap\mathbf{e}$,
the~map $\gamma$ is the~contraction of the~proper transform of the~curve $\mathbf{e}$ to an~ordinary double point of the~surface $\widehat{S}$,
and $\rho$ is the~contraction of the~proper transform of the~$\beta$-exceptional curve.
Then $\widehat{S}$ is a~singular del Pezzo surface of degree~$6$,
and $\rho$ is a~weighted blow up of the~point $P$ with weights $(1,2)$.

Let $\widehat{\mathbf{f}}$, $\widehat{\mathbf{s}}$ and $R_{\widehat{S}}$
be the~proper transforms on the~surface $\widehat{S}$ of the~curves $\mathbf{f}$, $\mathbf{s}$ and $R_S$, respectively,
and let $\mathbf{z}$ be the~$\rho$-exceptional curve.
In the~case ($\heartsuit$), we have
$$
\rho^*\big((L-uS)\vert_{S}\big)-v\mathbf{z}\sim_{\mathbb{R}}(2-2c)\widehat{\mathbf{s}}+(3-2c-u)\widehat{\mathbf{f}}+(7-6c-u-v)\mathbf{z},
$$
and the~intersections of the~curves $\mathbf{z}$, $\widehat{\mathbf{f}}$ and $\widehat{\mathbf{s}}$ are given in the~following table:
\begin{center}
\renewcommand\arraystretch{1.4}
\begin{tabular}{|c||c|c|c|}
\hline & $\mathbf{z}$ & $\widehat{\mathbf{f}}$ & $\widehat{\mathbf{s}}$ \\
\hline
\hline
$\mathbf{z}$            & $-\frac{1}{2}$ & $\frac{1}{2}$& $1$\\
\hline
$\widehat{\mathbf{f}}$  & $\frac{1}{2}$ & $-\frac{1}{2}$ & $0$\\
\hline
$\widehat{\mathbf{s}}$  & $1$ & $0$ & $-2$\\
\hline
\end{tabular}
\end{center}
Similarly, in the~case ($\diamondsuit$), we have
$$
\rho^*\big((L-uS)\vert_{S}\big)-v\mathbf{z}\sim_{\mathbb{R}}(2-2c)\widehat{\mathbf{s}}+(3-2c-u)\widehat{\mathbf{f}}+(8-6c-2u-v)\mathbf{z},
$$
and the~intersections of the~curves $\mathbf{z}$, $\widehat{\mathbf{f}}$ and $\widehat{\mathbf{s}}$ are given in the~following table:
\begin{center}
\renewcommand\arraystretch{1.4}
\begin{tabular}{|c||c|c|c|}
\hline & $\mathbf{z}$ & $\widehat{\mathbf{f}}$ & $\widehat{\mathbf{s}}$ \\
\hline
\hline
$\mathbf{z}$            & $-\frac{1}{2}$ & $1$& $\frac{1}{2}$\\
\hline
$\widehat{\mathbf{f}}$  & $1$ & $-2$ & $0$\\
\hline
$\widehat{\mathbf{s}}$  & $\frac{1}{2}$ & $0$ & $-\frac{1}{2}$\\
\hline
\end{tabular}
\end{center}
In both cases, let $\widehat{t}(u)$ be the~largest $v\in\mathbb R_{\geqslant 0}$ such that $\rho^*\big(P(u)\big\vert_{S}\big)-v\mathbf{z}$ is pseudoeffective. Then
$$
\widehat{t}(u)=\left\{\aligned
&7-6c-u\ \text{ in the~case ($\heartsuit$)}, \\
&8-6c-2u\ \text{ in the~case ($\diamondsuit$)}.
\endaligned
\right.
$$
For each $v\in [0,\widehat{t}(u)]$, let $\widehat{P}(u,v)$ be the~positive part of the~Zariski decomposition of this divisor,
and let $\widehat{N}(u,v)$ be its negative part.
Set
$$
S_L\big(W^{S}_{\bullet,\bullet};\mathbf{z}\big)=\frac{3}{L^3}\int\limits_0^{3-2c}\int\limits_0^{\widehat{t}(u)}\big(\widehat{P}(u,v)\big)^2dvdu.
$$
Similarly, for every point $O\in\mathbf{z}$, we set
$$
S\big(W_{\bullet, \bullet,\bullet}^{\widehat{S},\mathbf{z}};O\big)=\frac{3}{L^3}\int\limits_0^{3-2c}\int\limits_0^{\widehat{t}(u)}\big(\widehat{P}(u,v)\cdot\mathbf{z}\big)^2dvdu+F_O\big(W_{\bullet, \bullet,\bullet}^{\widehat{S},\mathbf{z}}\big),
$$
where
$$
F_O\big(W_{\bullet,\bullet,\bullet}^{\widehat{S},\mathbf{z}}\big)=\frac{6}{L^3}\int\limits_0^{3-2c}\int\limits_0^{\widehat{t}(u)}\big(\widehat{P}(u,v)\cdot\mathbf{z}\big)\cdot \mathrm{ord}_O\big(\widehat{N}(u,v)\big|_\mathbf{z}\big)dvdu.
$$
Let $Q$ be the~singular point of the~surface $\widehat{S}$. Then $Q\not\in R_{\widehat{S}}$, since
$$
Q=\left\{\aligned
&\widehat{\mathbf{f}}\cap\mathbf{z}\ \text{ in the~case ($\heartsuit$)}, \\
&\widehat{\mathbf{s}}\cap\mathbf{z}\ \text{ in the~case ($\diamondsuit$)}.
\endaligned
\right.
$$
Let $\Delta_{\widehat{S}}=cR_{\widehat{S}}$ and $\Delta_{\mathbf{z}}=\frac{1}{2}Q+\Delta_{\widehat{S}}\vert_{\mathbf{z}}$.
Then it follows from \cite{AbbanZhuang,Book,Fujita2021} that
\begin{equation}
\label{equation:2-18-OK-case-final}
\delta_P(Y,\Delta_Y)\geqslant\min\left\{
\min_{O\in\mathbf{z}}\frac{A_{\mathbf{z},\Delta_{\mathbf{z}}}(O)}{S_L(W_{\bullet, \bullet,\bullet}^{\widehat{S},\mathbf{z}};O)},
\frac{A_{S,\Delta_S}(\mathbf{z})}{S_L(W_{\bullet,\bullet}^S;\mathbf{z})},\frac{A_{Y,\Delta_Y}(S)}{S_{L}(S)}\right\},
\end{equation}
where $A_{Y,\Delta_Y}(S)=1$, $A_{S,\Delta_S}(\mathbf{z})=3-2c$ and $A_{\mathbf{z},\Delta_{\mathbf{z}}}(O)=1-\mathrm{ord}_O(\Delta_{\mathbf{z}})$ for every point $O\in\mathbf{z}$.

Let us compute $S_L(W^{S}_{\bullet,\bullet};\mathbf{z})$ and $S_L(W_{\bullet, \bullet,\bullet}^{\widehat{S},\mathbf{z}};O)$ for every point $O\in\mathbf{z}$.

First, we deal with the~case ($\heartsuit$).
In this case, if $c\leqslant\frac{1}{2}$ or if $c>\frac{1}{2}$ and $2c-1\leqslant u\leqslant 3-2c$, then
$$
\widehat{P}(u,v)\sim_{\mathbb{R}}\left\{\aligned
&(2-2c)\widehat{\mathbf{s}}+(3-2c-u)\widehat{\mathbf{f}}+(7-6c-u-v)\mathbf{z}\ \text{ if } 0\leqslant v\leqslant 3-2c-u, \\
&\frac{7-6c-u-v}{2}\big(\widehat{\mathbf{s}}+2\mathbf{z}\big)+(3-2c-u)\widehat{\mathbf{f}} \ \text{ if } 3-2c-u\leqslant v\leqslant 4-4c,\\
&\frac{7-6c-u-v}{2}\big(\widehat{\mathbf{s}}+2\mathbf{z}+2\widehat{\mathbf{f}}\big)\ \text{ if } 4-4c\leqslant v\leqslant 7-6c-u,
\endaligned
\right.
$$
and
$$
\widehat{N}(u,v)=\left\{\aligned
&0\ \text{ if } 0\leqslant v\leqslant 3-2c-u, \\
&\frac{v+u+2c-3}{2}\widehat{\mathbf{s}}\ \text{ if } 3-2c-u\leqslant v\leqslant 4-4c,\\
&\frac{v+u+2c-3}{2}\widehat{\mathbf{s}}+(v-4+4c)\widehat{\mathbf{f}}\ \text{ if } 4-4c\leqslant v\leqslant 7-6c-u,
\endaligned
\right.
$$
which gives
$$
\big(\widehat{P}(u,v)\big)^2=\left\{\aligned
&8c^2+4cu+12-20c-4u-\frac{v^2}{2}\ \text{ if } 0\leqslant v\leqslant 3-2c-u, \\
&\frac{(3-2c-u)(11-10c-u-2v)}{2}\ \text{ if } 3-2c-u\leqslant v\leqslant 4-4c,\\
&\frac{(7-6c-u-v)^2}{2}\ \text{ if } 4-4c\leqslant v\leqslant 7-6c-u,
\endaligned
\right.
$$
and
$$
\widehat{P}(u,v)\cdot\mathbf{z}=\left\{\aligned
&\frac{v}{2}\ \text{ if } 0\leqslant v\leqslant 3-2c-u, \\
&\frac{3-u-2c}{2}\ \text{ if } 3-2c-u\leqslant v\leqslant 4-4c,\\
&\frac{7-6c-u-v}{2}\ \text{ if } 4-4c\leqslant v\leqslant 7-6c-u.
\endaligned
\right.
$$
Similarly, in the~case ($\heartsuit$), if $c>\frac{1}{2}$ and $0\leqslant u\leqslant 2c-1$, then
$$
\widehat{P}(u,v)\sim_{\mathbb{R}}\left\{\aligned
&(2-2c)\widehat{\mathbf{s}}+(3-2c-u)\widehat{\mathbf{f}}+(7-6c-u-v)\mathbf{z}\ \text{ if } 0\leqslant v\leqslant 4-4c, \\
&(2-2c)\widehat{\mathbf{s}}+(7-6c-u-v)\big(\widehat{\mathbf{f}}+\mathbf{z}\big) \ \text{ if } 4-4c\leqslant v\leqslant 3-2c-u,\\
&\frac{7-6c-u-v}{2}\big(\widehat{\mathbf{s}}+2\mathbf{z}+2\widehat{\mathbf{f}}\big)\ \text{ if } 3-2c-u\leqslant v\leqslant 7-6c-u,
\endaligned
\right.
$$
and
$$
\widehat{N}(u,v)=\left\{\aligned
&0\ \text{ if } 0\leqslant v\leqslant 4-4c, \\
&(v-4+4c)\widehat{\mathbf{f}}\ \text{ if } 4-4c\leqslant v\leqslant 3-2c-u,\\
&\frac{v+u+2c-3}{2}\widehat{\mathbf{s}}+(v-4+4c)\widehat{\mathbf{f}}\ \text{ if } 3-2c-u\leqslant v\leqslant 7-6c-u,
\endaligned
\right.
$$
which gives
$$
\big(\widehat{P}(u,v)\big)^2=\left\{\aligned
&8c^2+4cu+12-20c-4u-\frac{v^2}{2}\ \text{ if } 0\leqslant v\leqslant 4-4c, \\
&4(1-c)(5-4c-u-v)\ \text{ if } 4-4c\leqslant v\leqslant 3-2c-u,\\
&\frac{(7-6c-u-v)^2}{2}\ \text{ if } 3-2c-u\leqslant v\leqslant 7-6c-u,
\endaligned
\right.
$$
and
$$
\widehat{P}(u,v)\cdot\mathbf{z}=\left\{\aligned
&\frac{v}{2}\ \text{ if } 0\leqslant v\leqslant 4-4c, \\
&2-2c\ \text{ if } 4-4c\leqslant v\leqslant 3-2c-u,\\
&\frac{7-6c-u-v}{2}\ \text{ if } 3-2c-u\leqslant v\leqslant 7-6c-u.
\endaligned
\right.
$$

Now, integrating, we get $S_L(W_{\bullet,\bullet}^S;\mathbf{z})=3-\frac{8}{3}c<3-2c=A_{S,\Delta_S}(\mathbf{z})$ and
$$
S_L\big(W_{\bullet, \bullet,\bullet}^{\widehat{S},\mathbf{z}};O\big)=
\left\{\aligned
&1-c-\frac{68c^2-124c+57}{96(1-c)}\ \text{ if } O\not\in\widehat{\mathbf{f}}\cup\widehat{\mathbf{s}}\ \text{and}\ c\leqslant\frac{1}{2}, \\
&1-c-\frac{8(2-c)(1-c)^2}{3(3-2c)^2}\ \text{ if } O\not\in\widehat{\mathbf{f}}\cup\widehat{\mathbf{s}}\ \text{and}\ c>\frac{1}{2}, \\
&1-c\ \text{ if } O\in\widehat{\mathbf{s}},\\
&\frac{1}{2}-\frac{c}{3}\ \text{ if } O\in\widehat{\mathbf{f}}.
\endaligned
\right.
$$
Hence, using \eqref{equation:2-18-OK-case-final}, we obtain $\delta_P(Y,\Delta_Y)>1$.

Now, we deal with the~case ($\diamondsuit$). If $0\leqslant u\leqslant 2-c$,~then
$$
\widehat{P}(u,v)\sim_{\mathbb{R}}\left\{\aligned
&(2-2c)\widehat{\mathbf{s}}+(3-2c-u)\widehat{\mathbf{f}}+(8-6c-2u-v)\mathbf{z}\ \text{ if } 0\leqslant v\leqslant 2-2c, \\
&(2-2c)\widehat{\mathbf{s}}+\frac{8-6c-2u-v}{2}\big(\widehat{\mathbf{f}}+2\mathbf{z}\big)\ \text{ if } 2-2c\leqslant v\leqslant 6-4c-2u,\\
&\frac{8-6c-2u-v}{2}\big(2\widehat{\mathbf{s}}+\widehat{\mathbf{f}}+2\mathbf{z}\big)\ \text{ if } 6-4c-2u\leqslant v\leqslant 8-6c-2u,
\endaligned
\right.
$$
and
$$
\widehat{N}(u,v)=\left\{\aligned
&0\ \text{ if } 0\leqslant v\leqslant 2-2c, \\
&\frac{v-2+2c}{2}\widehat{\mathbf{f}}\ \text{ if } 2-2c\leqslant v\leqslant 6-4c-2u,\\
&\frac{v-2+2c}{2}\widehat{\mathbf{f}}+(v+2u-6+4c)\widehat{\mathbf{s}}\ \text{ if } 6-4c-2u\leqslant v\leqslant 8-6c-2u,
\endaligned
\right.
$$
which gives
$$
\big(\widehat{P}(u,v)\big)^2=\left\{\aligned
&8c^2+4cu+12-20c-4u-\frac{v^2}{2}\ \text{ if } 0\leqslant v\leqslant 2-2c, \\
&2(1-c)(7-5c-2u-v)\ \text{ if } 2-2c\leqslant v\leqslant 6-4c-2u,\\
&\frac{(8-6c-2u-v)^2}{2}\ \text{ if } 6-4c-2u\leqslant v\leqslant 8-6c-2u,
\endaligned
\right.
$$
and
$$
\widehat{P}(u,v)\cdot\mathbf{z}=\left\{\aligned
&\frac{v}{2}\ \text{ if } 0\leqslant v\leqslant 2-2c, \\
&1-c\ \text{ if } 2-2c\leqslant v\leqslant 6-4c-2u,\\
&\frac{8-6c-2u-v}{2}\ \text{ if } 6-4c-2u\leqslant v\leqslant 8-6c-2u.
\endaligned
\right.
$$
Similarly, if $2-c\leqslant u\leqslant 3-2c$, then
$$
\widehat{P}(u,v)\sim_{\mathbb{R}}\left\{\aligned
&(2-2c)\widehat{\mathbf{s}}+(3-2c-u)\widehat{\mathbf{f}}+(8-6c-2u-v)\mathbf{z}\ \text{ if } 0\leqslant v\leqslant 6-4c-2u, \\
&(8-6c-2u-v)\big(\widehat{\mathbf{s}}+\mathbf{z}\big)+(3-2c-u)\widehat{\mathbf{f}}\ \text{ if } 6-4c-2u\leqslant v\leqslant 2-2c,\\
&\frac{8-6c-2u-v}{2}\big(2\widehat{\mathbf{s}}+\widehat{\mathbf{f}}+2\mathbf{z}\big)\ \text{ if } 2-2c\leqslant v\leqslant 8-6c-2u,
\endaligned
\right.
$$
and
$$
\widehat{N}(u,v)=\left\{\aligned
&0\ \text{ if } 0\leqslant v\leqslant 6-4c-2u, \\
&(v+2u-6+4c)\widehat{\mathbf{s}}\ \text{ if } 6-4c-2u\leqslant v\leqslant 2-2c,\\
&\frac{v-2+2c}{2}\widehat{\mathbf{f}}+(v+2u-6+4c)\widehat{\mathbf{s}}\ \text{ if } 2-2c\leqslant v\leqslant 8-6c-2u,
\endaligned
\right.
$$
which gives
$$
\big(\widehat{P}(u,v)\big)^2=\left\{\aligned
&8c^2+4cu+12-20c-4u-\frac{v^2}{2}\ \text{ if } 0\leqslant v\leqslant 6-4c-2u, \\
&2(3-2c-u)(5-4c-u-v)\ \text{ if } 6-4c-2u\leqslant v\leqslant 2-2c,\\
&\frac{(8-6c-2u-v)^2}{2}\ \text{ if } 2-2c\leqslant v\leqslant 8-6c-2u,
\endaligned
\right.
$$
and
$$
\widehat{P}(u,v)\cdot\mathbf{z}=\left\{\aligned
&\frac{v}{2}\ \text{ if } 0\leqslant v\leqslant 6-4c-2u, \\
&3-2c-u\ \text{ if } 6-4c-2u\leqslant v\leqslant 2-2c,\\
&\frac{8-6c-2u-v}{2}\ \text{ if } 2-2c\leqslant v\leqslant 8-6c-2u.
\endaligned
\right.
$$
Now, integrating, we get $S_L(W_{\bullet,\bullet}^S;\mathbf{z})=3-\frac{7}{3}c<3-2c=A_{S,\Delta_S}(\mathbf{z})$ and
$$
S_L\big(W_{\bullet, \bullet,\bullet}^{\widehat{S},\mathbf{z}};O\big)=
\left\{\aligned
&1-c-\frac{(1-c)(31c^2-90c+65)}{12(3-2c)^2}\ \text{ if } O\not\in\widehat{\mathbf{f}}\cup\widehat{\mathbf{s}}, \\
&\frac{1}{2}-\frac{c}{2}\ \text{ if } O\in\widehat{\mathbf{s}},\\
&1-\frac{2c}{3}\ \text{ if } O\in\widehat{\mathbf{f}}.
\endaligned
\right.
$$
Hence, using \eqref{equation:2-18-OK-case-final}, we get $\delta_P(Y,\Delta_Y)>1$.
This completes the~proof of the~lemma.
\end{proof}

Finally, we prove

\begin{lemma}
\label{lemma:2-18-blow-up}
Let $\mathbf{f}$ be the~fiber of the~projection $Y\to\mathbb{P}^2$ such that $P\in\mathbf{f}$.
Suppose that $\mathbf{f}\subset R$. Then $\delta_P(Y,\Delta_Y)>1$.
\end{lemma}

\begin{proof}
Let $\nu\colon\mathscr{Y}\to Y$ be the~blow up of the~smooth curve $\mathbf{f}$, let $E$ be the~$\nu$-exceptional surface.
Take $u\in\mathbb{R}_{\geqslant 0}$. Then $\nu^*(L)-uE$ is pseudoeffective $\iff$ $\nu^*(L)-uE$ is nef $\iff$ $u\leqslant 3-2c$, so
$$
S_L(E)=\frac{1}{L^3}\int\limits_0^{3-2c}\big(\nu^*(L)-uE\big)^3du=\frac{1}{6(1-c)(3-2c)^2}\int\limits_0^{3-2c}6(1-c)(3-2c-u)(3-2c+u)du=2-\frac{4}{3}c,
$$
which gives
$$
\delta_P(Y,\Delta_Y)\leqslant\frac{A_{Y,\Delta_Y}(E)}{S_L(E)}=1+\frac{c}{2(3-2c)}.
$$

Now, let $R_{\mathscr{Y}}$ be the~proper transform on $\mathscr{Y}$ of the~surface $R$,
let $R_E=R_{\mathscr{Y}}\vert_{E}$, let $\Delta_{E}=cR_E$,
and let $\mathbf{l}$ be the~fiber of the~projection $E\to\mathbf{f}$ such that $\nu(\mathbf{l})=P$.
Then $R_E$ is a~smooth curve, which implies that $(E,\Delta_{E})$ has Kawamata log terminal singularities.
For every point $\mathscr{P}\in\mathbf{l}$, set
$$
\delta_{\mathscr{P}}\big(E,\Delta_E;W^E_{\bullet,\bullet}\big)=
\inf_{\substack{F/E,\\ \mathscr{P}\in C_E(F)}}\frac{A_{E,\Delta_E}(F)}{S(W^E_{\bullet,\bullet};F)},
$$
where the~infimum is taken over all prime divisors $F$ over $E$ whose centers on $E$ contain~$\mathscr{P}$, and
$$
S\big(W^E_{\bullet,\bullet};F\big)=\frac{3}{L^3}\int\limits_{0}^{3-2c}\int\limits_0^\infty \mathrm{vol}\Big(\big(\nu^*(L)-uE\big)\big\vert_{E}-vF\Big)dvdu.
$$
Then it follows from \cite{AbbanZhuang,Book,Fujita2021} that
$$
\delta_P(Y,\Delta_Y)\geqslant\min\Bigg\{\inf_{\mathscr{P}\in\mathbf{l}}\delta_{\mathscr{P}}\big(E,\Delta_E;W^E_{\bullet,\bullet}\big),\frac{A_{Y,\Delta_Y}(E)}{S_L(E)}\Bigg\}=\min\Bigg\{\inf_{\mathscr{P}\in\mathbf{l}}\delta_{\mathscr{P}}\big(E,\Delta_E;W^E_{\bullet,\bullet}\big),1+\frac{c}{2(3-2c)}\Bigg\}.
$$
Thus, to complete the~proof, it is enough to show that $\delta_{\mathscr{P}}(E,\Delta_E;W^E_{\bullet,\bullet})>1$ for every point $\mathscr{P}\in\mathbf{l}$.

Fix a point $\mathscr{P}\in\mathbf{l}$.
Let $\mathbf{s}$ be the~smooth curve in $E\cong\mathbb{P}^1\times\mathbb{P}^1$ such that $\mathbf{s}^2=0$, $\mathbf{s}\cdot\mathbf{l}=1$,~$\mathscr{P}\in\mathbf{s}$.
Then $E\vert_{E}\sim-\mathbf{s}$, $R_E\sim 2\mathbf{l}+\mathbf{s}$, and $(\nu^*(L)-uE)\big\vert_{E}\sim_{\mathbb{R}}(2-2c)\mathbf{l}+u\mathbf{s}$.

Let $\alpha\colon\widetilde{E}\to E$ be the~blow up of the~point $\mathscr{P}$, let $\mathbf{e}$ be the~exceptional curve of the~blow~up~$\alpha$,
and let $\widetilde{\mathbf{s}}$, $\widetilde{\mathbf{l}}$, $R_{\widetilde{E}}$
be the~proper transforms on $\widetilde{E}$ of the~curves $\mathbf{s}$, $\mathbf{l}$, $R_E$, respectively.
Set $\Delta_{\widetilde{E}}=cR_{\widetilde{E}}$.
Then $\widetilde{E}$ is a~smooth del Pezzo surface of degree~$7$,
$\widetilde{\mathbf{s}}\cap\widetilde{\mathbf{l}}=\varnothing$,
and $\widetilde{\mathbf{s}}$, $\widetilde{\mathbf{l}}$, $\mathbf{e}$ are all $(-1)$-curves in $\widetilde{E}$.
Let $v$ be a~non-negative real number. Then
$$
\alpha^*\Big(\big(\nu^*(L)-uE\big)\big\vert_{E}\Big)-v\mathbf{e}\sim_{\mathbb{R}} (2-2c)\widetilde{\mathbf{l}}+u\widetilde{\mathbf{s}}+(2-2c+u-v)\mathbf{e},
$$
and it is pseudoeffective $\iff$ $v\leqslant 2-2c+u$.
For $v\in [0,2-2c+u]$, let $\widetilde{P}(u,v)$ be the~positive part of the~Zariski decomposition of this divisor,
and let $\widetilde{N}(u,v)$ be its negative part. Set
$$
S_L\big(W^{E}_{\bullet,\bullet};\mathbf{e}\big)=\frac{3}{L^3}\int\limits_0^{3-2c}\int\limits_0^{2-2c+u}\big(\widetilde{P}(u,v)\big)^2dvdu.
$$
Likewise, for every point $O\in\mathbf{e}$, we set
$$
S\big(W_{\bullet, \bullet,\bullet}^{\widetilde{E},\mathbf{e}};O\big)=\frac{3}{L^3}
\int\limits_0^{3-2c}\int\limits_0^{2-2c+u}\big(\widetilde{P}(u,v)\cdot\mathbf{e}\big)^2dvdu+F_O\big(W_{\bullet, \bullet,\bullet}^{\widetilde{E},\mathbf{e}}\big),
$$
where
$$
F_O\big(W_{\bullet, \bullet,\bullet}^{\widetilde{E},\mathbf{e}}\big)=\frac{6}{L^3}\int\limits_0^{3-2c}\int\limits_0^{2-2c+u}\big(\widetilde{P}(u,v)\cdot\mathbf{e}\big)\cdot \mathrm{ord}_O\big(\widetilde{N}(u,v)\big|_\mathbf{e}\big)dvdu.
$$
Then it follows from \cite{AbbanZhuang,Book,Fujita2021} that
\begin{equation}
\label{equation:2-18-final}
\delta_{\mathscr{P}}\big(E,\Delta_E;W^E_{\bullet,\bullet}\big)\geqslant\min\left\{
\min_{O\in\mathbf{e}}\frac{1-\mathrm{ord}_O\big(\Delta_{\widetilde{E}}\big\vert_{\mathbf{e}}\big)}{S_L(W_{\bullet, \bullet,\bullet}^{\widetilde{E},\mathbf{e}};O)},
\frac{A_{E,\Delta_E}(\mathbf{e})}{S_L(W_{\bullet,\bullet}^E;\mathbf{e})}\right\},
\end{equation}
where
$$
A_{E,\Delta_E}\big(\mathbf{e}\big)=
\left\{\aligned
&2-c\ \text{ if } \mathscr{P}\in R_E, \\
&2\ \text{ if } \mathscr{P}\not\in R_E.
\endaligned
\right.
$$
On the~other hand, if $0\leqslant u\leqslant 2-2c$, then
$$
\widetilde{P}(u,v)\sim_{\mathbb{R}}\left\{\aligned
&(2-2c)\widetilde{\mathbf{l}}+u\widetilde{\mathbf{s}}+(2-2c+u-v)\mathbf{e}\ \text{ if } 0\leqslant v\leqslant u, \\
&(2-2c+u-v)\big(\mathbf{e}+\widetilde{\mathbf{l}}\big)+u\widetilde{\mathbf{s}}\ \text{ if } u\leqslant v\leqslant 2-2c,\\
&(2-2c+u-v)\big(\mathbf{e}+\widetilde{\mathbf{l}}+\widetilde{\mathbf{s}}\big)\ \text{ if } 2-2c\leqslant v\leqslant 2-2c+u,
\endaligned
\right.
$$
and
$$
\widetilde{N}(u,v)=\left\{\aligned
&0\ \text{ if } 0\leqslant v\leqslant u, \\
&(v-u)\widetilde{\mathbf{l}}\ \text{ if } u\leqslant v\leqslant 2-2c,\\
&(v-u)\widetilde{\mathbf{l}}+(v-2+2c)\widetilde{\mathbf{s}}\ \text{ if } 2-2c\leqslant v\leqslant 2-2c+u,
\endaligned
\right.
$$
which gives
$$
\big(\widetilde{P}(u,v)\big)^2=\left\{\aligned
&(4-4c)u-v^2\ \text{ if } 0\leqslant v\leqslant u, \\
&u(4-4c+u-2v)\ \text{ if } u\leqslant v\leqslant 2-2c,\\
&(2-2c+u-v)^2\ \text{ if } 2-2c\leqslant v\leqslant 2-2c+u,
\endaligned
\right.
$$
and
$$
\widetilde{P}(u,v)\cdot\mathbf{e}=\left\{\aligned
&v\ \text{ if } 0\leqslant v\leqslant u, \\
&u\ \text{ if } u\leqslant v\leqslant 2-2c,\\
&2-2c+u-v\ \text{ if } 2-2c\leqslant v\leqslant 2-2c+u.
\endaligned
\right.
$$
Similarly, if $2-2c\leqslant u\leqslant 3-2c$, then
$$
\widetilde{P}(u,v)\sim_{\mathbb{R}}\left\{\aligned
&(2-2c)\widetilde{\mathbf{l}}+u\widetilde{\mathbf{s}}+(2-2c+u-v)\mathbf{e}\ \text{ if } 0\leqslant v\leqslant 2-2c, \\
&(2-2c)\widetilde{\mathbf{l}}+(2-2c+u-v)\big(\mathbf{e}+\widetilde{\mathbf{s}}\big)\ \text{ if } 2-2c\leqslant v\leqslant u,\\
&(2-2c+u-v)\big(\mathbf{e}+\widetilde{\mathbf{l}}+\widetilde{\mathbf{s}}\big)\ \text{ if } u\leqslant v\leqslant 2-2c+u,
\endaligned
\right.
$$
and
$$
\widetilde{N}(u,v)=\left\{\aligned
&0\ \text{ if } 0\leqslant v\leqslant 2-2c, \\
&(v-2+2c)\widetilde{\mathbf{s}}\ \text{ if } 2-2c\leqslant v\leqslant u,\\
&(v-u)\widetilde{\mathbf{l}}+(v-2+2c)\widetilde{\mathbf{s}}\ \text{ if } u\leqslant v\leqslant 2-2c+u,
\endaligned
\right.
$$
which gives
$$
\big(\widetilde{P}(u,v)\big)^2=\left\{\aligned
&(4-4c)u-v^2\ \text{ if } 0\leqslant v\leqslant 2-2c, \\
&4(1-c)(1-c+u-v)\ \text{ if } 2-2c\leqslant v\leqslant u,\\
&(2-2c+u-v)^2\ \text{ if } u\leqslant v\leqslant 2-2c+u,
\endaligned
\right.
$$
and
$$
\widetilde{P}(u,v)\cdot\mathbf{e}=\left\{\aligned
&v\ \text{ if } 0\leqslant v\leqslant 2-2c, \\
&2-2c\ \text{ if } 2-2c\leqslant v\leqslant u,\\
&2-2c+u-v\ \text{ if } u\leqslant v\leqslant 2-2c+u.
\endaligned
\right.
$$
Thus, integrating, we get $S_L(W_{\bullet,\bullet}^E;\mathbf{e})=2-\frac{5}{3}c<2-c$ and
$$
S_L\big(W_{\bullet, \bullet,\bullet}^{\widetilde{E},\mathbf{e}};O\big)=
\left\{\aligned
&1-c-\frac{2(5-3c)(1-c)^2}{3(3-2c)^2}\ \text{ if } O\not\in\widetilde{\mathbf{l}}\cup\widetilde{\mathbf{s}}, \\
&1-\frac{2}{3}c\ \text{ if } O\in\widetilde{\mathbf{s}},\\
&1-c\ \text{ if } O\in\widetilde{\mathbf{l}}.
\endaligned
\right.
$$
Therefore, if $\mathscr{P}\not\in R_E$,
then \eqref{equation:2-18-final} gives $\delta_{\mathscr{P}}(E,\Delta_E;W^E_{\bullet,\bullet})>1$.
Similarly, if $\mathscr{P}\in R_E$, then $\widetilde{\mathbf{l}}\cap R_{\widetilde{E}}=\varnothing$,
the~set $\widetilde{\mathbf{s}}\cap R_{\widetilde{E}}\cap\mathbf{e}$ consists of at most $1$ point,
and \eqref{equation:2-18-final} gives $\delta_{\mathscr{P}}(E,\Delta_E;W^E_{\bullet,\bullet})>1$
if $\widetilde{\mathbf{s}}\cap R_{\widetilde{E}}\cap\mathbf{e}=\varnothing$.

To complete the~proof, we may assume that the~intersection $\widetilde{\mathbf{s}}\cap R_{\widetilde{E}}\cap\mathbf{e}$ consists of one~point,
which means that the~curves $\mathbf{s}$ and $R_E$ are tangent at the~point $P$.
As in the~proof of Lemma~\ref{lemma:2-18-OK-case}, let us consider the~following commutative diagram:
$$
\xymatrix{
\widetilde{E}\ar@{->}[d]_\alpha&&\overline{E}\ar@{->}[ll]_\beta\ar@{->}[d]^\gamma\\
E&&\widehat{E}\ar@{->}[ll]^\rho}
$$
where $\beta$ is the~blow up of the~point $\widetilde{\mathbf{s}}\cap R_{\widetilde{E}}\cap\mathbf{e}$,
the~morphism $\gamma$ is the~contraction of the~proper transform of the~curve $\mathbf{e}$ to an~ordinary double point of the~surface $\widehat{E}$,
and $\rho$ is the~contraction of the~proper transform of the~$\beta$-exceptional curve.

Let $\widehat{\mathbf{s}}$, $\widehat{\mathbf{l}}$, $R_{\widehat{E}}$
be the~proper transforms on $\widehat{E}$ of the~curves $\mathbf{s}$, $\mathbf{l}$, $R_E$, respectively.
Then $R_{\widehat{E}}\cap \widehat{\mathbf{s}}=\varnothing$,
and the~curves $\widehat{\mathbf{s}}$, $\widehat{\mathbf{l}}$, $R_{\widehat{E}}$ are smooth.
Let $\mathbf{z}$ be the~$\rho$-exceptional curve.
Then $R_{\widehat{E}}\cap \widehat{\mathbf{l}}\cap\mathbf{z}=\varnothing$, $\mathbf{z}\cong\mathbb{P}^1$,
and the~intersections of the~curves $\mathbf{z}$, $\widehat{\mathbf{s}}$ and $\widehat{\mathbf{l}}$ are given in the~following table:
\begin{center}
\renewcommand\arraystretch{1.4}
\begin{tabular}{|c||c|c|c|}
\hline & $\mathbf{z}$ & $\widehat{\mathbf{s}}$ & $\widehat{\mathbf{l}}$ \\
\hline
\hline
$\mathbf{z}$            & $-\frac{1}{2}$ & $1$& $\frac{1}{2}$\\
\hline
$\widehat{\mathbf{s}}$  & $1$ & $-2$ & $0$\\
\hline
$\widehat{\mathbf{l}}$  & $\frac{1}{2}$ & $0$ & $-\frac{1}{2}$\\
\hline
\end{tabular}
\end{center}
Furthermore, we have
$$
\rho^*\Big(\big(\nu^*(L)-uE\big)\big\vert_{E}\Big)-v\mathbf{z}\sim_{\mathbb{R}}(2-2c)\widehat{\mathbf{l}}+u\widehat{\mathbf{s}}+(2-2c+2u-v)\mathbf{z},
$$
and it is pseudoeffective $\iff$ $v\leqslant 2-2c+2u$.
For $v\in [0,2-2c+2u]$, let $\widehat{P}(u,v)$ be the~positive part of the~Zariski decomposition of this divisor,
and let $\widehat{N}(u,v)$ be its negative part.~Set
$$
S_L\big(W^{E}_{\bullet,\bullet};\mathbf{z}\big)=\frac{3}{L^3}\int\limits_0^{3-2c}\int\limits_0^{2-2c+2u}\big(\widehat{P}(u,v)\big)^2dvdu.
$$
Similarly, for every point $O\in\mathbf{z}$, we set
$$
S\big(W_{\bullet, \bullet,\bullet}^{\widehat{E},\mathbf{z}};O\big)=\frac{3}{L^3}\int\limits_0^{3-2c}\int\limits_0^{2-2c+2u}\big(\widehat{P}(u,v)\cdot\mathbf{z}\big)^2dvdu+F_O\big(W_{\bullet, \bullet,\bullet}^{\widehat{E},\mathbf{z}}\big),
$$
where
$$
F_O\big(W_{\bullet,\bullet,\bullet}^{\widehat{E},\mathbf{z}}\big)=\frac{6}{L^3}\int\limits_0^{3-2c}\int\limits_0^{2-2c+2u}\big(\widehat{P}(u,v)\cdot\mathbf{z}\big)\cdot \mathrm{ord}_O\big(\widehat{N}(u,v)\big|_\mathbf{z}\big)dvdu.
$$
Let $Q$ be the~singular point of the~surface $\widehat{E}$. Then $Q=\widehat{\mathbf{l}}\cap\mathbf{z}$.
Let $\Delta_{\widehat{E}}=cR_{\widehat{E}}$ and  $\Delta_{\mathbf{z}}=\frac{1}{2}Q+\Delta_{\widehat{E}}\vert_{\mathbf{z}}$.
Then it follows from \cite{AbbanZhuang,Book,Fujita2021} that
\begin{equation}
\label{equation:2-18-final-final}
\delta_{\mathscr{P}}\big(E,\Delta_E;W^E_{\bullet,\bullet}\big)\geqslant\min\left\{
\min_{O\in\mathbf{z}}\frac{A_{\mathbf{z},\Delta_{\mathbf{z}}}(O)}{S_L(W_{\bullet, \bullet,\bullet}^{\widehat{E},\mathbf{z}};O)},
\frac{A_{E,\Delta_E}(\mathbf{z})}{S_L(W_{\bullet,\bullet}^E;\mathbf{z})}\right\},
\end{equation}
where $A_{E,\Delta_E}(\mathbf{z})=3-2c$ and $A_{\mathbf{z},\Delta_{\mathbf{z}}}(O)=1-\mathrm{ord}_O(\Delta_{\mathbf{z}})$.
On the~other hand, if $0\leqslant u\leqslant 1-c$, then
$$
\widehat{P}(u,v)\sim_{\mathbb{R}}\left\{\aligned
&(2-2c)\widehat{\mathbf{l}}+u\widehat{\mathbf{s}}+(2-2c+2u-v)\mathbf{z}\ \text{ if } 0\leqslant v\leqslant 2u, \\
&(2-2c+2u-v)\big(\mathbf{z}+\widehat{\mathbf{l}}\big)+u\widehat{\mathbf{s}}\ \text{ if } 2u\leqslant v\leqslant 2-2c,\\
&(2-2c+2u-v)\big(\mathbf{z}+\widehat{\mathbf{l}}+\widehat{\mathbf{s}}\big)\ \text{ if } 2-2c\leqslant v\leqslant 2-2c+2u,
\endaligned
\right.
$$
and
$$
\widehat{N}(u,v)=\left\{\aligned
&0\ \text{ if } 0\leqslant v\leqslant 2u, \\
&(v-2u)\widehat{\mathbf{l}}\ \text{ if } 2u\leqslant v\leqslant 2-2c,\\
&(v-2u)\widehat{\mathbf{l}}+\frac{v-2+2c}{2}\widehat{\mathbf{s}}\ \text{ if } 2-2c\leqslant v\leqslant 2-2c+2u,
\endaligned
\right.
$$
which gives
$$
\big(\widehat{P}(u,v)\big)^2=\left\{\aligned
&(4-4c)u-\frac{v^2}{2}\ \text{ if } 0\leqslant v\leqslant 2u, \\
&2u(2-2c+u-v)\ \text{ if } 2u\leqslant v\leqslant 2-2c,\\
&\frac{(2-2c+2u-v)^2}{2}\ \text{ if } 2-2c\leqslant v\leqslant 2-2c+2u,
\endaligned
\right.
$$
and
$$
\widehat{P}(u,v)\cdot\mathbf{z}=\left\{\aligned
&\frac{v}{2}\ \text{ if } 0\leqslant v\leqslant 2u, \\
&u\ \text{ if } 2u\leqslant v\leqslant 2-2c,\\
&\frac{2-2c+2u-v}{2}\ \text{ if } 2-2c\leqslant v\leqslant 1+u.
\endaligned
\right.
$$
Similarly, if $1-c\leqslant u\leqslant 3-2c$, then
$$
\widehat{P}(u,v)\sim_{\mathbb{R}}\left\{\aligned
&(2-2c)\widehat{\mathbf{l}}+u\widehat{\mathbf{s}}+(2-2c+2u-v)\mathbf{z}\ \text{ if } 0\leqslant v\leqslant 2-2c, \\
&\frac{2-2c+2u-v}{2}\big(2\mathbf{z}+\widehat{\mathbf{s}}\big)+(2-2c)\widehat{\mathbf{l}}\ \text{ if } 2-2c\leqslant v\leqslant 2u,\\
&\frac{2-2c+2u-v}{2}\big(2\mathbf{z}+2\widehat{\mathbf{l}}+\widehat{\mathbf{s}}\big)\ \text{ if } 2u\leqslant v\leqslant 2-2c+2u,
\endaligned
\right.
$$
and
$$
\widehat{N}(u,v)=\left\{\aligned
&0\ \text{ if } 0\leqslant v\leqslant 2-2c, \\
&\frac{v-2+2c}{2}\widehat{\mathbf{s}}\ \text{ if } 2-2c\leqslant v\leqslant 2u,\\
&(v-2u)\widehat{\mathbf{l}}+\frac{v-2+2c}{2}\widehat{\mathbf{s}}\ \text{ if } 2u\leqslant v\leqslant 2-2c+2u,
\endaligned
\right.
$$
which gives
$$
\big(\widehat{P}(u,v)\big)^2=\left\{\aligned
&(4-4c)u-\frac{v^2}{2}\ \text{ if } 0\leqslant v\leqslant 2-2c, \\
&2(1-c)(1-c+2u-v)\ \text{ if } 2-2c\leqslant v\leqslant 2u,\\
&\frac{(2-2c+2u-v)^2}{2}\ \text{ if } 2u\leqslant v\leqslant 1+2u,
\endaligned
\right.
$$
and
$$
\widehat{P}(u,v)\cdot\mathbf{z}=\left\{\aligned
&\frac{v}{2}\ \text{ if } 0\leqslant v\leqslant 2-2c, \\
&1-c\ \text{ if } 2-2c\leqslant v\leqslant 2u,\\
&\frac{2-2c+2u-v}{2}\ \text{ if } 2u\leqslant v\leqslant 1+u.
\endaligned
\right.
$$
Now, integrating, we get $S_L(W_{\bullet,\bullet}^E;\mathbf{z})=3-\frac{7}{3}c<3-2c=A_{E,\Delta_E}(\mathbf{z})$ and
$$
S_L\big(W_{\bullet, \bullet,\bullet}^{\widehat{E},\mathbf{z}};O\big)=
\left\{\aligned
&1-c-\frac{(1-c)(31c^2-90c+65)}{12(3-2c)^2}\ \text{ if } O\not\in\widehat{\mathbf{l}}\cup\widehat{\mathbf{s}}, \\
&\frac{1}{2}-\frac{c}{2}\ \text{ if } O\in\widehat{\mathbf{l}},\\
&1-\frac{2c}{3}\ \text{ if } O\in\widehat{\mathbf{s}}.
\endaligned
\right.
$$
Hence, using \eqref{equation:2-18-final-final}, we get $\delta_{\mathscr{P}}(E,\Delta_E;W^E_{\bullet,\bullet})>1$,
which gives $\delta_P(Y,\Delta_Y)>1$.
\end{proof}

Now, combining Lemmas~\ref{lemma:2-18-easy-case}, \ref{lemma:2-18-OK-case} and \ref{lemma:2-18-blow-up}, we obtain Theorem~\ref{theorem:2-18-K-stable}.

\section{Smooth Fano 3-folds in the~family \textnumero 3.4}
\label{section:3-4}

Let $Y=\mathbb{P}^1\times\mathbb{F}_1$.
Identify $Y=(\mathbb{A}^2\setminus 0)^3/\mathbb{G}_m^3$ for the~$\mathbb{G}_m^3$-action
$$
\big((x_0,x_1),(y_0,y_1),(z_0,z_1)\big)\mapsto
\Bigg(\big(\lambda x_0,\lambda x_1\big),\Big(\big(\mu y_0,\mu y_1\big),\frac{\nu z_0}{\mu},\nu z_1\Big),\Bigg)
$$
where $(\lambda,\mu,\nu)\in\mathbb{G}_m^3$, and $((x_0,x_1),(y_0,y_1),(z_0,z_1))$ are coordinates on $(\mathbb{A}^2)^3$.
We will use
\begin{itemize}
\item $([x_0:x_1],[y_0:y_1;z_0:z_1])$ as coordinates on $\mathbb{P}^1\times\mathbb{F}_1$,
\item $[x_0:x_1]$ as coordinates on the~first factor of~$Y=\mathbb{P}^1\times\mathbb{F}_1$,
\item $[y_0:y_1;z_0:z_1]$ as coordinates on the~second factor of $Y=\mathbb{P}^1\times\mathbb{F}_1$,
\item $[y_0:y_1]$ as coordinates on the~base of the~natural projection~$\mathbb{F}_1\to\mathbb{P}^1$.
\end{itemize}

To distinguish the~first factor of~$Y=\mathbb{P}^1\times\mathbb{F}_1$
and the~base of the~natural projection~$\mathbb{F}_1\to\mathbb{P}^1$,
we will use notations $\mathbb{P}^1_{x_0,x_1}$ and $\mathbb{P}^1_{y_0,y_1}$ for them, respectively.
Then $Y=\mathbb{P}^1_{x_0,x_1}\times\mathbb{F}_1$, and we have the~following commutative diagram:
$$
\xymatrix{
&Y\ar@{->}[d]^\phi\ar@{->}[dl]_{\pi_1}\ar@{->}[drr]^{\psi}\ar@{->}[rr]^{\pi_2}&&\mathbb{F}_1\ar@{->}[d]\\
\mathbb{P}^1_{x_0,x_1}&\mathbb{P}^1_{x_0,x_1}\times\mathbb{P}^1_{y_0,y_1}\ar@{->}[rr]\ar@{->}[l]&&\mathbb{P}^1_{y_0,y_1}}
$$
where $\pi_1$ and $\pi_2$ are projections to the~first and the~second factors, respectively,
$\phi$ is the~$\mathbb{P}^1$-bundle
$$
\big([x_0:x_1],[y_0:y_1;z_0:z_1]\big)\mapsto([x_0:x_1],[y_0:y_1]),
$$
the~morphism $\psi$ is the~$\mathbb{P}^1\times\mathbb{P}^1$-bundle  $([x_0:x_1],[y_0:y_1;z_0:z_1])\mapsto [y_0:y_1]$,
and all other morphisms are natural projections.
Let $F$ be a~fiber of the~morphism $\pi_1$, let $S$ be a~fiber of the~morphism $\psi$,
let $E$ be the~exceptional surface of the~birational contraction
$Y\to\mathbb{P}^1_{x_0,x_1}\times\mathbb{P}^2$ given by
$$
\big([x_0:x_1],[y_0:y_1;z_0:z_1]\big)\mapsto\big([x_0:x_1];[y_0z_0:y_1z_0:z_1]\big),
$$
let~$R$ be a smooth surface in $|2F+2E+2S|$,
and let $\eta\colon X\to \mathbb{P}^1_{x_0,x_1}\times\mathbb{F}_1$ be a double cover~ramified in the~surface $R$.
Then $X$ is a smooth Fano threefold in the~family \textnumero 3.4.

Recall that $X$ is K-stable $\iff$ $X$ is K-polystable, because $\mathrm{Aut}(X)$ is finite \cite{CheltsovShramovPrzyjalkowski}.
Let $\Delta_Y=\frac{1}{2}R$. Then it follows from \cite{Dervan,Fujita2019b,LiuZhu,Zhuang} that
\begin{center}
$X$ is K-polystable $\iff$ $(Y,\Delta_Y)$ is K-polystable.
\end{center}
The goal of this section is to prove the~following result.

\begin{theorem}
\label{theorem:3-4-K-stable}
The log Fano pair $(Y,\Delta_Y)$ is K-stable.
\end{theorem}

Before proving Theorem~\ref{theorem:3-4-K-stable}, observe that $E=\{z_0=0\}\subset Y$, and $R$ is given in $Y$ by
\begin{multline}
\label{equation:R}
x_0^2\Big(\big(a_0y_0^2+b_0y_0y_1+c_0y_1^2\big)z_0^2+\big(d_0y_0+e_0y_1\big)z_0z_1+f_0z_1^2\Big)+\\
+x_0x_1\Big(\big(a_1y_0^2+b_1y_0y_1+c_1y_1^2\big)z_0^2+\big(d_1y_0+e_1y_1\big)z_0z_1+f_1z_1^2\Big)+\\
+x_1^2\Big(\big(a_2y_0^2+b_2y_0y_1+c_2y_1^2\big)z_0^2+\big(d_2y_0+e_2y_1\big)z_0z_1+f_2z_1^2\Big)=0,
\end{multline}
where $a_0$, $b_0$, $c_0$, $d_0$, $e_0$, $f_0$, $a_1$, $b_1$, $c_1$, $d_1$, $e_1$, $f_1$,
$a_2$, $b_2$, $c_2$, $d_2$, $e_2$, $f_2$ are some numbers.

\begin{lemma}
\label{lemma:F-S}
Set $R_E=R\vert_{E}$, $R_S=R\vert_{S}$, $R_F=R\vert_{F}$.
Then
\begin{itemize}
\item[($\mathrm{i}$)] $R_E$ is a disjoint union of two fibers of the~projection $\pi_1\vert_{E}\colon E\to\mathbb{P}^1_{x_0,x_1}$,
\item[($\mathrm{ii}$)] the~curve $R_S$ is reduced,
\item[($\mathrm{iii}$)] if $R_F$ is reduced, then it has one or two ordinary double points,
\item[($\mathrm{iv}$)] if $R_F$ is not reduced, then $\mathrm{Sing}(R_F)=F\cap E$.
\end{itemize}
Let $P$ be a point in $F\cap S$ such that $P\not\in E$ and $P\in R$,
let $Z$ be the~fiber of $\phi$ that contains $P$, and let $C$ be the~fiber of $\pi_2$ that contains $P$.
Then
\begin{itemize}
\item[($\mathrm{v}$)] if $Z\subset R$, then $R_F$ and $R_S$ are singular at some points in $Z$,
\item[($\mathrm{vi}$)] if $C\subset R$, then $R_S$ is singular at some point in $C$.
\item[($\mathrm{vii}$)] at least one of the~surfaces $R_F$ and $R_S$ is smooth at $P$,
\item[($\mathrm{viii}$)] if $R_S$ is singular at $P$, and $Z\not\subset R$, then $R_F$ is smooth.
\end{itemize}
\end{lemma}

\begin{proof}
First, let us choose appropriate coordinates on $Y$ such that $F=\{x_1=0\}$ and $S=\{y_1=0\}$.
To prove $(\mathrm{i})$, observe that
$$
R_E=\big\{z_0=0, f_0x_0^2+f_1x_0x_1+f_2x_1^2=0\big\}\subset Y.
$$
Moreover, if $f_0x_0^2+f_1x_0x_1+f_2x_1^2$ is a square, then $R$ is singular. This proves $(\mathrm{i})$.

Let us prove $(\mathrm{ii})$. Using \eqref{equation:R}, we see that $R_S=\{f=0\}\subset S$ for
$$
f=x_0^2(a_0z_0^2+d_0z_0z_1+f_0z_1^2)+x_0x_1(a_1z_0^2+d_1z_0z_1+f_1z_1^2)+x_1^2(a_2z_0^2+d_2z_0z_1+f_2z_1^2),
$$
where we consider $([x_0:x_1],[z_0:z_1])$ as coordinates on $S\cong\mathbb{P}^1\times\mathbb{P}^1$.
Hence, if $R_S$ is not reduced, then~$f=gh^2$ for a non-constant polynomial $h$ and a polynomial $g$.
Then we can rewrite \eqref{equation:R} as
$$
y_1\Big(x_0^2\big((b_0y_0+c_0y_1)z_0^2+e_0z_0z_1\big)+x_0x_1\big((b_1y_0+c_1y_1)z_0^2+e_1z_0z_1\big)+x_1^2\big((b_2y_0+c_2y_1)z_0^2+e_2z_0z_1\big)\Big)+gh^2=0,
$$
which implies that the~surface $R$ is singular at every point of the~non-empty subset
$$
\big\{y_1=0,x_0^2\big(b_0y_0z_0^2+e_0z_0z_1\big)+x_0x_1\big(b_1y_0z_0^2+e_1z_0z_1\big)+x_1^2\big(b_2y_0z_0^2+e_2z_0z_1\big)=0,h=0\big\}\subset Y,
$$
which is impossible by assumption. Hence, we see that $R_S$ is reduced.  This proves $(\mathrm{ii})$.

Let us prove $(\mathrm{iii})$ and $(\mathrm{iv})$. Identify $F=\mathbb{F}_{1}$ with coordinates $[y_0:y_1;z_0:z_1]$.
Then
$$
R_F=\big\{(a_0y_0^2+b_0y_0y_1+c_0y_1^2)z_0^2+(d_0y_0+e_0y_1)z_0z_1+f_0z_1^2=0\big\}\subset F.
$$
Let $\upsilon\colon\mathbb{F}_1\to\mathbb{P}^2$ be the~blow up $[y_0:y_1;z_0:z_1]\mapsto[y_0z_0:y_1z_0:z_1]$,
and let $\mathbf{e}$ be its exceptional curve.
Then $\upsilon(R_F)$ is a reduced conic.
Furthermore, if $f_0\ne 0$, then $R_F\cap\mathbf{e}=\varnothing$,
and either $R_F$~is~smooth, or the~curve $R_F$ is a union of two smooth irreducible curves intersecting transversally at one point.
Thus, we may assume that $f_0=0$.
Then  $\upsilon(R_F)$ contains $\upsilon(\mathbf{e})$,
and $R_F=\mathbf{e}+R_F^\prime$, where
$$
R_F^\prime=\big\{(a_0y_0^2+b_0y_0y_1+c_0y_1^2)z_0+(d_0y_0+e_0y_1)z_1=0\big\}\subset F.
$$
If $d_0\ne 0$ or $e_0\ne 0$, then $R_F^\prime$ is the~proper transform of the~conic $\upsilon(R_F)$,
which is smooth at $\upsilon(\mathbf{e})$.
In this case, if $\upsilon(R_F)$ is irreducible, then the~curve $R_F^\prime$ is smooth,
and $R_F$ has one ordinary double point --- the~intersection point $\mathbf{e}\cap R_F^\prime$.
Similarly, if $\upsilon(R_F)$ is reducible, then $R_F$ has two ordinary double points ---
the~intersection point $\mathbf{e}\cap R_F^\prime$, and the~unique singular point of the~curve $R_F^\prime$.
Finally, if $d_0=0$ and $e_0=0$, then $R_F=2\mathbf{e}+\mathbf{l}+\mathbf{l}^\prime$,
where $\mathbf{l}+\mathbf{l}^\prime=\{a_0y_0^2+b_0y_0y_1+c_0y_1^2=0\}\subset F$,
so that $\mathbf{l}$ and $\mathbf{l}^\prime$ are distinct fibers of the~projection $\mathbb{F}_1\to\mathbb{P}^1_{y_0,y_1}$.
This proves $(\mathrm{iii})$ and $(\mathrm{iv})$.

Now, choosing appropriate coordinates on $Y$, we may assume that $P=([1:0],[1:0;1:0])$.
Then $a_0=0$, since $P\in R$.
Note also that $Z=\{x_1=0,y_1=0\}$ and $C=\{y_1=0,z_1=0\}$.

Both assertions ($\mathrm{v}$) and ($\mathrm{vi}$) are obvious.
Now, let us prove $(\mathrm{vii})$.
In the~affine chart $x_0y_0z_0\ne 0$, the~surface $R$ is given by
$$
a_1x+b_0y+d_0z+\text{higher order terms}=0,
$$
where $x=\frac{x_1}{x_0}$, $y=\frac{y_1}{y_0}$, $z=\frac{z_1}{z_0}$.
which implies that $(a_1,b_0,d_0)\ne (0,0,0)$, because $R$ is smooth at~$P$.
If $R_F$ is singular at $P$, then $b_0=0$ and $d_0=0$.
If $R_S$ is singular at $P$, then $a_1=0$ and $d_0=0$.
Hence, if both $R_F$ and $R_S$ are singular at $P$, then $(a_1,b_0,d_0)=(0,0,0)$.
This proves $(\mathrm{vii})$.

Let's prove ($\mathrm{viii}$).
Suppose that $R_S$ is singular at $P$, and $Z\not\subset R$. Then $a_1=d_0=0$ and $b_0f_0\ne 0$.
Observe that $R_F\cap \mathbf{e}=\varnothing$, since $f_0\ne 0$.
Now, computing the~defining equation of the~conic $\upsilon(R_F)$,
we see that this conic is smooth, because $b_0f_0\ne 0$. Then $R_F$ is also smooth.
This proves ($\mathrm{viii}$).
\end{proof}

\subsection{The proof}
\label{subsection:3-4-step-1}

Set $L=-(K_Y+\Delta_Y)$. Then $L\sim_{\mathbb{Q}}F+E+2S$ and $L^3=9$.
To prove~Theorem~\ref{theorem:3-4-K-stable},
we must show that
$\beta_{Y,\Delta_{Y}}(\mathbf{E})=A_{Y,\Delta_Y}(\mathbf{E})-S_{L}(\mathbf{E})>0$  for every prime divisor $\mathbf{E}$ over $Y$, where
$$
S_{L}(\mathbf{E})=\frac{1}{L^3}\int\limits_0^\infty\mathrm{vol}\big(L-u\mathbf{E}\big)du.
$$
Fix a prime divisor $\mathbf{F}$ over $Y$. Let us show that $\beta_{Y,\Delta_Y}(\mathbf{F})>0$.
Set $\mathfrak{C}=C_Y(\mathbf{F})$. Then
\begin{enumerate}
\item either $\mathfrak{C}$ is a~point,
\item or $\mathfrak{C}$ is an~irreducible curve,
\item or~$\mathfrak{C}$ is an~irreducible surface.
\end{enumerate}
In each case, let $P$ be some point in $\mathfrak{C}$. If $\beta_{Y,\Delta_Y}(\mathbf{F})\leqslant 0$, then $\delta_P(Y,\Delta_Y)\leqslant 1$, where
$$
\delta_P(Y,\Delta_Y)=
\inf_{\substack{\mathbf{E}/Y\\ P\in C_Y(\mathbf{E})}}\frac{A_{Y,\Delta_Y}(\mathbf{E})}{S_{L}(\mathbf{E})},
$$
where the~infimum is taken over all prime divisors $\mathbf{E}$ over $Y$ whose centers on $Y$ contain~$P$.

Changing coordinates on $Y$, we may assume that $P=([1:0],[1:0;a:b])$ for some $[a:b]\in\mathbb{P}^1$ such that $ab=0$.
Thus, we have the~following two possibilities:
\begin{itemize}
\item[($\clubsuit$)] $P=([1:0],[1:0;0:1])\in E$,
\item[($\spadesuit$)] $P=([1:0],[1:0;1:0])\not\in E$.
\end{itemize}
Moreover, we can choose $S$ to be the~fiber of the~morphism $\psi\colon Y\to\mathbb{P}^1_{y_0,y_1}$ that contains the~point~$P$,
and we can choose $F$ to be the~fiber of the~morphism $\pi_1\colon Y\to\mathbb{P}^1_{x_0,x_1}$ that contains $P$. Then
\begin{align*}
E&=\{z_0=0\}\cong\mathbb{P}^1\times\mathbb{P}^1,\\
S&=\{y_1=0\}\cong\mathbb{P}^1\times\mathbb{P}^1,\\
F&=\{x_1=0\}\cong\mathbb{F}_1.
\end{align*}

\begin{lemma}
\label{lemma:3-4-surfaces}
Suppose that $\mathfrak{C}$ is a~surface. Then $\beta_{Y,\Delta_{Y}}(\mathbf{F})>0$.
\end{lemma}

\begin{proof}
Since $\mathfrak{C}\sim n_FF+n_EE+n_SS$ for some non-negative integers $n_F$, $n_E$, $n_S$ that are~not all~zero, we have
$$
\beta_{Y,\Delta_{Y}}(\mathbf{F})=\beta_{Y,\Delta_{Y}}(\mathfrak{C})\geqslant\min\big\{\beta_{Y,\Delta_{Y}}(F), \beta_{Y,\Delta_{Y}}(E), \beta_{Y,\Delta_{Y}}(S)\big\},
$$
but $\beta_{Y,\Delta_{Y}}(F)=\frac{1}{2}$, $\beta_{Y,\Delta_{Y}}(E)=\frac{4}{9}$, $\beta_{Y,\Delta_{Y}}(S)=\frac{2}{9}$.
Indeed, let us compute $\beta_{Y,\Delta_{Y}}(E)$. Take $u\in\mathbb{R}_{\geqslant 0}$.
Then $L-uE$ is pseudoeffective $\iff$ $L-uE$ is nef $\iff$ $u\in[0,1]$.
Using this, we compute
$$
\beta_{Y,\Delta_{Y}}(E)=1-S_{L}(E)=1-\frac{1}{L^3}\int\limits_{0}^{1}(L-uE)^3du=1-\frac{1}{9}\int\limits_{0}^{1}\big(9-6u-3u^2\big)du=\frac{4}{9}.
$$
Similarly, we compute $\beta_{Y,\Delta_{Y}}(F)=\frac{1}{2}$ and $\beta_{Y,\Delta_{Y}}(S)=\frac{2}{9}$.
\end{proof}

Let $R_E=R\vert_{E}$ and $\Delta_{E}=\frac{1}{2}R_E$.
Then, by Lemma~\ref{lemma:F-S}, the~curve $R_E$ is a union of two distinct fibers of the~morphisms $\pi_1\vert_{E}\colon E\to\mathbb{P}_{x_0,x_1}^1$.

\begin{lemma}
\label{lemma:3-4-E}
Suppose that $P\in E$. Then $\delta_P(Y,\Delta_Y)\geqslant 1$.
Moreover, if $\mathfrak{C}\subset E$, then $\beta_{Y,\Delta_Y}(\mathbf{F})>0$.
\end{lemma}

\begin{proof}
Take $u\in\mathbb{R}_{\geqslant 0}$. From the~proof of Lemma~\ref{lemma:3-4-surfaces}, we know that
\begin{center}
$L-uE$ is pseudoeffective $\iff$ $L-uE$ is nef $\iff$ $u\in[0,1]$.
\end{center}

Let $\mathbf{l}$ and $\mathbf{s}$ be~some fibers of the~morphisms $\pi_1\vert_{E}\colon E\to\mathbb{P}^1_{x_0,x_1}$ and $\psi\vert_{E}\colon E\to\mathbb{P}^1_{y_0,y_1}$, respectively.
Choose $\mathbf{l}$ and $\mathbf{s}$ such that $P\in\mathbf{l}\cap\mathbf{s}$.
Take $v\in\mathbb{R}_{\geqslant 0}$. Then $(L-uE)\vert_{E}-v\mathbf{l}\sim_{\mathbb{R}}(1-v)\mathbf{l}+(1+u)\mathbf{s}$,
and this divisor is pseudoeffective $\iff$ it is nef $\iff$ $v\in[0,1]$.
Now, following \cite{AbbanZhuang,Book,Fujita2021}, we set
$$
S_L\big(W^{E}_{\bullet,\bullet};\mathbf{l}\big)=\frac{3}{L^3}\int\limits_0^1\int\limits_0^1\big((L-uE)\big\vert_{E}-v\mathbf{l}\big)^2dvdu
$$
and
$$
S_L\big(W_{\bullet, \bullet,\bullet}^{E,\mathbf{l}};P\big)=\frac{3}{L^3}\int\limits_0^1\int\limits_0^{1}\Big(\big(L-uE)\big\vert_{E}-v\mathbf{l}\big)\cdot\mathbf{l}\Big)^2dvdu.
$$
Integrating, we get $S_L(W^{E}_{\bullet,\bullet};\mathbf{l})=\frac{1}{2}$ and $S_L(W_{\bullet, \bullet,\bullet}^{E,\mathbf{l}};P)=\frac{7}{9}$.

If $\mathbf{l}$ is not an irreducible component of the~curve $R_E$, then it follows from \cite{AbbanZhuang,Book,Fujita2021} that
$$
\frac{A_{Y,\Delta_Y}(\mathbf{F})}{S_{L}(\mathbf{F})}\geqslant\delta_P(Y,\Delta_Y)\geqslant\min\left\{\frac{1}{S_L(W_{\bullet, \bullet,\bullet}^{E,\mathbf{l}}; P)}, \frac{1}{S_L(W_{\bullet,\bullet}^E;\mathbf{l})},\frac{1}{S_{L}(E)}\right\}=\frac{9}{7},
$$
because we computed $S_{L}(E)=\frac{5}{9}$ in the~proof of Lemma~\ref{lemma:3-4-surfaces}.
Similarly, if $\mathbf{l}\subset\mathrm{Supp}(R_E)$, then
$$
\frac{A_{Y,\Delta_Y}(\mathbf{F})}{S_{L}(\mathbf{F})}\geqslant
\delta_P(Y,\Delta_Y)\geqslant\min\left\{\frac{1}{S_L(W_{\bullet, \bullet,\bullet}^{E,\mathbf{l}}; P)},
\frac{1-\mathrm{ord}_{\mathbf{l}}(\Delta_E)}{S_L(W_{\bullet,\bullet}^E;\mathbf{f})},\frac{1}{S_{L}(E)}\right\}=1.
$$
Moreover, if $\mathfrak{C}=P$, then it follows from \cite{AbbanZhuang,Book,Fujita2021} that $\beta_{Y,\Delta_Y}(\mathbf{F})>0$.

Thus, we see that $\delta_P(Y,\Delta_Y)\geqslant 1$. In particular, we have $\beta_{Y,\Delta_Y}(\mathbf{F})\geqslant 0$.

To complete the~proof, we may assume that $\mathfrak{C}$ is a curve in $E$.
Let us show that $\beta_{Y,\Delta_Y}(\mathbf{F})>0$. Suppose that $\beta_{Y,\Delta_Y}(\mathbf{F})=0$.
Let us seek for a contradiction. As above, we let
$$
S_L\big(W^{E}_{\bullet,\bullet};\mathfrak{C}\big)=\frac{3}{L^3}\int\limits_0^1\int\limits_0^\infty\mathrm{vol}\big((L-uE)\vert_{E}-v\mathfrak{C}\big)dvdu.
$$
Then it follows from \cite{AbbanZhuang,Book,Fujita2021} that
$$
1=\frac{A_{Y,\Delta_Y}(\mathbf{F})}{S_{L}(\mathbf{F})}>\frac{1-\mathrm{ord}_{\mathfrak{C}}(\Delta_E)}{S_L(W_{\bullet,\bullet}^E;\mathfrak{C})}.
$$
If $\mathfrak{C}$ is an irreducible component of the~curve $R_E$, then $\mathfrak{C}=\mathbf{l}$,
so $S_L(W_{\bullet,\bullet}^E;\mathbf{l})=\frac{1}{2}$ and $\mathrm{ord}_{\mathbf{l}}(\Delta_E)=\frac{1}{2}$,
which gives us a contradiction.
Thus, we have $\mathrm{ord}_{\mathfrak{C}}(\Delta_E)=0$, which gives $S_L(W_{\bullet,\bullet}^E;\mathfrak{C})>1$. But
$$
S_L\big(W_{\bullet,\bullet}^E;\mathfrak{C}\big)\leqslant\mathrm{max}\Big\{S_L\big(W_{\bullet,\bullet}^E;\mathbf{l}\big),S_L\big(W_{\bullet,\bullet}^E;\mathbf{s}\big)\Big\},
$$
because $|\mathfrak{C}-\mathbf{l}|\ne\varnothing$ or $|\mathfrak{C}-\mathbf{s}|\ne\varnothing$. Hence, we conclude that $S_L(W_{\bullet,\bullet}^E;\mathbf{s})>1$.

Let us compute $S_L(W_{\bullet,\bullet}^E;\mathbf{s})$. For $v\in\mathbb{R}_{\geqslant 0}$, we have $(L-uE)\vert_{E}-v\mathbf{s}\sim_{\mathbb{R}}\mathbf{l}+(1+u-v)\mathbf{s}$,
and this divisor is pseudoeffective $\iff$ it is nef $\iff$ $v\in[0,1+u]$. Hence, we have
$$
1<S_L\big(W^{E}_{\bullet,\bullet};\mathbf{s}\big)=\frac{3}{L^3}\int\limits_0^1\int\limits_0^{1+u}\big(\mathbf{l}+(1+u-v)\mathbf{s}\big)^2dvdu=
\frac{3}{L^3}\int\limits_0^1\int\limits_0^{1+u}2(1+u-v)dvdu=\frac{7}{9},
$$
which is a contradiction.
\end{proof}

Let $R_F=R\vert_{F}$ and $\Delta_{F}=\frac{1}{2}R_F$. Set  $Z=S\cdot F$. Then $Z=\{x_1=0,y_1=0\}\subset Y$.

\begin{lemma}
\label{lemma:3-4-F}
Suppose that $R_F$ is smooth. Then $\delta_P(Y,\Delta_Y)\geqslant 1$.
If  $\mathfrak{C}=P$, then $\beta_{Y,\Delta_Y}(\mathbf{F})>0$.
\end{lemma}

\begin{proof}
We recall that $F=\{x_1=0\}\subset Y$. Let us identify $F=\mathbb{F}_1$ with coordinates $[y_0:y_1;z_0:z_1]$.
Let $\upsilon\colon F\to\mathbb{P}^2$ be the~blow up $[y_0:y_1;z_0:z_1]\mapsto[y_0z_0:y_1z_0:z_1]$,
and let $\mathbf{e}$ be its exceptional~curve.
Then $R_F\cap\mathbf{e}=\varnothing$, and $\upsilon(R_F)$ is a smooth conic in $\mathbb{P}^2$.
Moreover, we have
$$
R_F\sim 2\big(Z+\mathbf{e}\big),
$$
and $Z$ is the~fiber of the~natural projection $F\to\mathbb{P}_{y_0,y_1}^1$ over the~point $[0:1]$.

Take $u\in\mathbb{R}_{\geqslant 0}$.
Then $L-uF$ is pseudoeffective $\iff$ $L-uF$ is nef $\iff$  $u\leqslant 1$.
Set
$$
\delta_{P}\big(F,\Delta_F;W^{F}_{\bullet,\bullet}\big)=\inf_{\substack{\mathbf{f}/F,\\ P\in C_F(\mathbf{f})}}\frac{A_{F,\Delta_F}(\mathbf{f})}{S_L(W^{F}_{\bullet,\bullet};\mathbf{f})},
$$
where
$$
S_L\big(W^{F}_{\bullet,\bullet};\mathbf{f}\big)=\frac{3}{L^3}\int\limits_{0}^{1}\int\limits_0^\infty \mathrm{vol}\big((L-uF)\big\vert_{F}-v\mathbf{f}\big)dvdu,
$$
and the~infimum is taken over all prime divisors $\mathbf{f}$ over the~surface $F$ whose centers on $F$ contain~$P$.
Then  it follows from \cite{AbbanZhuang,Book,Fujita2021} that
$$
\frac{A_{Y,\Delta_Y}(\mathbf{F})}{S_{L}(\mathbf{F})}\geqslant\delta_P(Y,\Delta_{Y})\geqslant\min\left\{\delta_{P}\big(F,\Delta_F;W^{F}_{\bullet,\bullet}\big),\frac{1}{S_{L}(F)}\right\}.
$$
Further, if both these inequalities are equalities and $\mathfrak{C}=P$,
then \cite{AbbanZhuang,Book,Fujita2021} gives $\delta_P(Y,\Delta_{Y})=\frac{1}{S_{L}(F)}$.
Moreover, we know from the~proof of Lemma~\ref{lemma:3-4-surfaces} that $S_{L}(F)=\frac{1}{2}$.
Hence, to complete the~proof, it is enough to show that $\delta_{P}(F,\Delta_F;W^{F}_{\bullet,\bullet})\geqslant 1$.
Let us do this.

Note that $(F,\Delta_F)$ is a log Fano pair.
Recall from \cite{Book} that its $\delta$-invariant is the~number
$$
\delta(F,\Delta_F)=\inf_{\mathbf{f}/F}\frac{A_{F,\Delta_F}(\mathbf{f})}{S_{F,\Delta_F}(\mathbf{f})},
$$
where
$$
S_{F,\Delta_F}\big(\mathbf{f}\big)=\frac{1}{(K_F+\Delta_F)^2}\int\limits_0^\infty \mathrm{vol}\big(-(K_F+\Delta_F)-v\mathbf{f}\big)dv,
$$
and the~infimum is taken over all prime divisors $\mathbf{f}$ over the~surface $F$.
We claim that $\delta(F,\Delta_F)\geqslant 1$.
Indeed, either one can check this explicitly similar to what is done in \cite[\S~2]{Book},
or one can use the~fact that the~double cover of the~surface $F$ branched over the~curve $R_F$ is a smooth del Pezzo of degree~$6$,
which is known to be K-polystable, so $(F,\Delta_F)$ is also K-polystable \cite{Fujita2019b},
which gives~$\delta(F,\Delta_F)\geqslant 1$.
Then,~using the~idea of the~proof of \cite[Nemuro Lemma]{CheltsovFujitaKishimotoOkada}, we get
\begin{multline*}
S_L\big(W^{F}_{\bullet,\bullet};\mathbf{f}\big)=
\frac{3}{L^3}\int\limits_{0}^{1}\int\limits_0^\infty \mathrm{vol}\big((L-uF)\big\vert_{F}-v\mathbf{f}\big)dvdu=\frac{3}{L^3}\int\limits_{0}^{1}\int\limits_0^\infty \mathrm{vol}\big(L\big\vert_{F}-v\mathbf{f}\big)dvdu=\\
=\frac{3}{L^3}\int\limits_0^\infty \mathrm{vol}\big(L\big\vert_{F}-v\mathbf{f}\big)dv=\frac{1}{(K_F+\Delta_F)^2}\int\limits_0^\infty\mathrm{vol}\big(-(K_F+\Delta_F)-v\mathbf{f}\big)dv\leqslant A_{F,\Delta_F}(\mathbf{f})
\end{multline*}
for every divisor $\mathbf{f}$ over the~surface $F$. This exactly means that  $\delta_{P}(F,\Delta_F;W^{F}_{\bullet,\bullet})\geqslant 1$.
\end{proof}

Let $R_S=R\vert_{S}$ and $\Delta_{S}=\frac{1}{2}R_S$. Recall that $S=\{y_1=0\}$ and $Z=\{x_1=0,y_1=0\}\subset S$. Set
$$
C=\{y_1=0,az_1=bz_0\}\subset Y.
$$
Then $Z$ and $C$ are rulings of the~surface $S\cong\mathbb{P}^1\times\mathbb{P}^1$ such that $P=Z\cap C$.

\begin{lemma}
\label{lemma:3-4-S}
Suppose that $P\not\in E$. Then
\begin{enumerate}
\item[$(\mathrm{1})$] if $\mathfrak{C}\subset S$ and $\mathfrak{C}$ is a curve, then $\beta_{Y,\Delta_Y}(\mathbf{F})>0$,
\item[$(\mathrm{2})$] if $P\not\in R$, then $\delta_P(Y,\Delta_Y)>1$,
\item[$(\mathrm{3})$] if $P\in R$ and $R_S$ is smooth at $P$, then $\delta_P(Y,\Delta_Y)\geqslant 1$,
\item[$(\mathrm{4})$] if $P\in R$, $R_S$ is smooth at $P$, and $\mathfrak{C}=P$, then $\beta_{Y,\Delta_Y}(\mathbf{F})>0$,
\item[$(\mathrm{5})$] if $P\in R$, $R_S$ is smooth at $P$, and $Z\not\subset\mathrm{Supp}(R_S)$, then $\delta_P(Y,\Delta_Y)>1$.
\end{enumerate}
\end{lemma}

\begin{proof}
Let $u$ be a non-negative real number. Then $L-uS$ is pseudoeffective if and only if $u\leqslant 2$.
For~$u\in[0,2]$, let $P(u)$ be the~positive part of the~Zariski decomposition of the~divisor~$L-uS$,
and let $N(u)$ be the~negative part of the~Zariski decomposition of the~divisor~$L-uS$.  Then
$$
P(u)\sim_{\mathbb{R}}\left\{\aligned
&F+E+(2-u)S\ \text{ for } 0\leqslant u\leqslant 1, \\
&F+(2-u)(E+S)\ \text{ for } 1\leqslant u\leqslant 2,
\endaligned
\right.
$$
and
$$
N(u)= \left\{\aligned
&0\ \text{ for } 0\leqslant u\leqslant 1, \\
&(u-1)E\ \text{ for } 1\leqslant u\leqslant 2.
\endaligned
\right.
$$
Observe that $R_S\sim 2(Z+C)$ and
$$
P(u)\big\vert_{S}\sim_{\mathbb{R}}\left\{\aligned
&Z+C\ \text{ for } 0\leqslant u\leqslant 1, \\
&Z+(2-u)C\ \text{ for } 1\leqslant u\leqslant 2.
\endaligned
\right.
$$

Let $G$ be an~irreducible curve in $S$ that passes though $P$. Take $v\in\mathbb{R}_{\geqslant 0}$. Set
$$
t(u)=\sup\Big\{v\in \mathbb R_{\geqslant 0} \ \big|\ \text{the~divisor $P(u)\big|_{S}-vG$ is pseudoeffective}\Big\}.
$$
Since $S\cong\mathbb{P}^1\times\mathbb{P}^1$, the~divisor $P(u)|_S-vG$ is nef $\iff$ $v\leqslant t(u)$.
Set
$$
S_L\big(W^S_{\bullet,\bullet};G\big)=\frac{3}{L^3}\int\limits_{0}^{2}\int\limits_0^{t(u)}\big(P(u)\big|_{S}-vG\big)^2dvdu
$$
and
$$
S_L\big(W_{\bullet, \bullet,\bullet}^{S,G};P\big)=\frac{3}{L^3}\int\limits_0^2\int\limits_0^{t(u)}\Big(\big(P(u)|_S-vG\big)\cdot G\Big)^2dvdu.
$$
If $G=\mathfrak{C}$ is a curve in $S$, it follows from \cite{AbbanZhuang,Book,Fujita2021} that
$$
\frac{A_{Y,\Delta_Y}(\mathbf{F})}{S_{L}(\mathbf{F})}\geqslant
\min\left\{\frac{1-\mathrm{ord}_{\mathfrak{C}}(\Delta_S)}{S_L(W_{\bullet,\bullet}^S;G)},\frac{1}{S_{L}(S)}\right\}.
$$
Moreover, if this inequality is an equality, it further follows from \cite{AbbanZhuang,Book,Fujita2021} that
$$
\frac{A_{Y,\Delta_Y}(\mathbf{F})}{S_{L}(\mathbf{F})}=\frac{1}{S_{L}(S)}.
$$
On the~other hand,  we know from the~proof of Lemma~\ref{lemma:3-4-surfaces} that $S_{L}(S)=\frac{7}{9}$. Moreover, since $|G-Z|\ne\varnothing$ or $|G-C|\ne\varnothing$ holds, we have
$$
S_L\big(W^S_{\bullet,\bullet};G\big)\leqslant\max\big\{S_L\big(W^S_{\bullet,\bullet};Z\big),S_L\big(W^S_{\bullet,\bullet};C\big)\big\}.
$$
Therefore, to prove assertion ($\mathrm{1}$), it is enough to check that $S_L(W^S_{\bullet,\bullet};Z)\leqslant\frac{1}{2}$ and $S_L(W^S_{\bullet,\bullet};C)\leqslant\frac{1}{2}$.
This is not difficult. Indeed, if $G=Z$, then $t(u)=1$ for every $u\in[0,2]$, and
$$
S_L(W_{\bullet,\bullet}^S;Z)=\frac{1}{3}\int\limits_{0}^{1}\int\limits_{0}^{1}2(1-v)dvdu+\frac{1}{3}\int\limits_{1}^{2}\int\limits_{0}^{1}2(1-v)(2-u)dvdu=\frac{1}{2}.
$$
Similarly, if $G=C$, then
$$
t(u)=\left\{\aligned
&1\ \text{ for } 0\leqslant u\leqslant 1, \\
&2-u\ \text{ for } 1\leqslant u\leqslant 2,
\endaligned
\right.
$$
and
$$
S_L(W_{\bullet,\bullet}^S;C)=\frac{1}{3}\int\limits_{0}^{1}\int\limits_{0}^{1}2(1-v)dvdu+\frac{1}{3}\int\limits_{1}^{2}\int\limits_{0}^{2-u}2(2-u-v)dvdu=\frac{4}{9}.
$$
This proves ($\mathrm{1}$).

Let $G$ be one of the~curves $Z$ or $C$.
If $G\not\subset\mathrm{Supp}(R_S)$, then it follows from \cite{AbbanZhuang,Book,Fujita2021}~that
\begin{equation}
\label{equation:3-4-S-easy-case}
\frac{A_{Y,\Delta_Y}(\mathbf{F})}{S_{L}(\mathbf{F})}\geqslant
\delta_P(Y,\Delta_Y)\geqslant\min\left\{\frac{1-\mathrm{ord}_{P}(\Delta_S\vert_{G})}{S_L(W_{\bullet, \bullet,\bullet}^{S,G};P)},
\frac{1}{S_L(W_{\bullet,\bullet}^S;G)},\frac{1}{S_{L}(S)}\right\}.
\end{equation}
On the~other hand, we compute
$$
S_L(W_{\bullet, \bullet,\bullet}^{S,G};P)=\left\{\aligned
&\frac{4}{9}\ \text{if}\ G=Z, \\
&\frac{1}{2}\ \text{if}\ G=C.\\
\endaligned
\right.
$$
If $P\not\in R$, then $Z\not\subset\mathrm{Supp}(R_S)$ and $C\not\subset\mathrm{Supp}(R_S)$,
so \eqref{equation:3-4-S-easy-case} gives $\delta_P(Y,\Delta_Y)\geqslant\frac{9}{7}$.
This proves ($\mathrm{2}$).

Now, we suppose that $P\in R$ and $R_S$ is smooth at $P$.
Then $Z\not\subset\mathrm{Supp}(R_S)$ or~\mbox{$C\not\subset\mathrm{Supp}(R_S)$}.
Moreover, if $Z\not\subset\mathrm{Supp}(R_S)$ and $R_S$ intersects  $Z$ transversally at $P$,
then \eqref{equation:3-4-S-easy-case} gives $\delta_P(Y,\Delta_Y)\geqslant\frac{9}{8}$.
Therefore, to prove ($\mathrm{3}$), ($\mathrm{4}$) and ($\mathrm{5}$) we may assume that
\begin{itemize}
\item either $Z$ is an~irreducible component of the~curve $R_S$,
\item or the~curve $R_S$ is tangent to $Z$ at the~point $P$.
\end{itemize}
Then $C$ is not an~irreducible component of the~curve $R_S$,
and $R_S$ intersects $C$ transversally~at~$P$.
Hence, using \eqref{equation:3-4-S-easy-case}, we obtain $\delta_P(Y,\Delta_Y)\geqslant 1$.
This proves ($\mathrm{3}$).

We have $\beta_{Y,\Delta_Y}(\mathbf{F})\geqslant 0$.
If $\mathfrak{C}=P$ and $\beta_{Y,\Delta_Y}(\mathbf{F})=0$,
then both inequalities in \eqref{equation:3-4-S-easy-case} are equalities.
In this case, it follows from \cite{AbbanZhuang,Book,Fujita2021}~that
$\delta_P(Y,\Delta_Y)=\frac{1}{S_{L}(S)}=\frac{9}{7}$, which contradicts $\beta_{Y,\Delta_Y}(\mathbf{F})\leqslant 0$.
Therefore, if  $\mathfrak{C}=P$, then $\beta_{Y,\Delta_Y}(\mathbf{F})>0$.
This proves ($\mathrm{4}$).

Finally, let us prove ($\mathrm{5}$). We suppose that $Z$ is not an~irreducible component of the~curve $R_S$.
Then $R_S$ is tangent to the~curve $Z$ at the~point $P$.
Let $\alpha\colon\widetilde{S}\to S$ be the~blow up of the~point~$P$,
and let $\beta\colon\overline{S}\to\widetilde{S}$ be the~blow up of the~intersection point of the~$\alpha$-exceptional curve
and the~proper transform of the~curve $Z$.
Then there exists the~following commutative diagram:
$$
\xymatrix{
\widetilde{S}\ar@{->}[d]_\alpha&&\overline{S}\ar@{->}[ll]_\beta\ar@{->}[d]^\gamma\\
S&&\widehat{S}\ar@{->}[ll]^\rho}
$$
where $\gamma$ is the~contraction of the~proper transform of the~$\alpha$-exceptional curve to an~ordinary double point of the~surface $\widehat{S}$,
and $\rho$ is the~contraction of the~proper transform of the~$\beta$-exceptional curve.
Then $\widehat{S}$ is a~singular del Pezzo surface of degree~$6$,
and $\rho$ is a~weighted blow up with weights $(1,2)$.

Denote by $\widehat{Z}$, $\widehat{C}$, $R_{\widehat{S}}$ the~proper transforms on $\widehat{S}$ via $\rho$ of the~curves $Z$, $C$, $R_S$, respectively.
Let~$\mathbf{e}$~be the~$\rho$-exceptional curve, and let
$$
\widehat{t}(u)=\sup\Big\{v\in \mathbb R_{\geqslant 0} \ \big|\ \text{the~divisor $\rho^*\big(P(u)\big\vert_{S}\big)-v\mathbf{e}$ is pseudoeffective}\Big\}.
$$
Observe that
$$
\rho^*\big(P(u)\big\vert_{S}\big)-v\mathbf{e}\sim_{\mathbb{R}}
\left\{\aligned
&\widehat{Z}+\widehat{C}+(3-v)\mathbf{e}\ \text{for } u\in[0,1], \\
&\widehat{Z}+(2-u)\widehat{C}+(4-u-v)\mathbf{e}\ \text{for } u\in[1,2].
\endaligned
\right.
$$
Thus, we conclude that
$$
\widehat{t}(u)=\left\{\aligned
&3\ \text{for } u\in[0,1], \\
&4-u\ \text{for } u\in[1,2].
\endaligned
\right.
$$
Now, for every $u\in[0,2]$ and every $v\in[0,\widehat{t}(u)]$,
we let $\widehat{P}(u,v)$ be the~positive part of the~Zariski decomposition of
the~divisor $\rho^*(P(u)\vert_{S})-v\mathbf{e}$, and let $\widehat{N}(u,v)$ be its negative part.
Let
$$
S_L\big(W^{S}_{\bullet,\bullet};\mathbf{e}\big)=\frac{3}{L^3}\int\limits_0^2\int\limits_0^{\widehat{t}(u)}\big(\widehat{P}(u,v)\big)^2dvdu.
$$
For every point $O\in\mathbf{e}$, let
$$
S\big(W_{\bullet, \bullet,\bullet}^{\widehat{S},\mathbf{e}};O\big)=\frac{3}{L^3}\int\limits_0^2\int\limits_0^{\widehat{t}(u)}\big(\widehat{P}(u,v)\cdot\mathbf{e}\big)^2dvdu+F_O\big(W_{\bullet, \bullet,\bullet}^{\widehat{S},\mathbf{e}}\big),
$$
where
$$
F_O\big(W_{\bullet,\bullet,\bullet}^{\widehat{S},\mathbf{e}}\big)=\frac{6}{L^3}\int\limits_0^2\int\limits_0^{\widehat{t}(u)}\big(\widehat{P}(u,v)\cdot\mathbf{e}\big)\cdot \mathrm{ord}_O\big(\widehat{N}(u,v)\big|_\mathbf{e}\big)dvdu.
$$
Let $Q$ be the~singular point of the~surface $\widehat{S}$. Then $Q=\widehat{C}\cap\mathbf{e}$.
Let $\Delta_{\widehat{S}}=\frac{1}{2}R_{\widehat{S}}$ and $\Delta_{\mathbf{e}}=\frac{1}{2}Q+\Delta_{\widehat{S}}\vert_{\mathbf{e}}$.
Then it follows from \cite{AbbanZhuang,Book,Fujita2021} that
\begin{equation}
\label{equation:3-4-OK-case-final}
\delta_P(Y,\Delta_Y)\geqslant\min\left\{
\min_{O\in\mathbf{e}}\frac{A_{\mathbf{e},\Delta_{\mathbf{e}}}(O)}{S_L(W_{\bullet, \bullet,\bullet}^{\widehat{S},\mathbf{e}};O)},
\frac{A_{S,\Delta_S}(\mathbf{e})}{S_L(W_{\bullet,\bullet}^S;\mathbf{e})},\frac{A_{Y,\Delta_Y}(S)}{S_{L}(S)}\right\},
\end{equation}
where $A_{Y,\Delta_Y}(S)=1$, $A_{S,\Delta_S}(\mathbf{e})=2$, and
$A_{\mathbf{e},\Delta_{\mathbf{e}}}(O)=1-\mathrm{ord}_O(\Delta_{\mathbf{e}})$.
Note that, since $R_S\cdot Z=2$ and $R_S$ is tangent to $Z$ at $P$
from the assumption, the local intersection number of $R_S$ and $Z$ at $P$ is equal
to $2$. Thus we get $\widehat{C}\cap\widehat{Z}\cap\mathbf{e}=\varnothing$ and
we have $\Delta_{\mathbf{e}}=\frac{1}{2}\left(\widehat{C}\cap\mathbf{e}\right)+\frac{1}{2}
\left(\widehat{Z}\cap\mathbf{e}\right)$.
Moreover, if $0\leqslant u\leqslant 1$, then
$$
\widehat{P}(u,v)\sim_{\mathbb{R}}\left\{\aligned
&\widehat{Z}+\widehat{C}+(3-v)\mathbf{e}\ \text{ if } 0\leqslant v\leqslant 1, \\
&\widehat{C}+\frac{3-v}{2}\big(\widehat{Z}+2\mathbf{e}\big)\ \text{ if } 1\leqslant v\leqslant 2,\\
&\frac{3-v}{2}\big(2\widehat{C}+\widehat{Z}+2\mathbf{e}\big)\ \text{ if } 2\leqslant v\leqslant 3,
\endaligned
\right.
$$
and
$$
\widehat{N}(u,v)=\left\{\aligned
&0\ \text{ if } 0\leqslant v\leqslant 1, \\
&\frac{v-1}{2}\widehat{Z}\ \text{ if } 1\leqslant v\leqslant 2,\\
&\frac{v-1}{2}\widehat{Z}+(v-2)\widehat{C}\ \text{ if } 2\leqslant v\leqslant 3,
\endaligned
\right.
$$
which gives
$$
\big(\widehat{P}(u,v)\big)^2=\left\{\aligned
&2-\frac{v^2}{2}\ \text{ if } 0\leqslant v\leqslant 1, \\
&\frac{5}{2}-v\ \text{ if } 1\leqslant v\leqslant 2,\\
&\frac{(3-v)^2}{2}\ \text{ if } 2\leqslant v\leqslant 3,
\endaligned
\right.
$$
and
$$
\widehat{P}(u,v)\cdot\mathbf{e}=\left\{\aligned
&\frac{v}{2}\ \text{ if } 0\leqslant v\leqslant 1, \\
&\frac{1}{2}\ \text{ if } 1\leqslant v\leqslant 2,\\
&\frac{3-v}{2}\ \text{ if } 2\leqslant v\leqslant 3.
\endaligned
\right.
$$
Similarly, if $1\leqslant u\leqslant 2$, then
$$
\widehat{P}(u,v)\sim_{\mathbb{R}}\left\{\aligned
&\widehat{Z}+(2-u)\widehat{C}+(4-u-v)\mathbf{e}\ \text{ if } 0\leqslant v\leqslant 2-u, \\
&\frac{4-u-v}{2}\big(\widehat{Z}+2\mathbf{e}\big)+(2-u)\widehat{C}\ \text{ if } 2-u\leqslant v\leqslant 2,\\
&\frac{4-u-v}{2}\big(2\widehat{C}+\widehat{Z}+2\mathbf{e}\big)\ \text{ if } 2\leqslant v\leqslant 4-u,
\endaligned
\right.
$$
and
$$
\widehat{N}(u,v)=\left\{\aligned
&0\ \text{ if } 0\leqslant v\leqslant 2-u, \\
&\frac{v+u-2}{2}\widehat{Z}\ \text{ if } 2-u\leqslant v\leqslant 2,\\
&\frac{v+u-2}{2}\widehat{Z}+(v-2)\widehat{C}\ \text{ if } 2\leqslant v\leqslant 4-u,
\endaligned
\right.
$$
which gives
$$
\big(\widehat{P}(u,v)\big)^2=\left\{\aligned
&4-2u-\frac{v^2}{2}\ \text{ if } 0\leqslant v\leqslant 2-u, \\
&\frac{(u-2)(u+2v-6)}{2}\ \text{ if } 2-u\leqslant v\leqslant 2,\\
&\frac{(4-u-v)^2}{2}\ \text{ if } 2\leqslant v\leqslant 4-u,
\endaligned
\right.
$$
and
$$
\widehat{P}(u,v)\cdot\mathbf{e}=\left\{\aligned
&\frac{v}{2}\ \text{ if } 0\leqslant v\leqslant 2-u, \\
&1-\frac{u}{2}\ \text{ if } 2-u\leqslant v\leqslant 2,\\
&\frac{4-u-v}{2}\ \text{ if } 2\leqslant v\leqslant 4-u.
\endaligned
\right.
$$
Now, integrating, we get $S_L(W_{\bullet,\bullet}^S;\mathbf{e})=\frac{13}{9}<2=A_{S,\Delta_S}(\mathbf{e})$ and
$$
S_L\big(W_{\bullet, \bullet,\bullet}^{\widehat{S},\mathbf{e}};O\big)=
\left\{\aligned
&\frac{3}{16}\ \text{ if } O\not\in\widehat{Z}\cup\widehat{C}, \\
&\frac{2}{9}\ \text{ if } O\in\widehat{C},\\
&\frac{1}{2}\ \text{ if } O\in\widehat{Z}.
\endaligned
\right.
$$
Hence, using \eqref{equation:3-4-OK-case-final} and the description
of $\Delta_{\mathbf{e}}$, we obtain $\delta_P(Y,\Delta_Y)>1$ as required.
\end{proof}

Now, we are ready to prove

\begin{lemma}
\label{lemma:3-4-center-point}
Suppose that $\beta_{Y,\Delta_Y}(\mathbf{F})\leqslant 0$. Then $\mathfrak{C}$ is a point.
\end{lemma}

\begin{proof}
Suppose $\mathfrak{C}$ is not a point.
By Lemma~\ref{lemma:3-4-surfaces}, the~center $\mathfrak{C}$ is not a surface.
Then $\mathfrak{C}$ is a curve.
We may assume that $P$ is a general point in $\mathfrak{C}$.
By Lemma~\ref{lemma:3-4-E}, we have $\mathfrak{C}\not\subset E$, so $P\not\in E$ either.

If $\psi(\mathfrak{C})=\mathbb{P}^1_{y_0,y_1}$, then $S$ is a general fiber of the~morphism $\psi$,
which implies that $R_S$ is smooth, so that $Z\not\subset R$ by Lemma~\ref{lemma:F-S}.
Then $\delta_P(Y,\Delta_Y)>1$ by Lemma~\ref{lemma:3-4-S}, which contradicts $\beta_{Y,\Delta_Y}(\mathbf{F})\leqslant 0$.
Thus, we see that $\psi(\mathfrak{C})$ is a point in $\mathbb{P}^1_{y_0,y_1}$.
This means that $\mathfrak{C}\subset S$.

Now, applying Lemma~\ref{lemma:3-4-S}, we get $\beta_{Y,\Delta_Y}(\mathbf{F})>0$, which is a contradiction.
\end{proof}

Now, we suppose that $\beta_{Y,\Delta_Y}(\mathbf{F})\leqslant 0$. Let us seek for a contradiction.
First, applying Lemma~\ref{lemma:3-4-center-point}, we see that the~center $\mathfrak{C}=C_Y(\mathbf{F})$ is a point.
Using notations we introduced earlier, we have $P=\mathfrak{C}$.
Moreover, applying Lemmas~\ref{lemma:F-S}, \ref{lemma:3-4-E}, \ref{lemma:3-4-F}, \ref{lemma:3-4-S}, we obtain the~following assertions:
\begin{itemize}
\item $P\not\in E$ by Lemma~\ref{lemma:3-4-E},
\item $R_F$ is singular by Lemma~\ref{lemma:3-4-F},
\item $P\in R$ by Lemma~\ref{lemma:3-4-S},
\item $R_S$ is singular at $P$ by Lemma~\ref{lemma:3-4-S},
\item $R_F$ is smooth at $P$ by Lemma~\ref{lemma:F-S},
\item $Z\subset R$ by Lemma~\ref{lemma:3-4-S}.
\end{itemize}
In particular, the~curve $R_S$ is reducible.
Namely, we have $R_S=Z+T$,
where $T$ is a possibly reducible reduced curve in $|Z+2C|$ such that $P\in T$.

\begin{lemma}
\label{lemma:S-A1}
The curve $R_S$ does not have an~ordinary double singularity at $P$.
\end{lemma}

\begin{proof}
Suppose that $R_S$ has an~ordinary double singularity at $P$. Let us seek for a~contradiction.
Let us use notations introduced in the~proof of Lemma~\ref{lemma:3-4-S}.
Then we have
$$
P(u)\big\vert_{S}\sim_{\mathbb{R}}\left\{\aligned
&C+Z\ \text{ for } 0\leqslant u\leqslant 1, \\
&(2-u)C+Z\ \text{ for } 1\leqslant u\leqslant 2.\\
\endaligned
\right.
$$

Let $\alpha\colon\widetilde{S}\to S$ be the~blow up of the~point~$P$, let $\mathbf{e}$ be the~$\alpha$-exceptional curve.
For $u\in[0,2]$,~let
$$
\widetilde{t}(u)=\max\Big\{v\in \mathbb R_{\geqslant 0} \ \big| \ \alpha^*\big(P(u)|_S\big)-v\mathbf{e} \text{ is pseudoeffective}\Big\}.
$$
For $v\in [0,\widetilde{t}(u)]$, let $\widetilde{P}(u,v)$ be the~positive part of the~Zariski decomposition of $\alpha^*(P(u)|_S)-v\mathbf{e}$,
and let $\widetilde{N}(u,v)$ be the~negative part of the~Zariski decomposition of this divisor. Set
$$
S\big(W_{\bullet, \bullet}^{S};\mathbf{e}\big)=\frac{3}{L^3}\int\limits_0^2\int\limits_0^{\widetilde{t}(u)}\widetilde{P}(u,v)^2dvdu.
$$
Then, for every point $O\in\mathbf{e}$, we set
$$
S\big(W_{\bullet, \bullet,\bullet}^{\widetilde{S},\mathbf{e}};O\big)=
\frac{3}{L^3}\int\limits_0^2\int\limits_0^{\widetilde{t}(u)}\Big(\widetilde{P}(u,v)\cdot\mathbf{e}\Big)^2dvdu+F_O\big(W_{\bullet, \bullet,\bullet}^{\widetilde{S},\mathbf{e}}\big),
$$
where
$$
F_O\big(W_{\bullet, \bullet,\bullet}^{\widetilde{S},\mathbf{e}}\big)=
\frac{6}{L^3}\int\limits_0^2\int\limits_0^{\widetilde{t}(u)}\big(\widetilde{P}(u,v)\cdot\mathbf{e}\big)\cdot\mathrm{ord}_O\Big(\widetilde{N}(u,v)\big|_{\mathbf{e}}\Big)dvdu.
$$
Let $\widetilde{C}$, $\widetilde{Z}$, $\widetilde{T}$ be the~proper transforms on $\widetilde{S}$ of the~curves $C$, $Z$, $T$, respectively.
Set $\Delta_{\widetilde{S}}=\frac{1}{2}\widetilde{Z}+\frac{1}{2}\widetilde{T}$.
Then $\widetilde{Z}$ and $\widetilde{T}$ intersect $\mathbf{e}$ transversally at two distinct points,
since $T$ and $Z$ are not tangent at~$P$.
Set $\Delta_\mathbf{e}=\Delta_{\widetilde{S}}\vert_{\mathbf{e}}$.
Then it follows from \cite{AbbanZhuang,Book,Fujita2021}~that
$$
1\geqslant\frac{A_{Y,\Delta_Y}(\mathbf{F})}{S_{L}(\mathbf{F})}\geqslant
\delta_P(Y,\Delta_Y)\geqslant\min\left\{\min_{O\in\mathbf{e}}\frac{A_{\mathbf{e},\Delta_{\mathbf{e}}}(O)}{S(W_{\bullet, \bullet,\bullet}^{\widetilde{S},\mathbf{e}};O)},\frac{A_{S,\Delta_S}(\mathbf{e})}{S(W_{\bullet, \bullet}^{S};\mathbf{e})},\frac{A_{Y,\Delta_Y}(S)}{S_{L}(S)}\right\},
$$
and not all inequalities here are equalities. Note that $A_{Y,\Delta_Y}(S)=1$, $A_{S,\Delta_S}(\mathbf{e})=1$, and
$$
A_{\mathbf{e},\Delta_{\mathbf{e}}}(O)=1-\mathrm{ord}_{O}(\Delta_{\mathbf{e}})=
\left\{\aligned
&\frac{1}{2}\ \text{ if } O=\widetilde{Z}\cap\mathbf{e}, \\
&\frac{1}{2}\ \text{ if } O=\widetilde{T}\cap\mathbf{e}, \\
&1\ \text{ if } O\not\in\widetilde{Z}\cup\widetilde{T}.
\endaligned
\right.
$$
Since $S_{L}(S)=\frac{7}{9}$, we conclude that $S(W_{\bullet, \bullet}^{S};\mathbf{e})>1$ or
there exists a point $O\in\mathbf{e}$ such that
$$
S\big(W_{\bullet, \bullet,\bullet}^{\widetilde{S},\mathbf{e}};O\big)>1-\mathrm{ord}_{O}\big(\Delta_{\mathbf{e}}\big).
$$
Let us compute $S(W_{\bullet, \bullet}^{S};\mathbf{e})$, and let us compute $S(W_{\bullet, \bullet,\bullet}^{\widetilde{S},\mathbf{e}};O)$
for every point $O\in\mathbf{e}$.

Let $v$ be a non-negative real number.
Then
$$
\alpha^*\big(P(u)|_{S}\big)-v\mathbf{e}\sim_{\mathbb{R}}
\left\{\aligned
&\widetilde{C}+\widetilde{Z}+(2-v)\mathbf{e}\ \text{ for } 0\leqslant u\leqslant 1, \\
&(2-u)\widetilde{C}+\widetilde{Z}+(3-u-v)\mathbf{e}\ \text{ for } 1\leqslant u\leqslant 2.
\endaligned
\right.
$$
Since $\widetilde{Z}$ and $\widetilde{C}$ are disjoint $(-1)$-curves in $\widetilde{S}$, we have
$$
\widetilde{t}(u)=
\left\{\aligned
&2\ \text{ for } 0\leqslant u\leqslant 1, \\
&3-u\ \text{ for } 1\leqslant u\leqslant 2.
\endaligned
\right.
$$
Furthermore, if $0\leqslant u\leqslant 1$, then
$$
\widetilde{P}(u,v)\sim_{\mathbb{R}}
\left\{\aligned
&\widetilde{C}+\widetilde{Z}+(2-v)\mathbf{e}\ \text{ for } 0\leqslant v\leqslant 1, \\
&(2-v)\big(\widetilde{C}+\widetilde{Z}+\mathbf{e}\big)\ \text{ for } 1\leqslant v\leqslant 2,
\endaligned
\right.
$$
and
$$
\widetilde{N}(u,v)=
\left\{\aligned
&0\ \text{ for } 0\leqslant v\leqslant 1, \\
&(v-1)\big(\widetilde{C}+\widetilde{Z}\big)\ \text{ for } 1\leqslant v\leqslant 2,
\endaligned
\right.
$$
which gives
$$
\big(\widetilde{P}(u,v)\big)^2=
\left\{\aligned
&2-v^2\ \text{ for } 0\leqslant v\leqslant 1, \\
&(2-v)^2\ \text{ for } 1\leqslant v\leqslant 2,
\endaligned
\right.
$$
and
$$
\widetilde{P}(u,v)\cdot\mathbf{e}=
\left\{\aligned
&v\ \text{ for } 0\leqslant v\leqslant 1, \\
&2-v\ \text{ for } 1\leqslant v\leqslant 2.
\endaligned
\right.
$$
Similarly, if $1\leqslant u\leqslant 2$, then
$$
\widetilde{P}(u,v)\sim_{\mathbb{R}}
\left\{\aligned
&(2-u)\widetilde{C}+\widetilde{Z}+(3-u-v)\mathbf{e}\ \text{ for } 0\leqslant v\leqslant 2-u, \\
&(2-u)\widetilde{C}+(3-u-v)\big(\widetilde{Z}+\mathbf{e}\big)\ \text{ for } 2-u\leqslant v\leqslant 1,\\
&(3-u-v)\big(\widetilde{C}+\widetilde{Z}+\mathbf{e}\big)\ \text{ for } 1\leqslant v\leqslant 3-u,
\endaligned
\right.
$$
and
$$
\widetilde{N}(u,v)=
\left\{\aligned
&0\ \text{ for } 0\leqslant v\leqslant 2-u, \\
&(v+u-2)\widetilde{Z}\ \text{ for } 2-u\leqslant v\leqslant 1,\\
&(v+u-2)\widetilde{Z}+(v-1)\widetilde{C}\ \text{ for } 1\leqslant v\leqslant 3-u,
\endaligned
\right.
$$
which gives
$$
\big(\widetilde{P}(u,v)\big)^2=
\left\{\aligned
&4-2u-v^2\ \text{ for } 0\leqslant v\leqslant 2-u, \\
&(2-u)(4-u-2v) \ \text{ for } 2-u\leqslant v\leqslant 1,\\
&(3-u-v)^2\ \text{ for } 1\leqslant v\leqslant 3-u,
\endaligned
\right.
$$
and
$$
\widetilde{P}(u,v)\cdot\mathbf{e}=
\left\{\aligned
&v\ \text{ for } 0\leqslant v\leqslant 2-u, \\
&2-u \ \text{ for } 2-u\leqslant v\leqslant 1,\\
&3-u-v\ \text{ for } 1\leqslant v\leqslant 3-u.
\endaligned
\right.
$$
Therefore, integrating,
we get $S(W_{\bullet, \bullet}^{\widetilde{S}};\mathbf{e})=\frac{17}{18}<1$ and
$$
S\big(W_{\bullet, \bullet,\bullet}^{\widetilde{S},\mathbf{e}};O\big)=
\frac{11}{36}+F_O\big(W_{\bullet, \bullet,\bullet}^{\widetilde{S},\mathbf{e}}\big)=
\left\{\aligned
&\frac{11}{36}\ \text{ if } O\not\in\widetilde{C}\cup\widetilde{Z}\,\\
&\frac{4}{9}\ \text{ if } O\in\widetilde{C},\\
&\frac{1}{2}\ \text{ if } O\in\widetilde{Z},
\endaligned
\right.
$$
which gives $S(W_{\bullet, \bullet,\bullet}^{\widetilde{S},\mathbf{e}};O)\leqslant 1-\mathrm{ord}_{O}(\Delta_{\mathbf{e}})$ for every point $O\in\mathbf{e}$.
This is a contradiction.
\end{proof}

Now, using Lemma~\ref{lemma:S-A1}, we see that one of the~following two remaining cases occurs:
\begin{itemize}
\item[($\mathbb{A}_3$)] $R_S=Z+T$, where $T$ is a smooth curve in $|Z+2C|$ that is tangent to $Z$ at the~point $P$,
\item[($\mathbb{D}_4$)] $R_S=Z+T=Z+C+T^\prime$, where $T^\prime$ is a smooth curve in $|Z+C|$ such that $P\in T^\prime$.
\end{itemize}
This imposes certain constraints on the~equation \eqref{equation:R}, which can be listed as follows:
\begin{itemize}
\item $a_0=0$, since $P=([1:0],[1:0;1:0])\in R$,
\item $a_1=0$ and $d_0=0$, since $R_S$ is singular at $P$,
\item $f_0=0$ and $d_0=0$, since $Z\subset R$,
\item $d_1=0$, since $R_S$ does not have ordinary double point at $P$.
\end{itemize}
Changing coordinates on $Y$, we can simplified \eqref{equation:R} a bit more.
First, we may assume that $b_0=1$, since $R$ is smooth at $P$.
Second, we have $R\cap E=\{z_0=0,x_1(f_2x_0+f_1x_1)=0\}$, but $R\cap E$~is~smooth.
Hence, we can change the~coordinate $x_0$ such that $f_2=0$ and $f_1=1$.
This simplifies \eqref{equation:R} as
\begin{multline}
\label{equation:R-simplified}
\quad\quad\quad\quad\quad x_0^2\big((c_0y_1^2+y_0y_1)z_0^2+e_0y_1z_0z_1\big)+\\
+x_0x_1\big((b_1y_0y_1+c_1y_1^2)z_0^2+e_1y_1z_0z_1+z_1^2\big)+\\
+x_1^2\big((a_2y_0^2+b_2y_0y_1+c_2y_1^2)z_0^2+(d_2y_0+e_2y_1)z_0z_1\big)=0.\quad\quad\quad
\end{multline}
Recall that $S=\{y_1=0\}\subset Y$, so we can identify $S=\mathbb{P}^1\times\mathbb{P}^1$ with coordinates $([x_0:x_1],[z_0:z_1])$.
Using this identification, we see that $Z=\{x_1=0\}\subset S$, $C=\{z_1=0\}\subset S$, and
$$
T=\big\{a_2x_1z_0^2+d_2x_1z_0z_1+x_0z_1^2=0\big\}\subset S.
$$
Note that $T$ is irreducible $\iff$ $a_2\ne 0$.
Further, if $a_2=0$, then $T=C+T^\prime$ for $T^\prime=\{d_2x_1z_0+x_0z_1=0\}$,
where $d_2\ne 0$, since $R_S$ is reduced.
Thus, the~cases ($\mathbb{A}_3$) and  ($\mathbb{D}_4$) can be described as follows:
\begin{itemize}
\item[($\mathbb{A}_3$)] $a_2\ne 0$,
\item[($\mathbb{D}_4$)] $a_2=0$ and $d_2\ne 0$.
\end{itemize}
We will exclude the~remaining cases ($\mathbb{D}_4$) and ($\mathbb{A}_3$) in Sections~\ref{subsection:3-4-step-2} and \ref{subsection:3-4-step-3}, respectively.

\subsection{Exclusion of the~case ($\mathbb{D}_4$)}
\label{subsection:3-4-step-2}

Let us continue the~proof of Theorem~\ref{theorem:3-4-K-stable} started in Section~\ref{subsection:3-4-step-1}.
Now, we assume that the~surface $R$ is given by \eqref{equation:R-simplified} and we have $a_2=0$, i.e. we are in the~case~($\mathbb{D}_4$).
In the~chart $\mathbb{A}^3_{x,y,z}=\{x_0y_0z_0\ne 0\}$ with coordinates $x=\frac{x_1}{x_0}$, $y=\frac{y_1}{y_0}$, $z=\frac{z_1}{z_0}$,
we have $P=(0,0,0)$, and the~surface $R$ is given by the~following equation:
$$
y+xz^2+d_2x^2z+\big(b_1xy+e_0yz+b_2x^2y+e_1xyz+e_2x^2yz+c_0y^2+c_1xy^2+c_2x^2y^2\big)=0,
$$
where $y+xz^2+d_2x^2z$ is the~smallest degree term for the~weights $\mathrm{wt}(x)=1$, $\mathrm{wt}(y)=3$, $\mathrm{wt}(z)=1$.
Let $\lambda\colon W_0\to Y$ be the~corresponding weighted blow up of the~point $P$ with weights $(1,3,1)$,
and let $G$ be the~$\lambda$-exceptional surface. Then~$G\cong\mathbb{P}(1,3,1)$.

Let  $R_{W_0}$, $F_{W_0}$ and $S_{W_0}$ be the~proper transforms on $W_0$ of the~surfaces $R$, $F$ and $S$, respectively.
Set $R_G=R_{W_0}\vert_{G}$, $\Delta_G=\frac{1}{2}R_G$ and $\Delta_{W_0}=\frac{1}{2}R_{W_0}$.
Note that
$$
\big(K_{W_0}+\Delta_{W_0}+G\big)\big|_G\sim_{\mathbb{Q}}K_G+\Delta_G.
$$
Let us also consider $(x,y,z)$ as coordinates on $G\cong\mathbb{P}(1,3,1)$ with $\mathrm{wt}(x)=1$,
$\mathrm{wt}(y)=3$, $\mathrm{wt}(z)=1$. Then $F_{W_0}\vert_{G}=\{x=0\}$, $S_{W_0}\vert_{G}=\{y=0\}$, and
$$
R_G=\big\{y+xz^2+d_2x^2z=0\big\}\subset R.
$$
Recall from the~end of Section~\ref{subsection:3-4-step-1} that $d_2\ne 0$.
Since $\mathrm{ord}_G(R)=3$, we have $A_{Y,\Delta_Y}(G)=\frac{7}{2}$. Then
$$
\delta_P(Y,\Delta_Y)\leqslant\frac{A_{Y,\Delta_Y}(G)}{S_{L}(G)}=\frac{7}{2S_{L}(G)},
$$
where
$$
S_{L}(G)=\frac{1}{L^3}\int\limits_{0}^{\infty}\mathrm{vol}\big(\lambda^*(L)-uG\big)du.
$$
Let us compute $S_{L}(G)$. To do this, note that $Y$ is toric,
and the~blow up $\lambda\colon W_0\to Y$ is also toric for the~torus action on $Y$ with an open orbit $\{x_0y_0z_0x_1y_1z_1\ne 0\}\subset Y$,
so the~threefold $W_0$ is toric, and $G$ is a torus invariant divisor.
Let us present toric data for the~threefolds $Y$ and $W_0$.

Let $\Sigma_Y$ be the~simplicial fan in $\mathbb{R}^3$ defined by the~following data:
\begin{itemize}
\item the~list of primitive generators of rays of $\Sigma_Y$ is
$$
v_1=(1,0,0),  v_2=(0,0,1),  v_3=(0,1,0), v_4=(0,0,-1),  v_5=(0,-1,1),  v_6=(-1,0,0);
$$

\item the~list of maximal cones of $\Sigma_Y$ is
$$
[1,2,3],  [1,3,4],   [1,4,5],  [1,2,5], [2,3,6],  [3,4,6],  [4,5,6],  [2,5,6],
$$
where  $[i,j,k]$ is the~cone generated by the~rays $v_i$, $v_j$, and $v_k$.
\end{itemize}
Then $Y$ is defined by $\Sigma_Y$.
Let $\Sigma_{W_0}$ be the~simplicial fan in $\mathbb{R}^3$ defined by the~following data:
\begin{itemize}
\item the~list of primitive generators of rays in $\Sigma_{W_0}$ is
\begin{align*}
v_0&=(1,3,-1), & v_1&=(1,0,0), & v_2&=(0,0,1), & v_3&=(0,1,0), \\
v_4&=(0,0,-1), & v_5&=(0,-1,1), & v_6&=(-1,0,0);
\end{align*}
\item  the~list of maximal cones  in  $\Sigma_{W_0}$ is
$$
[0,1,3],   [0,1,4],   [0,3,4],   [1,2,3],   [1,2,5], [1,4,5],      [2,3,6],  [2,5,6],  [3,4,6],   [4,5,6].
$$
\end{itemize}
Then the~toric threefold $W_0$ is given by the~fan $\Sigma_{W_0}$, which can be diagramed as follows:
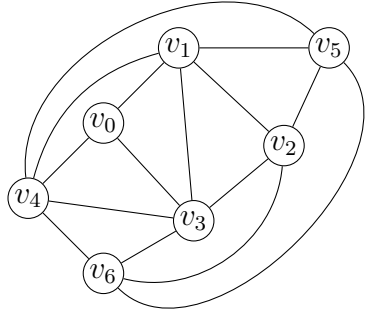
\begin{center}
\begin{tikzpicture}[node distance=1cm,auto]
\node[state, inner sep=1pt,minimum size=0pt] (q_0) {$v_0$};
\node[state, inner sep=1pt,minimum size=0pt] (q_1) at (1,1) {$v_1$};
\node[state,inner sep=1pt,minimum size=0pt] (q_2) at (2.4,-.3) {$v_2$};
\node[state,inner sep=1pt,minimum size=0pt] (q_3) at (1.2,-1.3) {$v_3$};
\node[state,inner sep=1pt,minimum size=0pt] (q_4) at (-1,-1) {$v_4$};
\node[state,inner sep=1pt,minimum size=0pt] (q_5) at (3,1) {$v_5$};
\node[state,inner sep=1pt,minimum size=0pt] (q_6) at (0,-2) {$v_6$};

\path[-] (q_0) edge node[swap] {} (q_1);
\path[-] (q_0) edge node[swap] {} (q_3);
\path[-] (q_0) edge node[swap] {} (q_4);
\path[-] (q_1) edge node[swap] {} (q_2);
\path[-] (q_1) edge node[swap] {} (q_3);
\path[-] (q_1) edge node[swap] {} (q_5);
\path[-] (q_3) edge node[swap] {} (q_4);
\path[-] (q_4) edge node[swap] {} (q_6);
\path[-] (q_3) edge node[swap] {} (q_6);
\path[-] (q_2) edge node[swap] {} (q_3);
\path[-] (q_2) edge node[swap] {} (q_5);
\path[-, bend right = 30] (q_1) edge node {} (q_4);
\path[-, bend left = 70] (q_4) edge node {} (q_5);
\path[-, bend left = 50] (q_2) edge node {} (q_6);
\path[-, bend left = 90] (q_5) edge node {} (q_6);
\end{tikzpicture}
\end{center}

Let us compute $S_{L}(G)$.
Let  $P_{L}$ be the~convex polytope in the~dual space of $\mathbb{R}^3$ associated~to~$L$.
Then, since $L$ corresponds to the~lattice point $(1,2,1)$, we have
$$
P_L=\big\{x_1\geqslant -1, x_3\geqslant -1, x_2\geqslant -2, -x_3\geqslant 0, -x_2+x_3\geqslant 0, -x_1\geqslant 0\big\}.
$$
Thus, since $G$ corresponds to $v_0=(1,3,-1)$, it follows from  \cite[Corollary~7.7]{BlumJonsson} that
$$
S_{L}(G)=-\min_{v\in P_{L}\cap \mathbb{Z}^3}v\cdot (1,3,-1)+\frac{3!}{L^3}\iiint_{P_{L}}(x_1, x_2, x_3)\cdot (1,3,-1)dx_1dx_2dx_3=\frac{59}{18}.
$$
where $\cdot$ is the~standard inner product in $\mathbb{R}^3$. Consequently, we obtain $\frac{A_{Y,\Delta_Y}(G)}{S_{L}(G)}=\frac{63}{59}.$

Now, let us exclude the~case ($\mathbb{D}_4$) using the~results obtained in \cite{AbbanZhuang,Book,Fujita2021}.
To do this, we~must find the~Zariski decomposition of the~divisor $\lambda^*(L)-uG$ for every $u\in\mathbb{R}_{\geqslant 0}$.
First, let us compute intersections of torus invariant divisors in $W_0$.
Let $T_i$ be the~torus invariant divisor corresponding to the~ray $v_i$. Then $T_0=G$,
and it follows from \cite[\S 6.4]{CLS} that
\begin{equation}
\label{eq:intersection}
T_iT_jT_k=\left\{
\aligned
& \frac{1}{|[i,j,k]|} \ \ \ \ \ \text {if $[i,j,k]$ belongs to the~list of maximal cones  in  $\Sigma_{W_0}$}\\
& 0 \ \ \ \ \ \text{otherwise},
\endaligned  \right.
\end{equation}
where $|[i,j,k]|$ stands for the~absolute value of the~determinant of the~$3\times 3$ matrix given by $v_i, v_j, v_k$.
This gives $T_0T_1T_4=\frac{1}{3}$ and
$$
T_0T_1T_3=T_0T_3T_4=T_1T_2T_3=T_1T_4T_5=T_1T_2T_5=T_2T_3T_6=T_3T_4T_6=T_4T_5T_6=T_2T_5T_6=1,
$$
while all other $T_iT_jT_k=0$ with distinct indices $i,j,k$.
The characters $\chi_1$, $\chi_2$, $\chi_3$ corresponding to the~lattice points $(1,0,0)$, $(0,1,0)$, $(0,0,1)$ in the~dual lattice
generate the~following relations among  the~torus invariant divisors:
\begin{equation}
\label{eq:principal1}
\aligned
&0\sim \mathrm{div}(\chi_1)=T_0+T_1-T_6, \\
&0\sim \mathrm{div}(\chi_2)=3T_0+T_3-T_5, \\
&0\sim \mathrm{div}(\chi_3)=-T_0+T_2-T_4+T_5.
\endaligned
\end{equation}
Now, using these relations, we can determine the~intersection numbers $T_i^2T_j$ for $i\ne j$.
For instance, we have $T_3^2T_6=(T_5-3T_0)T_3T_6=0$ and $T_2^2T_6=(T_0+T_4-T_5)T_2T_6=-1$.

For all possible indices $i\ne j$, let us denote by $T_iT_j$ the~torus invariant curve that
is given by the~intersection of the~divisors $T_i$ and $T_j$ provided that $T_i\cap T_j\ne\varnothing$.
Note that
\begin{center}
$T_i\cap T_j\ne\varnothing$ $\iff$ the~$2$-dimensional cone generated by  $v_i$ and  $v_j$ belongs to the~fan~$\Sigma_{W_0}$.
\end{center}
If $T_i\cap T_j\ne\varnothing$, then $T_iT_j$ is not necessarily reduced,
but its support coincides with the~torus invariant curve that corresponds to the~$2$-dimensional cone generated by the~rays $v_i$ and~$v_j$,
which we will denote by $\lfloor T_iT_j\rfloor$.

Let $u$ be a non-negative real number. For simplicity, set $L_u=\lambda^*(L)-uT_0$. Then
$$
L_u= (7-u)T_0+T_1+T_2+2T_3.
$$
Now, we can compute the~intersection of the~$\mathbb{R}$-divisor $L_u$ with each torus invariant curve in $W_0$.
For instance we have
\begin{align*}
L_uT_0T_1&=\left((7-u)T_0+T_1+T_2+2T_3\right)T_0T_1 =(7-u)T_0^2T_1+T_0T_1^2+T_0T_1T_2+2T_0T_1T_3=\frac{u}{3},\\
L_uT_0T_3&=\left((7-u)T_0+T_1+T_2+2T_3\right)T_0T_3 =(7-u)T_0^2T_3+T_0T_1T_3+T_0T_2T_3+2T_0T_3^2=u, \\
L_uT_0T_4&=\left((7-u)T_0+T_1+T_2+2T_3\right)T_0T_4 =(7-u)T_0^2T_4+T_0T_1T_4+T_0T_2T_4+2T_0T_3T_4=\frac{u}{3}.
\end{align*}
Similarly, we compute
\begin{align*}
L_uT_1T_2&=1, & L_uT_1T_3&=1-u, & L_uT_1T_4&=\frac{6-u}{3}, & L_uT_1T_5&=1, & L_uT_2T_3&=1, & L_uT_2T_5&=1, \\
L_uT_2T_6&=1, & L_uT_3T_4&=1-u, & L_uT_3T_6&=1,             & L_uT_4T_5&=1, & L_uT_4T_6&=2, & L_uT_5T_6&=1.
\end{align*}
Therefore, we see that $L_u$ is nef for $0\leqslant u\leqslant 1$.

To find Zariski decomposition of the~divisor $L_u$ for small $u>1$,
we must perform a small birational map $W_0\dasharrow W_1$ along the~two torus invariant curves  $\lfloor T_1T_3\rfloor$ and $\lfloor T_3T_4\rfloor$,
because these are the~only curves  that intersect $L_u$ negatively for small $u>1$.
The~corresponding change of fans can be diagramed as follows:

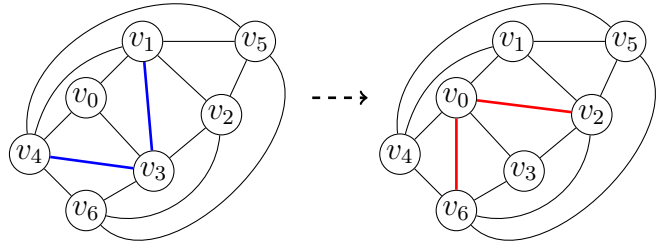
\begin{center}
\begin{tikzpicture}[node distance=1cm,auto, scale=0.75]
\node[state, inner sep=1pt,minimum size=0pt] (q_0) {$v_0$};
\node[state, inner sep=1pt,minimum size=0pt] (q_1) at (1,1) {$v_1$};
\node[state,inner sep=1pt,minimum size=0pt] (q_2) at (2.4,-.3) {$v_2$};
\node[state,inner sep=1pt,minimum size=0pt] (q_3) at (1.2,-1.3) {$v_3$};
\node[state,inner sep=1pt,minimum size=0pt] (q_4) at (-1,-1) {$v_4$};
\node[state,inner sep=1pt,minimum size=0pt] (q_5) at (3,1) {$v_5$};
\node[state,inner sep=1pt,minimum size=0pt] (q_6) at (0,-2) {$v_6$};

\node[state,inner sep=0pt,minimum size=0pt] (q_7) at (4,0) {};
\node[state,inner sep=0pt,minimum size=0pt] (q_8) at (5,0) {};
\path[->, dashed, line width=1pt] (q_7) edge node[swap] {} (q_8);

\path[-] (q_0) edge node[swap] {} (q_1);
\path[-] (q_0) edge node[swap] {} (q_3);
\path[-] (q_0) edge node[swap] {} (q_4);
\path[-] (q_1) edge node[swap] {} (q_2);
\path[-, blue, line width=1pt] (q_1) edge node[swap] {} (q_3);
\path[-] (q_1) edge node[swap] {} (q_5);
\path[-, blue, line width=1pt] (q_3) edge node[swap] {} (q_4);
\path[-] (q_4) edge node[swap] {} (q_6);
\path[-] (q_3) edge node[swap] {} (q_6);
\path[-] (q_2) edge node[swap] {} (q_3);
\path[-] (q_2) edge node[swap] {} (q_5);
\path[-, bend right = 30] (q_1) edge node {} (q_4);
\path[-, bend left = 70] (q_4) edge node {} (q_5);
\path[-, bend left = 50] (q_2) edge node {} (q_6);
\path[-, bend left = 90] (q_5) edge node {} (q_6);
\end{tikzpicture}
\begin{tikzpicture}[node distance=1cm,auto, scale=0.75]
\node[state, inner sep=1pt,minimum size=0pt] (q_0) {$v_0$};
\node[state, inner sep=1pt,minimum size=0pt] (q_1) at (1,1) {$v_1$};
\node[state,inner sep=1pt,minimum size=0pt] (q_2) at (2.4,-.3) {$v_2$};
\node[state,inner sep=1pt,minimum size=0pt] (q_3) at (1.2,-1.3) {$v_3$};
\node[state,inner sep=1pt,minimum size=0pt] (q_4) at (-1,-1) {$v_4$};
\node[state,inner sep=1pt,minimum size=0pt] (q_5) at (3,1) {$v_5$};
\node[state,inner sep=1pt,minimum size=0pt] (q_6) at (0,-2) {$v_6$};

\path[-] (q_0) edge node[swap] {} (q_1);
\path[-] (q_0) edge node[swap] {} (q_3);
\path[-] (q_0) edge node[swap] {} (q_4);
\path[-] (q_1) edge node[swap] {} (q_2);
\path[-, red, line width=1pt] (q_0) edge node[swap] {} (q_2);
\path[-] (q_1) edge node[swap] {} (q_5);
\path[-, red, line width=1pt] (q_0) edge node[swap] {} (q_6);
\path[-] (q_4) edge node[swap] {} (q_6);
\path[-] (q_3) edge node[swap] {} (q_6);
\path[-] (q_2) edge node[swap] {} (q_3);
\path[-] (q_2) edge node[swap] {} (q_5);

\path[-, bend right = 30] (q_1) edge node {} (q_4);
\path[-, bend left = 70] (q_4) edge node {} (q_5);
\path[-, bend left = 50] (q_2) edge node {} (q_6);
\path[-, bend left = 90] (q_5) edge node {} (q_6);
\end{tikzpicture}
\end{center}
The toric 3-fold $W_1$ is defined by the~simplicial fan $\Sigma_{W_1}$ in $\mathbb{R}^3$ determined by the~following data:
\begin{itemize}
\item the~list of primitive generators of rays of $\Sigma_{W_1}$ is
\begin{align*}
v_0&=(1,3,-1), & v_1&=(1,0,0), & v_2&=(0,0,1), & v_3&=(0,1,0), \\
v_4&=(0,0,-1), & v_5&=(0,-1,1), & v_6&=(-1,0,0);
\end{align*}
\item the~list of maximal cones of $\Sigma_{W_1}$ is
$$
[0,1,2], [0,2,3],  [0,3,6],  [0,4,6],  [0,1,4],  [1,4,5],  [1,2,5], [2,3,6], [4,5,6], [2,5,6].
$$
\end{itemize}
On the~3-fold $W_1$, we use the~same notations for the~transformed $L_u$, the~torus invariant divisors and curves as on $W_0$.
Since the~formula \eqref{eq:intersection} is valid on $W_1$, we get
$$
\aligned &T_0T_1T_2= T_0T_1T_4=T_0T_4T_6=\frac{1}{3}, \\
& T_0T_2T_3=T_0T_3T_6=T_1T_4T_5=T_1T_2T_5=T_2T_3T_6=T_4T_5T_6=T_2T_5T_6=1,\\
\endaligned
$$
and all other $T_iT_jT_k=0$ with distinct indices $i, j, k$.
Since we have the~same list of primitive generators of rays as on $\Sigma_{W_0}$, the~relations \eqref{eq:principal1} are valid on $W_1$.
This gives
\begin{align*}
L_uT_0T_1&=\frac{1}{3},   & L_uT_0T_2&=\frac{u-1}{3},  & L_uT_0T_3&=2-u,  & L_uT_0T_4&=\frac{1}{3},   & L_uT_0T_6&=\frac{u-1}{3}, \\
L_uT_1T_2&=\frac{4-u}{3}, & L_uT_1T_4&=\frac{6-u}{3},  & L_uT_1T_5&=1,    & L_uT_2T_3&=2-u,           & L_uT_2T_5&=1, \\
L_uT_2T_6&=1,             & L_uT_3T_6&=2-u,            & L_uT_4T_5&=1,    & L_uT_4T_6&=\frac{7-u}{3}, & L_uT_5T_6&=1.
\end{align*}
Therefore, the~divisor $L_u$ is nef for $1\leqslant u\leqslant 2$.

The~unique torus invariant surface $T_3$  that contains $\lfloor T_0T_3\rfloor $, $\lfloor T_2T_3\rfloor $, $\lfloor T_3T_6\rfloor $ is of Picard rank~$1$.
Since $(L_u-aT_3)T_0T_3=2-u+3a$ for any non-negative real number~$a$, the~Nakayama--Zariski decomposition $L_u=P(u)+N(u)$ for $u>2$ on the~3-fold $W_1$ must satisfy
$$
N(u)\geqslant \frac{u-2}{3}T_3,
$$
where $P(u)$ is the~positive part of the~decomposition, and $N(u)$ is the~negative part. Set
$$
P^1_u=L_u-\frac{u-2}{3}T_3=(7-u)T_0+T_1+T_2+\frac{8-u}{3}T_3.
$$
Then
\begin{align*}
P^1_uT_0T_1&=\frac{1}{3},   & P^1_uT_0T_2&=\frac{1}{3},   & P^1_uT_0T_3&=0, & P^1_uT_0T_4&=\frac{1}{3},   & P^1_uT_0T_6&=\frac{1}{3}, \\
P^1_uT_1T_2&=\frac{4-u}{3}, & P^1_uT_1T_4&=\frac{6-u}{3}, & P^1_uT_1T_5&=1, & P^1_uT_2T_3&=0,             & P^1_uT_2T_5&=1, \\
P^1_uT_2T_6&=\frac{5-u}{3}, & P^1_uT_3T_6&=0,             & P^1_uT_4T_5&=1, & P^1_uT_4T_6&=\frac{7-u}{3}, & P^1_uT_5T_6&=1.
\end{align*}
Therefore, if $2\leqslant u\leqslant 4$, then $P^1_u$ is nef, and hence $L_u=P^1_u+\frac{u-2}{3}T_3$ is the~Zariski decomposition,
i.e.~$P^1_u$ is the~positive part, and $\frac{u-2}{3}T_3$ is the~negative part.

For  small enough $u>4$, the~curve $\lfloor T_1T_2 \rfloor$ is the~only curve in $W_1$ that intersects $P^1_u$ negatively.
Let $W_1\dasharrow W_2$ be the~small birational map of this curve.
Then the~change of fans can be diagramed as follows:

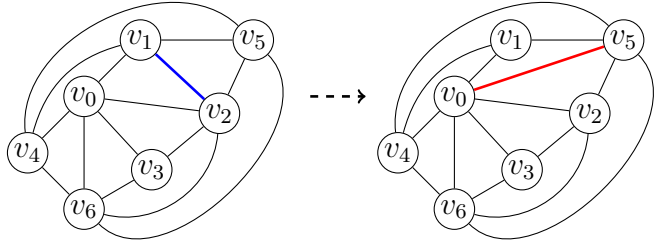
\begin{center}
\begin{tikzpicture}[node distance=1cm,auto, scale=0.75]
\node[state, inner sep=1pt,minimum size=0pt] (q_0) {$v_0$};
\node[state, inner sep=1pt,minimum size=0pt] (q_1) at (1,1) {$v_1$};
\node[state,inner sep=1pt,minimum size=0pt] (q_2) at (2.4,-.3) {$v_2$};
\node[state,inner sep=1pt,minimum size=0pt] (q_3) at (1.2,-1.3) {$v_3$};
\node[state,inner sep=1pt,minimum size=0pt] (q_4) at (-1,-1) {$v_4$};
\node[state,inner sep=1pt,minimum size=0pt] (q_5) at (3,1) {$v_5$};
\node[state,inner sep=1pt,minimum size=0pt] (q_6) at (0,-2) {$v_6$};

\path[-] (q_0) edge node[swap] {} (q_1);
\path[-] (q_0) edge node[swap] {} (q_3);
\path[-] (q_0) edge node[swap] {} (q_4);
\path[-, blue, line width=1pt] (q_1) edge node[swap] {} (q_2);
\path[-] (q_0) edge node[swap] {} (q_2);
\path[-] (q_1) edge node[swap] {} (q_5);
\path[-] (q_0) edge node[swap] {} (q_6);
\path[-] (q_4) edge node[swap] {} (q_6);
\path[-] (q_3) edge node[swap] {} (q_6);
\path[-] (q_2) edge node[swap] {} (q_3);
\path[-] (q_2) edge node[swap] {} (q_5);

\path[-, bend right = 30] (q_1) edge node {} (q_4);
\path[-, bend left = 70] (q_4) edge node {} (q_5);
\path[-, bend left = 50] (q_2) edge node {} (q_6);
\path[-, bend left = 90] (q_5) edge node {} (q_6);

\node[state,inner sep=0pt,minimum size=0pt] (q_7) at (4,0) {};
\node[state,inner sep=0pt,minimum size=0pt] (q_8) at (5,0) {};
\path[->, dashed, line width=1pt] (q_7) edge node[swap] {} (q_8);

\end{tikzpicture}
\begin{tikzpicture}[node distance=1cm,auto, scale=0.75]
\node[state, inner sep=1pt,minimum size=0pt] (q_0) {$v_0$};
\node[state, inner sep=1pt,minimum size=0pt] (q_1) at (1,1) {$v_1$};
\node[state,inner sep=1pt,minimum size=0pt] (q_2) at (2.4,-.3) {$v_2$};
\node[state,inner sep=1pt,minimum size=0pt] (q_3) at (1.2,-1.3) {$v_3$};
\node[state,inner sep=1pt,minimum size=0pt] (q_4) at (-1,-1) {$v_4$};
\node[state,inner sep=1pt,minimum size=0pt] (q_5) at (3,1) {$v_5$};
\node[state,inner sep=1pt,minimum size=0pt] (q_6) at (0,-2) {$v_6$};

\path[-] (q_0) edge node[swap] {} (q_1);
\path[-] (q_0) edge node[swap] {} (q_3);
\path[-] (q_0) edge node[swap] {} (q_4);
\path[-, red, line width=1pt] (q_0) edge node[swap] {} (q_5);
\path[-] (q_0) edge node[swap] {} (q_2);
\path[-] (q_1) edge node[swap] {} (q_5);
\path[-] (q_0) edge node[swap] {} (q_6);
\path[-] (q_4) edge node[swap] {} (q_6);
\path[-] (q_3) edge node[swap] {} (q_6);
\path[-] (q_2) edge node[swap] {} (q_3);
\path[-] (q_2) edge node[swap] {} (q_5);

\path[-, bend right = 30] (q_1) edge node {} (q_4);
\path[-, bend left = 70] (q_4) edge node {} (q_5);
\path[-, bend left = 50] (q_2) edge node {} (q_6);
\path[-, bend left = 90] (q_5) edge node {} (q_6);
\end{tikzpicture}
\end{center}
The toric 3-fold $W_2$ is  defined by the~simplicial fan $\Sigma_{W_2}$ in $\mathbb{R}^3$ determined by the~following data:
\begin{itemize}
\item the~list of primitive generators of rays of $\Sigma_{W_2}$ is
\begin{align*}
v_0&=(1,3,-1), & v_1&=(1,0,0), & v_2&=(0,0,1), & v_3&=(0,1,0), \\
v_4&=(0,0,-1), & v_5&=(0,-1,1), & v_6&=(-1,0,0);
\end{align*}
\item the~list of maximal cones of $\Sigma_{W_2}$ is
$$
[0,1,5],  [0,2,5],   [0,2,3],   [0,3,6],  [ 0,4,6],  [0,1,4],  [1,4,5],  [2,3,6], [4,5,6],  [2,5,6].
$$
\end{itemize}

As before,  we keep the~same notations for the~transformed $L_u$ and $P^1_u$, the~torus invariant divisors and curves on $W_2$.
It follows from \eqref{eq:intersection} that
$$
\aligned
&T_0T_4T_6=T_0T_1T_4=\frac{1}{3}, \\
&T_0T_1T_5=\frac{1}{2}, \\
& T_0T_2T_5=T_0T_2T_3=T_0T_3T_6=T_1T_4T_5=T_2T_3T_6=T_4T_5T_6=T_2T_5T_6=1,\\
\endaligned
$$
and all other $T_iT_jT_k=0$ with distinct $i, j, k$. We have
$$
P_u^1=(7-u)T_0+T_1+T_2+\frac{8-u}{3}T_3,
$$
and we compute
\begin{align*}
P^1_uT_0T_1&=\frac{6-u}{6}, & P^1_uT_0T_2&=\frac{5-u}{3},   & P^1_uT_0T_3&=0,             & P^1_uT_0T_4&=\frac{1}{3},   & P^1_uT_0T_5&=\frac{u-4}{2}, \\
P^1_uT_0T_6&=\frac{1}{3},   & P^1_uT_1T_4&=\frac{6-u}{3},   & P^1_uT_1T_5&=\frac{6-u}{2}, & P^1_uT_2T_3&=0,             & P^1_uT_2T_5&=5-u, \\
P^1_uT_2T_6&=\frac{5-u}{3}, & P^1_uT_3T_6&=0,               & P^1_uT_4T_5&=1,             & P^1_uT_4T_6&=\frac{7-u}{3}, & P^1_uT_5T_6&=1.
\end{align*}
Hence, if $u\in [4,5]$, then $P^1_u$ is nef on $W_2$, so $L_u=P^1_u+\frac{u-2}{3}T_3$ is the~required Zariski decomposition.

Observe that $T_2$ is the~unique torus invariant surface that contains the~curves $T_0T_2$, $T_2T_5$, $T_2T_6$,
and $T_0T_2$ is nef on $T_2$, since $(T_0|_{T_2})^2=T_0^2T_2=0$.
For non-negative real numbers $a$ and $b$, we have
$$
\aligned
&(P_u^1-aT_2-bT_3)T_0T_2=\frac{5-u}{3}+a-b,\\
&(P_u^1-aT_2-bT_3)T_0T_3=-a+3b.
\endaligned
$$
These intersections are non-negative for $a\geqslant \frac{u-5}{2}$ and $b\geqslant \frac{u-5}{6}$.
Therefore, the~Nakayama-Zariski decomposition $L_u=P(u)+N(u)$ on $W_2$ satisfies
$$
N(u)\geqslant \frac{u-2}{3}T_3+\left(\frac{u-5}{2}T_2+\frac{u-5}{6}T_3\right)=\frac{u-5}{2}T_2+\frac{u-3}{2}T_3,
$$
where $P(u)$ stands for the~positive part, and $N(u)$ stands for the~negative part. Put
$$
P^2_u=P^1_u-\left(\frac{u-5}{2}T_2+\frac{u-5}{6}T_3\right).
$$
Then
\begin{align*}
P^2_uT_0T_1&=\frac{6-u}{6}, & P^2_uT_0T_2&=0,             & P^2_uT_0T_3&=0,             & P^2_uT_0T_4&=\frac{1}{3},   & P^2_uT_0T_5&=\frac{1}{2}, \\
P^2_uT_0T_6&=\frac{7-u}{6}, & P^2_uT_1T_4&=\frac{6-u}{3}, & P^2_uT_1T_5&=\frac{6-u}{2}, & P^2_uT_2T_3&=0,             & P^2_uT_2T_5&=0, \\
P^2_uT_2T_6&=0,             & P^2_uT_3T_6&=0,             & P^2_uT_4T_5&=1,             & P^2_uT_4T_6&=\frac{7-u}{3}, & P^2_uT_5T_6&=\frac{7-u}{2}.
\end{align*}
Hence, the~divisor $P_u^2$ is nef for $u\in[5,6]$, which implies that $P(u)=P_u^2$ and
$$
N(u)=\frac{u-5}{2}T_2+\frac{u-3}{2}T_3.
$$
This gives the~Zariski decomposition of the~divisor $L_u$ on the~3-fold $W_2$ for $u\in[5,6]$.

The surface $T_1$ is the~unique torus invariant surface that contains the~curves $T_0T_1$, $T_1T_4$, $T_1T_5$,
it has Picard rank $1$, and it is disjoint from $T_2$ and $T_3$. But
$$
(P^2_u-aT_1)T_0T_1=\frac{6-u}{6}+\frac{a}{6}.
$$
Therefore, the~Nakayama-Zariski decomposition $L_u=P(u)+N(u)$ on $W_2$ for $u>6$ satisfies
$$
N(u)\geqslant (u-6)T_1+\frac{u-5}{2}T_2+\frac{u-3}{2}T_3,
$$
where $P(u)$ is the~positive part, and $N(u)$ is the~negative part. Set $P_u^3=P^2_u-(u-6)T_1$. Then
\begin{align*}
P^3_uT_0T_1&=0,             & P^3_uT_0T_2&=0,  & P^3_uT_0T_3&=0,   & P^3_uT_0T_4&=\frac{7-u}{3}, & P^3_uT_0T_5&=\frac{7-u}{2}, \\
P^3_uT_0T_6&=\frac{7-u}{6}, & P^3_uT_1T_4&=0,  & P^3_uT_1T_5&=0,   & P^3_uT_2T_3&=0,             & P^3_uT_2T_5&=0, \\
P^3_uT_2T_6&=0,             & P^3_uT_3T_6&=0,  & P^3_uT_4T_5&=7-u, & P^3_uT_4T_6&=\frac{7-u}{3}, & P^3_uT_5T_6&=\frac{7-u}{2}.
\end{align*}
Then $P(u)=P^3_u$ is the~positive part of the~Zariski decomposition of $L_u$ on $W_2$ for $u\in[6,7]$,
and the~negative part is
$$
N(u)=(u-6)T_1+\frac{u-5}{2}T_2+\frac{u-3}{2}T_3.
$$
If $u>7$, then $L_u$ is not pseudoeffective.

\begin{remark}
\label{remark:projective}
The toric varieties $W_0$, $W_1$, $W_2$ are projective.
Indeed, the~variety $W_0$ is obtained by taking a weighted blowup of a projective variety.
On $W_1$, the~transformed $L_\frac{3}{2}$ is an ample divisor.
On $W_2$, we can obtain an ample divisor from $P^1_{\frac{9}{2}}+\frac{1}{m} T_2$ by taking sufficiently  large  integer $m$.
\end{remark}

To apply \cite{AbbanZhuang,Book,Fujita2021},
we must consider a common partial resolution of the~3-folds $W_0$, $W_1$, $W_2$.
Namely, let $\widetilde{W}$ be the~toric 3-fold defined by the~simplicial fan $\Sigma_{\widetilde{W}}$ in $\mathbb{R}^3$ given by
\begin{itemize}
\item the~list of primitive generators of rays of $\Sigma_{\widetilde{W}}$ is
\begin{align*}
v_0&=(1,3,-1), & v_1&=(1,0,0),  & v_2&=(0,0,1), & v_3&=(0,1,0), & v_4&=(0,0,-1), & v_5&=(0,-1,1), \\
v_6&=(-1,0,0), & v_7&=(0,3,-1), & v_8&=(1,3,0), & v_9&=(1,2,0), & v_{10}&=(1,0,2);
\end{align*}

\item the~list of maximal cones of $\Sigma_{\widetilde{W}}$ is
\begin{align*}
[0, 1,4]&, & [0,1,9]&, & [0,3,7]&, & [0,3,8]&, & [0,4,7]&, & [0,8,9]&, & [1,4,5]&, & [1,5,10]&, & [1,9,10], \\
[2,3,6]&, & [2,3,8]&, & [2,5,6]&, & [2,5,10]&, & [2,8,10]&, & [3,6,7]&, & [4,5,6]&, & [4,6,7]&, & [8,9,10].
\end{align*}
\end{itemize}
The fan $\Sigma_{\widetilde{W}}$ can be diagramed as follows:

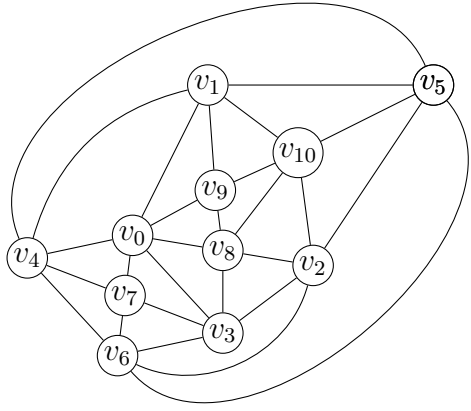
\begin{center}
\begin{tikzpicture}[node distance=1cm,auto]
\node[state, inner sep=1pt,minimum size=0pt] (q_1) at (1,2) {$v_1$};
\node[state,inner sep=1pt,minimum size=0pt] (q_2) at (2.4,-.4) {$v_2$};
\node[state,inner sep=1pt,minimum size=0pt] (q_3) at (1.2,-1.3) {$v_3$};
\node[state,inner sep=1pt,minimum size=0pt] (q_4) at (-1.4,-.3) {$v_4$};
\node[state,inner sep=1pt,minimum size=0pt] (q_5) at (4,2) {$v_5$};
\node[state,inner sep=1pt,minimum size=0pt] (q_6) at (-.2,-1.6) {$v_6$};
\node[state,inner sep=1pt,minimum size=0pt] (q_7) at (-.1,-.8) {$v_{7}$};
\node[state,inner sep=1pt,minimum size=0pt] (q_8) at (1.2,-.2) {$v_{8}$};
\node[state, inner sep=1pt,minimum size=0pt] (q_0) {$v_0$};
\node[state,inner sep=1pt,minimum size=0pt] (q_9) at (1.1,.6) {$v_{9}$};
\node[state,inner sep=1pt,minimum size=0pt] (q_10) at (2.2,1.1) {$v_{10}$};
\node[state,inner sep=1pt,minimum size=0pt] (q_5) at (4,2) {$v_5$};

\path[-] (q_0) edge node[swap] {} (q_1);
\path[-] (q_0) edge node[swap] {} (q_3);
\path[-] (q_0) edge node[swap] {} (q_4);
\path[-] (q_1) edge node[swap] {} (q_10);
\path[-] (q_8) edge node[swap] {} (q_3);
\path[-] (q_1) edge node[swap] {} (q_5);
\path[-] (q_4) edge node[swap] {} (q_7);
\path[-] (q_3) edge node[swap] {} (q_7);
\path[-] (q_0) edge node[swap] {} (q_7);
\path[-] (q_6) edge node[swap] {} (q_7);
\path[-] (q_4) edge node[swap] {} (q_6);
\path[-] (q_3) edge node[swap] {} (q_6);
\path[-] (q_2) edge node[swap] {} (q_3);
\path[-] (q_2) edge node[swap] {} (q_5);
\path[-] (q_0) edge node[swap] {} (q_8);
\path[-] (q_2) edge node[swap] {} (q_8);

\path[-] (q_10) edge node[swap] {} (q_9);
\path[-] (q_0) edge node[swap] {} (q_9);
\path[-] (q_1) edge node[swap] {} (q_9);
\path[-] (q_8) edge node[swap] {} (q_9);

\path[-] (q_8) edge node[swap] {} (q_10);
\path[-] (q_2) edge node[swap] {} (q_10);
\path[-] (q_5) edge node[swap] {} (q_10);

\path[-, bend right = 30] (q_1) edge node {} (q_4);
\path[-, bend left = 90] (q_4) edge node {} (q_5);
\path[-, bend left = 50] (q_2) edge node {} (q_6);
\path[-, bend left = 90] (q_5) edge node {} (q_6);
\end{tikzpicture}
\end{center}
Then there exists the~following commutative diagram:
$$
\xymatrix{
&&\widetilde{W}\ar@{->}[dll]_{\zeta_0}\ar@{->}[rrd]^{\zeta_2}\ar@{->}[d]_{\zeta_1}&&\\%
W_0\ar@{-->}[rr]\ar@{->}[d]_{\lambda}&&W_1\ar@{-->}[rr]&&W_2\\
Y&&&&}
$$
where $\zeta_0$, $\zeta_1$ and $\zeta_2$ are toric birational morphisms.

Let us denote by $\widetilde{T}_i$ the~torus invariant divisor on $\widetilde{W}$ corresponding to the~ray $v_i$ in the~fan $\Sigma_{\widetilde{W}}$.
Then the~formula \eqref{eq:intersection} implies that
\begin{equation}
\label{eq:intersection4}
\aligned
&\widetilde{T}_1\widetilde{T}_9\widetilde{T}_{10}=\frac{1}{4},  \\
& \widetilde{T}_0\widetilde{T}_1\widetilde{T}_4=\widetilde{T}_0\widetilde{T}_4\widetilde{T}_7=\widetilde{T}_2\widetilde{T}_8\widetilde{T}_{10}=\widetilde{T}_4\widetilde{T}_6\widetilde{T}_7=\frac{1}{3}, \\
& \widetilde{T}_0\widetilde{T}_1\widetilde{T}_9=\widetilde{T}_1\widetilde{T}_5\widetilde{T}_{10}=\widetilde{T}_8\widetilde{T}_9\widetilde{T}_{10}=\frac{1}{2}, \\
& \widetilde{T}_0\widetilde{T}_3\widetilde{T}_7=\widetilde{T}_0\widetilde{T}_3\widetilde{T}_8=\widetilde{T}_0\widetilde{T}_8\widetilde{T}_9=\widetilde{T}_1\widetilde{T}_4\widetilde{T}_5=\widetilde{T}_2\widetilde{T}_3\widetilde{T}_6=1,\\
&\widetilde{T}_2\widetilde{T}_3\widetilde{T}_8=\widetilde{T}_2\widetilde{T}_5\widetilde{T}_6=\widetilde{T}_2\widetilde{T}_5\widetilde{T}_{10}=
\widetilde{T}_3\widetilde{T}_6\widetilde{T}_7=\widetilde{T}_4\widetilde{T}_5\widetilde{T}_6=1,\\
\endaligned
\end{equation}
and other $\widetilde{T}_i\widetilde{T}_j\widetilde{T}_k$ with distinct indices $i, j, k$ are $0$.
Further, the~characters $\chi_1$, $\chi_2$, $\chi_3$ corresponding to the~lattice points $(1,0,0)$, $(0,1,0)$, $(0,0,1)$ in the~dual lattice
yield the~following relations:
\begin{equation}
\label{eq:principal2}
\aligned
&0\sim \mathrm{div}(\chi_1)=\widetilde{T}_0+\widetilde{T}_1-\widetilde{T}_6+\widetilde{T}_8+\widetilde{T}_9+\widetilde{T}_{10}, \\
&0\sim \mathrm{div}(\chi_2)=3\widetilde{T}_0+\widetilde{T}_3-\widetilde{T}_5+3\widetilde{T}_7+3\widetilde{T}_8+2\widetilde{T}_{9}, \\
&0\sim \mathrm{div}(\chi_3)=-\widetilde{T}_0+\widetilde{T}_2-\widetilde{T}_4+\widetilde{T}_5-\widetilde{T}_7+2\widetilde{T}_{10}.
\endaligned
\end{equation}
Moreover, we have
\begin{equation*}
\begin{array}{ll}
\zeta^*_0(T_0)=\widetilde{T}_0, &  \zeta^*_0(T_1)=\widetilde{T}_1+\widetilde{T}_8+\widetilde{T}_9+\widetilde{T}_{10},\\
\zeta^*_0(T_2)=\widetilde{T}_2+2\widetilde{T}_{10}, & \zeta^*_0(T_3)=\widetilde{T}_3+3\widetilde{T}_7+3\widetilde{T}_8+2\widetilde{T}_{9},\\
\zeta^*_1(T_0)=\widetilde{T}_0+\widetilde{T}_7+\widetilde{T}_8+\frac{2}{3}\widetilde{T}_{9}, &  \zeta^*_1(T_1)=\widetilde{T}_1+\frac{1}{3}\widetilde{T}_9+\widetilde{T}_{10},\\
\zeta^*_1(T_2)=\widetilde{T}_2+\widetilde{T}_8+\frac{2}{3}\widetilde{T}_{9}+2\widetilde{T}_{10}, & \zeta^*_1(T_3)=\widetilde{T}_3,\\
\zeta^*_2(T_0)=\widetilde{T}_0+\widetilde{T}_7+\widetilde{T}_8+\widetilde{T}_9+\widetilde{T}_{10}, &  \zeta^*_2(T_1)=\widetilde{T}_1,\\
\zeta^*_2(T_2)=\widetilde{T}_2+\widetilde{T}_{8}, & \zeta^*_2(T_3)=\widetilde{T}_3.
\end{array}
\end{equation*}
Let us briefly explain how we get these expressions.
For instance, the~divisor $T_0$ on $W_0$ does not contain centers of $\zeta_0$-exceptional surfaces, so $\zeta^*_0(T_0)=\widetilde{T}_0$.
Similarly, the~divisor $T_0$ on $W_2$ contains centers of the~following $\zeta_2$-exceptional divisors:
$\widetilde{T}_7$, $\widetilde{T}_8$, $\widetilde{T}_9$, $\widetilde{T}_{10}$,
which implies that
$$
\zeta_2^*(T_0)=\widetilde{T}_0+a_7\widetilde{T}_7+a_8\widetilde{T}_8+a_9\widetilde{T}_9+a_{10}\widetilde{T}_{10}
$$
for some positive rational numbers $a_7$, $a_8$, $a_9$, $a_{10}$.
Then we obtain
$$
\aligned
&0=\left(\widetilde{T}_0+a_7\widetilde{T}_7+a_8\widetilde{T}_8+a_9\widetilde{T}_9+a_{10}\widetilde{T}_{10}\right)\widetilde{T}_3\widetilde{T}_7 =1-a_7,\\
&0=\left(\widetilde{T}_0+a_7\widetilde{T}_7+a_8\widetilde{T}_8+a_9\widetilde{T}_9+a_{10}\widetilde{T}_{10}\right)\widetilde{T}_3\widetilde{T}_8 =1-a_8,\\
&0=\left(\widetilde{T}_0+a_7\widetilde{T}_7+a_8\widetilde{T}_8+a_9\widetilde{T}_9+a_{10}\widetilde{T}_{10}\right)\widetilde{T}_8\widetilde{T}_{10} =-\frac{1}{3}a_8+\frac{1}{2}a_9-\frac{1}{6}a_{10},\\
&0=\left(\widetilde{T}_0+a_7\widetilde{T}_7+a_8\widetilde{T}_8+a_9\widetilde{T}_9+a_{10}\widetilde{T}_{10}\right)\widetilde{T}_1\widetilde{T}_{10}=\frac{1}{4}a_8-\frac{1}{4}a_{10},
\endaligned
$$
which gives $a_7=a_8=a_9=a_{10}=1$.
Here, all intersections are derived from \eqref{eq:intersection4} and \eqref{eq:principal2}.

For every $u\in[0,7]$, the~Zariski decomposition of the~divisor $\zeta_0^*(L_u)$ exists on the~3-fold~$\widetilde{W}$.
Let~$P_{\widetilde{W}}(u)$ and $N_{\widetilde{W}}(u)$ be its positive and negative parts, respectively.
Then their expressions as linear combinations of the~torus invariant divisors on $\widetilde{W}$ are given in Table~\ref{table:ZDu}.

We now consider the~toric surface $\widetilde{T}_0$.
Its fan is the~image of the~fan $\Sigma_{\widetilde{W}}$ under the~quotient lattice homomorphism $\mathbb{Z}^3\to\mathbb{Z}^3/\mathbb{Z}v_0\cong \mathbb{Z}^2$.
We may assume that $v_1\mapsto w_1=(1,0)$ and $v_3\mapsto w_4=(0,1)$, which determines the~quotient homomorphism.
Then the~list of primitive generators of the~rays in the~fan consists of
$$
w_1=(1,0), w_2=(1,2), w_3=(1,3), w_4=(0,1), w_5=(-1,0), w_6=(-1,-3).
$$

Let $\zeta$ be the~restriction morphism $\zeta_0|_{\widetilde{T}_0}\colon \widetilde{T}_0\to T_0$.
Then $\zeta$ contracts the~torus invariant curves defined by $w_5$, $w_3$, $w_2$,
since $\zeta_0$ contracts $\lfloor\widetilde{T}_0\widetilde{T}_7\rfloor$,~$\lfloor\widetilde{T}_0\widetilde{T}_8\rfloor$,~$\lfloor\widetilde{T}_0\widetilde{T}_9\rfloor$.
This can be illustrated as follows.
 
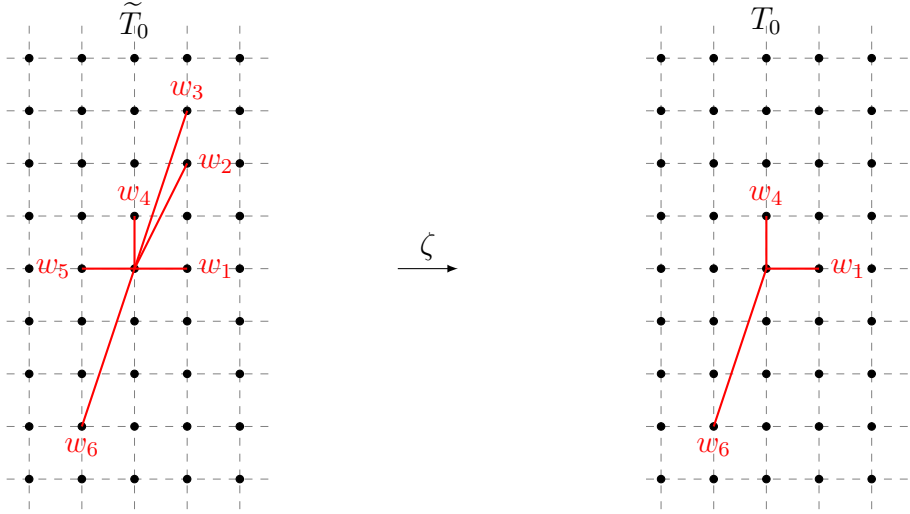
\begin{center}
\begin{tikzpicture}[>=latex]
    \begin{scope}
     \draw[style=help lines,dashed] (2.5,1) grid[step=.7cm] (6,7.5); 

    \foreach \x in {4,5,...,8}{                           
        \foreach \y in {2,3,...,10}{                       
        \node[draw,circle,inner sep=1pt,fill] at (.7*\x,.7*\y) {}; 
 }
}
                  \draw[thick, red,-] (4.2,4.2) -- (4.9,6.3)  node [above ] {$w_3$};
         \draw[thick, red,-] (4.2,4.2) -- (4.2,4.9)  node [above ] {$w_4$};
\draw[thick, red,-] (4.2,4.2) -- (4.9,5.6)  node [right] {$w_2$};
\draw[thick, red,-] (4.2,4.2) -- (4.9,4.2)  node [right] {$w_1$};
\draw[thick, red,-] (4.2,4.2) -- (3.5,2.1) node [below] {$w_6$};
\draw[thick, red,-] (4.2,4.2) -- (3.5,4.2) node [left] {$w_5$};
\draw[style=help lines,dashed] (11,1) grid[step=.7cm] (14.5,7.5);
\foreach \x in {16,17,...,20}{                           
        \foreach \y in {2,3,...,10}{                       
        \node[draw,circle,inner sep=1pt,fill] at (.7*\x,.7*\y) {}; 

        }
    }
 \draw[thick, red,-] (12.6,4.2) -- (12.6,4.9)  node [above ] {$w_4$};
\draw[thick, red,-] (12.6,4.2) -- (13.3,4.2)  node [right] {$w_1$};
\draw[thick, red,-] (12.6,4.2) -- (11.9,2.1) node [below] {$w_6$};
\draw[->] (7.7,4.2) -- (8.5,4.2) node [above left] {};
        \node (O) at (8.1,4.5) {$\zeta$};
        \node (O) at (4.2,7.5) {$\widetilde{T}_0$};
           \node (O) at (12.6,7.5) {$T_0$};
    \end{scope}
\end{tikzpicture}
\end{center}

Let $\alpha_1,\ldots,\alpha_6$ be the~torus invariant curves in $\widetilde{T}_0$ defined by the~rays $w_1,\ldots,w_6$, respectively.
Set $\overline{\alpha}_1=\zeta(\alpha_1)$, $\overline{\alpha}_4=\zeta(\alpha_4)$, $\overline{\alpha}_6=\zeta(\alpha_6)$.
Note that $\alpha_1+\alpha_2+\alpha_3=\alpha_5+\alpha_6$ and $2\alpha_2+3\alpha_3+\alpha_4=3\alpha_6$.
With these relations, \cite[\S 6.4]{CLS} yields the~following intersection matrix:
\begin{equation*}
A:=\left(\alpha_i\alpha_j\right)=
\begin{pmatrix}[1.4]
  -\frac{1}{6} & \frac{1}{2}& 0 &0 &0 &\frac{1}{3}\\
   \frac{1}{2} & -\frac{3}{2}& 1&0 &0& 0\\
    0 & 1& -1 &1&0&0\\
    0& 0& 1 &-3&1&0\\
     0 & 0& 0& 1 &-\frac{1}{3}&\frac{1}{3}\\
     \frac{1}{3} & 0&0&0& \frac{1}{3} &0\\
\end{pmatrix}.
\end{equation*}

It follows from \cite[Lemma~12.5.2]{CLS}  that
$$
\widetilde{T}_1\big|_{\widetilde{T}_0}=\alpha_1, \widetilde{T}_3\big|_{\widetilde{T}_0}=\alpha_4, \widetilde{T}_4\big|_{\widetilde{T}_0}=\alpha_6,
\widetilde{T}_7\big|_{\widetilde{T}_0}=\alpha_5,  \widetilde{T}_8\big|_{\widetilde{T}_0}=\alpha_3,  \widetilde{T}_9\big|_{\widetilde{T}_0}=\alpha_2.
$$
Moreover, \eqref{eq:principal2} implies
$$
\widetilde{T}_0\big|_{\widetilde{T}_0}=-\big(\widetilde{T}_1-\widetilde{T}_6+\widetilde{T}_8+\widetilde{T}_9+\widetilde{T}_{10}\big)\big|_{\widetilde{T}_0}=-\big(\alpha_1+\alpha_2+\alpha_3\big).
$$

Set $\widetilde{P}(u)=P_{\widetilde{W}}(u)|_{\widetilde{T}_0}$ and $\widetilde{N}(u)=N_{\widetilde{W}}(u)|_{\widetilde{T}_0}$.
Then we can express $\widetilde{P}(u)$ and $\widetilde{N}(u)$
as linear combinations of the~curves $\alpha_1$, $\alpha_2$, $\alpha_3$, $\alpha_4$, $\alpha_5$, $\alpha_6$.
These expressions are presented in Table~\ref{table:widetildePNu}.

We are ready to apply \cite{AbbanZhuang,Book,Fujita2021} to estimate $\delta_P(Y,\Delta_Y)$ from below.
Let $Q$ be a point in~$G=T_0$, let $C$ be a smooth curve in $G$ such that $Q\in C\not\subset\Delta_G$,
and let $\widetilde{C}$ be its proper transform on $\widetilde{T}_0$. Then $\zeta$ induces an isomorphism $\widetilde{C}\cong C$.
For every $u\in[0,7]$, let
$$
t(u)=\sup\Big\{v\in \mathbb R_{\geqslant 0} \ \big|\ \text{$\widetilde{P}(u)-vC$ is pseudoeffective}\Big\}.
$$
For every $v\in[0,t(u)]$, let $P(u,v)$ be the~positive part of the~Zariski decomposition of $\widetilde{P}(u)-vC$,
and let $N(u,v)$ be its negative part. Set
$$
S_L\big(W^{G}_{\bullet,\bullet};C\big)=\frac{3}{L^3}\int\limits_0^{7}\big(\widetilde{P}(u)\big)^2\mathrm{ord}_C\big(\widetilde{N}(u)\big)du+\frac{3}{L^3}\int\limits_0^{7}\int\limits_0^{t(u)}\big(P(u,v)\big)^2dvdu.
$$
Now, we write $\zeta^*(C)=\widetilde{C}+\Sigma$ for an effective $\mathbb{R}$-divisor $\Sigma$ on the~surface $\widetilde{T}_0$.
For every $u\in[0,7]$, write $\widetilde{N}(u)=d(u)C+N^\prime(u)$,
where $d(u)=\mathrm{ord}_C(\widetilde{N}(u))$, and $N^\prime(u)$ is an effective $\mathbb{R}$-divisor~on~$\widetilde{T}_0$.
Now, as in \cite[Definition~4.16]{Fujita2019b}, we set
$$
F_Q\big(W_{\bullet, \bullet,\bullet}^{G,C}\big)=\frac{6}{L^3}\int\limits_0^{7}\int\limits_0^{t(u)}\big(P(u,v)\cdot\widetilde{C}\big)\cdot\mathrm{ord}_Q\Big(\big(N^\prime(u)+N(u,v)-(v+d(u))\Sigma\big)\big|_{\widetilde{C}}\Big)dvdu,
$$
where we consider $Q$ as a point in $\widetilde{C}$ using the~isomorphism $\widetilde{C}\cong C$ induced by $\zeta$.
Finally, we set
$$
S\big(W_{\bullet, \bullet,\bullet}^{G,C};Q\big)=\frac{3}{L^3}
\int\limits_0^{7}\int\limits_0^{t(u)}\big(P(u,v)\cdot\widetilde{C}\big)^2dvdu+F_Q\big(W_{\bullet, \bullet,\bullet}^{G,C}\big).
$$

We have $(K_{G}+C+\Delta_G)\vert_{C}\sim_{\mathbb{R}}K_C+\Delta_C$ for an~effective divisor $\Delta_C$ known as the~different \cite{Prokhorov},
which can be computed locally near any point in $C$.
Using \cite[Corollary 4.18]{Fujita2021}, we obtain
$$
\delta_P(Y,\Delta_Y)\geqslant\min\left\{
\frac{A_{Y,\Delta_Y}(G)}{S_{L}(G)},
\inf_{Q\in G}
\min\left\{\frac{A_{G,\Delta_G}(C)}{S_L(W^{G}_{\bullet,\bullet};C)},
\frac{A_{C,\Delta_C}(Q)}{S(W_{\bullet,\bullet,\bullet}^{G,C};Q)}\right\}\right\},
$$
where $A_{G,\Delta_G}(C)=1$, because  $C\not\subset\Delta_G$ by assumption.
On the~other hand, we assumed that there exists a prime divisor $\mathbf{F}$ over $Y$ such that $\beta_{Y,\Delta_Y}(\mathbf{F})\leqslant 0$.
Moreover, we proved that $C_Y(\mathbf{F})=P$, so
$$
1\geqslant\frac{A_{Y,\Delta_Y}(\mathbf{F})}{S_{L}(\mathbf{F})}\geqslant
\delta_P(Y,\Delta_Y)\geqslant\min\left\{
\frac{A_{Y,\Delta_Y}(G)}{S_{L}(G)},
\inf_{Q\in G}
\min\left\{\frac{A_{G,\Delta_G}(C)}{S_L(W^{G}_{\bullet,\bullet};C)},
\frac{A_{C,\Delta_C}(Q)}{S(W_{\bullet,\bullet,\bullet}^{G,C};Q)}\right\}\right\}.
$$
Therefore, since $\frac{A_{Y,\Delta_Y}(G)}{S_{L}(G)}=\frac{63}{59}$,
it follows from \cite[Corollary 4.18]{Fujita2021} and \cite[Theorem 3.3]{AbbanZhuang} that
$$
\inf_{Q\in G}
\min\left\{\frac{A_{G,\Delta_G}(C)}{S_L(W^{G}_{\bullet,\bullet};C)},
\frac{A_{C,\Delta_C}(Q)}{S(W_{\bullet,\bullet,\bullet}^{G,C};Q)}\right\}<1.
$$
Therefore, to exclude the~case ($\mathbb{D}_4$),
it is enough to show that for every point $Q\in G$, there exists a smooth irreducible curve $C\subset G$ such that $Q\in C\not\subset\Delta_G$
and
\begin{equation}
\label{equation:3-4-D4-case}
S_L\big(W^{G}_{\bullet,\bullet};C\big)\leqslant 1\leqslant\frac{A_{C,\Delta_C}(Q)}{S(W_{\bullet,\bullet,\bullet}^{G,C};Q)}
\end{equation}
This is what we will do in the~rest of this section.

Let $Q$ be a~point in $G=T_0\cong\mathbb{P}(1,3,1)$.
Recall that $\overline{\alpha}_1$, $\overline{\alpha}_4$, $\overline{\alpha}_6$ are all torus invariant curves~in~$G$.
Let $Q_{14}=\overline{\alpha}_1\cap \overline{\alpha}_4$, $Q_{16}=\overline{\alpha}_1\cap \overline{\alpha}_6$,
$Q_{46}=\overline{\alpha}_4\cap \overline{\alpha}_6$, where $Q_{16}$ is the~singular point of the~surface~$G$.
Recall that $R_G$ meets the~curve $\overline{\alpha}_4$ transversally at three distinct points including $Q_{14}$~and~$Q_{46}$.
Let us denote by $Q_4$ the~point in $R_G\cap\overline{\alpha}_4$ that is different from $Q_{14}$ and $Q_{46}$.

Now, let us choose the~curve $C$.
If $Q\in\overline{\alpha}_1\cup\overline{\alpha}_4\cup\overline{\alpha}_6$, we choose $C$ as follows:
\begin{itemize}
\item if $Q\in\overline{\alpha}_1$, $Q\ne Q_{14}$, $Q\ne Q_{16}$, we let $C=\overline{\alpha}_1$,
\item if $Q\in\overline{\alpha}_4$, $Q\ne Q_{14}$, $Q\ne Q_{46}$, we let $C=\overline{\alpha}_4$,
\item if $Q\in\overline{\alpha}_6$, $Q\ne Q_{16}$, $Q\ne Q_{46}$, we let $C=\overline{\alpha}_6$,
\item if $Q=Q_{14}$, we let $C=\overline{\alpha}_1$ or $C=\overline{\alpha}_4$,
\item if $Q=Q_{16}$, we let $C=\overline{\alpha}_1$ or $C=\overline{\alpha}_6$,
\item if $Q=Q_{46}$, we let $C=\overline{\alpha}_4$.
\end{itemize}
Similarly, if $Q\not\in\overline{\alpha}_1\cup\overline{\alpha}_4\cup\overline{\alpha}_6$,
there exists a unique curve $\overline{\alpha}_0\in|\mathcal{O}_G(1)|$ such that $\overline{\alpha}_0$ contains~$Q$.
In~this case, we let $C=\overline{\alpha}_0$, and we let $\alpha_0$ be the~proper transform of the~curve $\overline{\alpha}_0$ on the~surface~$\widetilde{T}_0$.
Then the~divisor $\Sigma$ and the~different $\Delta_C$ can be described as follows:
\begin{itemize}
\item[($\overline{\alpha}_1$)] if $C=\overline{\alpha}_1$, then $\Sigma=\alpha_2+\alpha_3$ and $\Delta_C=\frac{2}{3}Q_{16}+\frac{1}{2}Q_{14}$,
\item[($\overline{\alpha}_4$)] if $C=\overline{\alpha}_4$, then $\Sigma=2\alpha_2+3\alpha_3+3\alpha_5$ and $\Delta_C=\frac{1}{2}Q_{14}+\frac{1}{2}Q_{46}+\frac{1}{2}Q_4$,
\item[($\overline{\alpha}_6$)] if $C=\overline{\alpha}_6$, then $\Sigma=\alpha_5$ and $\Delta_C=\frac{1}{2}Q_{46}+\frac{2}{3}Q_{16}$,
\item[($\overline{\alpha}_0$)] if $C=\overline{\alpha}_0$, then $\Sigma=0$ and $\Delta_C=\Delta_{G}\vert_{C}+\frac{2}{3}Q_{16}$.
\end{itemize}
In the~last case,  we have $\mathrm{ord}_Q(\Delta_C)\leqslant\frac{1}{2}$, because the~curves $\overline{\alpha}_0$ and $R_G$ meet transversally.

In each possible case, we compute $t(u)$ as follows in Table~\ref{table:t-u-1}.

For each $u\in[0,7]$ and $v\in[0,t(u)]$, we can express
the~divisors $P(u,v)$ and $N(u,v)$ as linear combinations of the~curves
$\alpha_1$, $\alpha_2$, $\alpha_3$, $\alpha_4$, $\alpha_5$, $\alpha_6$.
These expressions are listed in Tables~\ref{table:ZDuv1},~\ref{table:ZDuv4},~\ref{table:ZDuv6},~\ref{table:ZDuv0}.

We now regard the~divisor $P(u,v)$  as a row vector $\mathbf{p}(u,v)\in\mathbb{R}^6$ defined as
$$
\mathbf{p}(u,v)=\big(c_{1}(u,v), c_{2}(u,v), c_{3}(u,v), c_{4}(u,v), c_{5}(u,v), c_{6}(u,v)\big),
$$
where $P(u,v)=c_{1}(u,v)\alpha_1+c_{2}(u,v)\alpha_2+c_{3}(u,v)\alpha_3+c_{4}(u,v)\alpha_4+c_{5}(u,v)\alpha_5+c_{6}(u,v)\alpha_6$. Then
$$
\big(P(u,v)\big)^2=\mathbf{p}(u,v)A\mathbf{p}(u,v)^T.
$$
Thus, we have
$$
S_L\big(W^{G}_{\bullet,\bullet};C\big)=\frac{3}{9}\int\limits_{0}^7\mathbf{p}(u,0)A\mathbf{p}(u,0)^T\cdot d(u)du+\frac{3}{9}\int\limits_{0}^7\int\limits_0^{t(u)}\mathbf{p}(u,v)A\mathbf{p}(u,v)^Tdvdu.
$$
Now, integrating we get
$$
S_L\big(W^{G}_{\bullet,\bullet};C\big)=
\left\{\aligned
&\frac{1}{2} \ \text{ if } C=\overline{\alpha}_1, \\
&\frac{7}{9} \ \text{ if } C=\overline{\alpha}_4, \\
&\frac{4}{9} \ \text{ if } C=\overline{\alpha}_6, \\
&\frac{11}{36} \ \text{ if } C=\overline{\alpha}_0.
\endaligned
\right.
$$
In each case, we have $S_L(W^{G}_{\bullet,\bullet};C)<1$ as required for \eqref{equation:3-4-D4-case}.

To present a formula for $S_L(W^{G,C}_{\bullet,\bullet,\bullet};Q)$,
let $\boldsymbol{e}_1$, $\boldsymbol{e}_2$, $\boldsymbol{e}_3$, $\boldsymbol{e}_4$, $\boldsymbol{e}_5$, $\boldsymbol{e}_6$
be the~standard basis for $\mathbb{R}^6$, and let $\boldsymbol{e}_0=\boldsymbol{e}_1+\boldsymbol{e}_2+\boldsymbol{e}_3$.
If $C=\overline{\alpha}_i$ for $i\in\{1,4,6,0\}$, then
$$
S_L\big(W^{G,C}_{\bullet,\bullet,\bullet};Q\big)=
\frac{3}{9}\int\limits_{0}^{7}\int\limits_0^{t(u)}\Big(\mathbf{p}(u,v)A\boldsymbol{e}_i^T\Big)^2dvdu+F_Q\big(W_{\bullet, \bullet,\bullet}^{G,C}\big),
$$
where
$$
F_Q\big(W_{\bullet, \bullet,\bullet}^{G,C}\big)=\frac{6}{9}\int\limits_{0}^7\int\limits_0^{t(u)}\Big(\mathbf{p}(u,v)A\boldsymbol{e}_i^T\Big)\cdot\mathrm{ord}_{Q}\Big(\big(N^\prime(u)+N(u,v)-(v+d(u))\Sigma\big)\big|_{\widetilde{C}}\Big)dvdu.
$$
In particular, if $Q\not\in\overline{\alpha}_1\cup\overline{\alpha}_4\cup\overline{\alpha}_6$,
then $C=\overline{\alpha}_0$, so that
$$
S_L\big(W^{G,C}_{\bullet,\bullet,\bullet};Q\big)=\frac{3}{9}\int\limits_{0}^7\int\limits_0^{t(u)}
\Big(\mathbf{p}(u,v)A\boldsymbol{e}_0^T\Big)^2dvdu=\frac{5}{24}<\frac{1}{2}\leqslant 1-\mathrm{ord}_{Q}(\Delta_{C})=A_{C,\Delta_{C}}(Q),
$$
which gives \eqref{equation:3-4-D4-case}. Similarly, if $Q\in\overline{\alpha}_1$ and $C=\overline{\alpha}_1$, then
$$
S_L\big(W^{G,C}_{\bullet,\bullet,\bullet};Q\big)=\frac{3}{9}\int\limits_{0}^7\int\limits_0^{t(u)}\Big(\mathbf{p}(u,v)A\boldsymbol{e}_1^T\Big)^2dvdu+F_Q\big(W_{\bullet, \bullet,\bullet}^{G,C}\big)=\frac{4}{27}+F_Q\big(W_{\bullet, \bullet,\bullet}^{G,C}\big)=
\left\{\aligned
&\frac{83}{108}\ \text{ if }\ Q=Q_{14},\\
&\frac{4}{27}\ \text{ if }\ Q\ne Q_{14},
\endaligned\right.
$$
while $\Delta_{C}=\frac{2}{3}Q_{16}+\frac{1}{2}Q_{14}$.
This gives \eqref{equation:3-4-D4-case} for $Q\in\overline{\alpha}_1\setminus\{Q_{14}\}$.
If $Q\in\overline{\alpha}_6\setminus\{Q_{16}\}$ and $C=\overline{\alpha}_6$,~then
$$
S_L\big(W^{G,C}_{\bullet,\bullet,\bullet};Q\big)=\frac{3}{9}\int\limits_{0}^7\int\limits_0^{t(u)}
\Big(\mathbf{p}(u,v)A\boldsymbol{e}_6^T\Big)^2dvdu+F_Q\big(W_{\bullet, \bullet,\bullet}^{G,C}\big)
=\left\{\aligned
&\frac{126}{162}\ \text{ if } Q=Q_{46},\\
&\frac{25}{162}\ \text{ if } Q\ne Q_{46},
\endaligned \right.
$$
while $\Delta_{C}=\frac{1}{2}Q_{46}+\frac{2}{3}Q_{16}$, which gives \eqref{equation:3-4-D4-case} for $Q\in\overline{\alpha}_6\setminus\{Q_{46}, Q_{16}\}$.
If $Q\in\overline{\alpha}_4$ and $C=\overline{\alpha}_4$, then
$$
S_L\big(W^{G,C}_{\bullet,\bullet,\bullet};Q\big)=\frac{3}{9}\int\limits_{0}^7\int\limits_0^{t(u)}
\Big(\mathbf{p}(u,v)A\boldsymbol{e}_4^T\Big)^2dvdu+F_Q\big(W_{\bullet, \bullet,\bullet}^{G,C}\big)=
\left\{\aligned
&\frac{4}{9}\ \text{ if } Q=Q_{46},\\
&\frac{1}{2}\ \text{ if } Q=Q_{14},\\
&\frac{11}{36}\ \text{ if } Q\ne Q_{46}\ \text{ and }\ Q\ne Q_{14},
\endaligned \right.
$$
while $\Delta_{C}=\frac{1}{2}Q_{14}+\frac{1}{2}Q_{46}+\frac{1}{2}Q_4$.
This gives \eqref{equation:3-4-D4-case} for $Q\in\overline{\alpha}_4$.

Therefore, we see that \eqref{equation:3-4-D4-case} holds for every $Q\in G$ for an appropriate choice of the~curve $C$,
which excludes the~case ($\mathbb{D}_4$) as we explained earlier.

\subsection{Exclusion of the~case ($\mathbb{A}_3$)}
\label{subsection:3-4-step-3}

Let us finish the~proof of Theorem~\ref{theorem:3-4-K-stable}.
Now, we assume~that the~surface $R$ is given by the~equation \eqref{equation:R-simplified} with $a_2\ne 0$.
In the~chart $\mathbb{A}^3_{x,y,z}=\{x_0y_0z_0\ne 0\}$ with coordinates $x=\frac{x_1}{x_0}$, $y=\frac{y_1}{y_0}$, $z=\frac{z_1}{z_0}$,
we have $P=(0,0,0)$, and the~surface $R$ is given by
$$
y+xz^2+a_2x^2+\big(e_0yz+d_2x^2z+b_1xy+e_1xyz+c_0y^2+b_2x^2y+e_2x^2yz+c_1xy^2+c_2x^2y^2\big)=0,
$$
where $y+xz^2+a_2x^2$ is the~smallest degree term for the~weights $\mathrm{wt}(x)=2$, $\mathrm{wt}(y)=4$, $\mathrm{wt}(z)=1$.
Let $\lambda\colon W_0\to Y$ be the~corresponding weighted blow up of the~point $P$ with weights $(2,4,1)$,
and let $G$ be the~$\lambda$-exceptional surface. Then~$G\cong\mathbb{P}(1,2,1)$,
and we can also consider $(x,y,z)$ as global coordinates on $G$ with $\mathrm{wt}(x)=1$,
$\mathrm{wt}(y)=2$, $\mathrm{wt}(z)=1$.

Let  $R_{W_0}$, $F_{W_0}$ and $S_{W_0}$ be the~proper transforms on $W_0$ of the~surfaces $R$, $F$ and $S$, respectively.
Set $R_G=R_{W_0}\vert_{G}$, let $n_G$ be the~curve $\{z=0\}\subset G$, set $\Delta_G=\frac{1}{2}R_G+\frac{1}{2}n_G$ and $\Delta_{W_0}=\frac{1}{2}R_{W_0}$.
Then
$$
\big(K_{W_0}+\Delta_{W_0}+G\big)\big|_G\sim_{\mathbb{Q}}K_G+\Delta_G.
$$
Note that $F_{W_0}\vert_{G}=\{x=0\}$, $S_{W_0}\vert_{G}=\{y=0\}$ and $R_G=\{y+xz+a_2x^2=0\}$.

The remaining part of this subsection is very similar to what has been done in Section~\ref{subsection:3-4-step-2}, so we will omit some details here.
We have $A_{Y,\Delta_Y}(G)=5$. Using \cite[Corollary~7.7]{BlumJonsson}, we get $S_{Y,\Delta_Y}(G)=\frac{41}{9}$.

Both 3-folds $Y$ and $W_0$ are toric, and the~weighted blow up $\lambda$ is also toric.
Let $\Sigma_{Y}$ and $\Sigma_{W_0}$ be the~fans of the~3-folds $Y$ and $W_0$, respectively.
Then the~fan $\Sigma_{Y}$ is presented in Section~\ref{subsection:3-4-step-2},
and the~fan $\Sigma_{W_0}$ is the~simplicial fan in $\mathbb{R}^3$ defined by the~following data:
\begin{itemize}
\item the~list of primitive generators of rays in $\Sigma_{W_0}$ is
\begin{align*}
v_0&=(2,4,-1), & v_1&=(1,0,0), & v_2&=(0,0,1), & v_3&=(0,1,0), \\
v_4&=(0,0,-1), & v_5&=(0,-1,1), & v_6&=(-1,0,0);
\end{align*}
\item  the~list of maximal cones  in  $\Sigma_{W_0}$ is
$$
[0,1,3],   [0,1,4],   [0,3,4],  [1,2,3],   [1,2,5], [1,4,5],  [2,3,6],  [2,5,6],  [3,4,6],  [4,5,6],
$$
where $[i,j,k]$ is the~cone generated by the~rays $v_i$, $v_j$, and $v_k$.
\end{itemize}
As in Section~\ref{subsection:3-4-step-2}, let us denote by $T_i$ the~torus invariant divisor that corresponds to the~ray $v_i$.
Note that $T_0$ is the~exceptional divisor $G$.

Take $u\in\mathbb{R}_{\geqslant 0}$. As in Section~\ref{subsection:3-4-step-2}, we let $L_u=\lambda^*(L)-u T_0$. Then
$$
L_u\sim_{\mathbb{R}}(10-u)T_0+T_1+T_2+2T_3,
$$
which implies that $L_u$ is pseudoeffective if and only if $u\in[0,10]$.

Let $W_1$, $W_2$, $W_3$ be the~toric 3-folds defined by the~simplicial fans
$\Sigma_{W_1}$, $\Sigma_{W_2}$, $\Sigma_{W_3}$ in $\mathbb{R}^3$, respectively, which are determined by the~following data:
\begin{itemize}
\item the~list of primitive generators of rays of the~fans $\Sigma_{W_1}$, $\Sigma_{W_2}$, $\Sigma_{W_3}$ is
\begin{align*}
v_0&=(2,4,-1), & v_1&=(1,0,0), & v_2&=(0,0,1), & v_3&=(0,1,0), \\
v_4&=(0,0,-1), & v_5&=(0,-1,1), & v_6&=(-1,0,0);
\end{align*}
\item the~list of maximal cones of $\Sigma_{W_1}$ is
$$
[0,1,2], [0,2,3],  [0,1,4],  [0,3,4],  [1,4,5],  [1,2,5],  [2,3,6], [3,4,6], [4,5,6], [2,5,6];
$$
\item the~list of maximal cones of $\Sigma_{W_2}$ is
$$
[0,3,4], [0,4,6],  [0,1,2],  [0,2,3],  [0,1,4],  [1,4,5],  [1,2,5], [2,3,6], [4,5,6], [2,5,6];
$$
\item the~list of maximal cones of $\Sigma_{W_3}$ is
$$
[0,1,5], [0,2,5],  [0,3,6],  [0,4,6],  [0,2,3],  [0,1,4],  [1,4,5], [2,3,6], [4,5,6], [2,5,6].
$$
\end{itemize}
Then $W_1$, $W_2$, $W_3$ are projective, and there are small birational maps $W_0\dasharrow W_1\dasharrow W_2\dasharrow W_3$,
which can be illustrated by the~following self-explanatory toric diagrams:
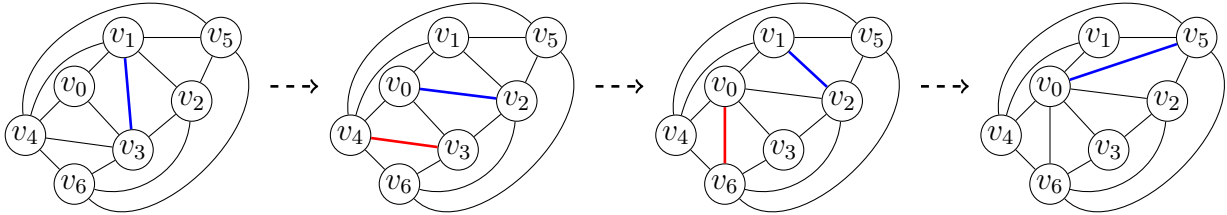
\begin{center}
\begin{tikzpicture}[node distance=1cm,auto, scale=0.65]
\node[state, inner sep=1pt,minimum size=0pt] (q_0) {$v_0$};
\node[state, inner sep=1pt,minimum size=0pt] (q_1) at (1,1) {$v_1$};
\node[state,inner sep=1pt,minimum size=0pt] (q_2) at (2.4,-.3) {$v_2$};
\node[state,inner sep=1pt,minimum size=0pt] (q_3) at (1.2,-1.3) {$v_3$};
\node[state,inner sep=1pt,minimum size=0pt] (q_4) at (-1,-1) {$v_4$};
\node[state,inner sep=1pt,minimum size=0pt] (q_5) at (3,1) {$v_5$};
\node[state,inner sep=1pt,minimum size=0pt] (q_6) at (0,-2) {$v_6$};

\node[state,inner sep=0pt,minimum size=0pt] (q_7) at (4,0) {};
\node[state,inner sep=0pt,minimum size=0pt] (q_8) at (5,0) {};
\path[->, dashed, line width=1pt] (q_7) edge node[swap] {} (q_8);

\path[-] (q_0) edge node[swap] {} (q_1);
\path[-] (q_0) edge node[swap] {} (q_3);
\path[-] (q_0) edge node[swap] {} (q_4);
\path[-] (q_1) edge node[swap] {} (q_2);
\path[-, blue, line width=1pt] (q_1) edge node[swap] {} (q_3);
\path[-] (q_1) edge node[swap] {} (q_5);
\path[-] (q_3) edge node[swap] {} (q_4);
\path[-] (q_4) edge node[swap] {} (q_6);
\path[-] (q_3) edge node[swap] {} (q_6);
\path[-] (q_2) edge node[swap] {} (q_3);
\path[-] (q_2) edge node[swap] {} (q_5);
\path[-, bend right = 30] (q_1) edge node {} (q_4);
\path[-, bend left = 70] (q_4) edge node {} (q_5);
\path[-, bend left = 50] (q_2) edge node {} (q_6);
\path[-, bend left = 90] (q_5) edge node {} (q_6);
\end{tikzpicture}
\begin{tikzpicture}[node distance=1cm,auto, scale=0.65]
\node[state, inner sep=1pt,minimum size=0pt] (q_0) {$v_0$};
\node[state, inner sep=1pt,minimum size=0pt] (q_1) at (1,1) {$v_1$};
\node[state,inner sep=1pt,minimum size=0pt] (q_2) at (2.4,-.3) {$v_2$};
\node[state,inner sep=1pt,minimum size=0pt] (q_3) at (1.2,-1.3) {$v_3$};
\node[state,inner sep=1pt,minimum size=0pt] (q_4) at (-1,-1) {$v_4$};
\node[state,inner sep=1pt,minimum size=0pt] (q_5) at (3,1) {$v_5$};
\node[state,inner sep=1pt,minimum size=0pt] (q_6) at (0,-2) {$v_6$};

\node[state,inner sep=0pt,minimum size=0pt] (q_7) at (4,0) {};
\node[state,inner sep=0pt,minimum size=0pt] (q_8) at (5,0) {};
\path[->, dashed, line width=1pt] (q_7) edge node[swap] {} (q_8);

\path[-] (q_0) edge node[swap] {} (q_1);
\path[-] (q_0) edge node[swap] {} (q_3);
\path[-] (q_0) edge node[swap] {} (q_4);
\path[-] (q_1) edge node[swap] {} (q_2);
\path[-, blue, line width=1pt] (q_0) edge node[swap] {} (q_2);
\path[-] (q_1) edge node[swap] {} (q_5);
\path[-, red, line width=1pt] (q_4) edge node[swap] {} (q_3);
\path[-] (q_4) edge node[swap] {} (q_6);
\path[-] (q_3) edge node[swap] {} (q_6);
\path[-] (q_2) edge node[swap] {} (q_3);
\path[-] (q_2) edge node[swap] {} (q_5);
\path[-, bend right = 30] (q_1) edge node {} (q_4);
\path[-, bend left = 70] (q_4) edge node {} (q_5);
\path[-, bend left = 50] (q_2) edge node {} (q_6);
\path[-, bend left = 90] (q_5) edge node {} (q_6);
\end{tikzpicture}
\begin{tikzpicture}[node distance=1cm,auto, scale=0.65]
\node[state, inner sep=1pt,minimum size=0pt] (q_0) {$v_0$};
\node[state, inner sep=1pt,minimum size=0pt] (q_1) at (1,1) {$v_1$};
\node[state,inner sep=1pt,minimum size=0pt] (q_2) at (2.4,-.3) {$v_2$};
\node[state,inner sep=1pt,minimum size=0pt] (q_3) at (1.2,-1.3) {$v_3$};
\node[state,inner sep=1pt,minimum size=0pt] (q_4) at (-1,-1) {$v_4$};
\node[state,inner sep=1pt,minimum size=0pt] (q_5) at (3,1) {$v_5$};
\node[state,inner sep=1pt,minimum size=0pt] (q_6) at (0,-2) {$v_6$};

\path[-] (q_0) edge node[swap] {} (q_1);
\path[-] (q_0) edge node[swap] {} (q_3);
\path[-] (q_0) edge node[swap] {} (q_4);
\path[-, blue, line width=1pt] (q_1) edge node[swap] {} (q_2);
\path[-] (q_0) edge node[swap] {} (q_2);
\path[-] (q_1) edge node[swap] {} (q_5);
\path[-, red, line width=1pt] (q_0) edge node[swap] {} (q_6);
\path[-] (q_4) edge node[swap] {} (q_6);
\path[-] (q_3) edge node[swap] {} (q_6);
\path[-] (q_2) edge node[swap] {} (q_3);
\path[-] (q_2) edge node[swap] {} (q_5);

\path[-, bend right = 30] (q_1) edge node {} (q_4);
\path[-, bend left = 70] (q_4) edge node {} (q_5);
\path[-, bend left = 50] (q_2) edge node {} (q_6);
\path[-, bend left = 90] (q_5) edge node {} (q_6);

\node[state,inner sep=0pt,minimum size=0pt] (q_7) at (4,0) {};
\node[state,inner sep=0pt,minimum size=0pt] (q_8) at (5,0) {};
\path[->, dashed, line width=1pt] (q_7) edge node[swap] {} (q_8);

\end{tikzpicture}
\begin{tikzpicture}[node distance=1cm,auto, scale=0.65]
\node[state, inner sep=1pt,minimum size=0pt] (q_0) {$v_0$};
\node[state, inner sep=1pt,minimum size=0pt] (q_1) at (1,1) {$v_1$};
\node[state,inner sep=1pt,minimum size=0pt] (q_2) at (2.4,-.3) {$v_2$};
\node[state,inner sep=1pt,minimum size=0pt] (q_3) at (1.2,-1.3) {$v_3$};
\node[state,inner sep=1pt,minimum size=0pt] (q_4) at (-1,-1) {$v_4$};
\node[state,inner sep=1pt,minimum size=0pt] (q_5) at (3,1) {$v_5$};
\node[state,inner sep=1pt,minimum size=0pt] (q_6) at (0,-2) {$v_6$};

\path[-] (q_0) edge node[swap] {} (q_1);
\path[-] (q_0) edge node[swap] {} (q_3);
\path[-] (q_0) edge node[swap] {} (q_4);
\path[-, blue, line width=1pt] (q_0) edge node[swap] {} (q_5);
\path[-] (q_0) edge node[swap] {} (q_2);
\path[-] (q_1) edge node[swap] {} (q_5);
\path[-] (q_0) edge node[swap] {} (q_6);
\path[-] (q_4) edge node[swap] {} (q_6);
\path[-] (q_3) edge node[swap] {} (q_6);
\path[-] (q_2) edge node[swap] {} (q_3);
\path[-] (q_2) edge node[swap] {} (q_5);

\path[-, bend right = 30] (q_1) edge node {} (q_4);
\path[-, bend left = 70] (q_4) edge node {} (q_5);
\path[-, bend left = 50] (q_2) edge node {} (q_6);
\path[-, bend left = 90] (q_5) edge node {} (q_6);
\end{tikzpicture}
\end{center}

As in Section~\ref{subsection:3-4-step-2}, let us use the~same notations for the~corresponding torus invariant divisors
and torus invariant curves on each 3-fold $W_i$.
Similarly, we will use the~same notation for the~strict transforms of the~divisor $L_u$ on each 3-fold $W_i$.
As in Section~\ref{subsection:3-4-step-2}, we see that
\begin{itemize}
\item $L_u$ is nef on $W_0$ for $u\in[0,1]$;
\item $L_u$ is nef on $W_1$ for $u\in[1,2]$;
\item $L_u$ is nef on $W_2$ for $u\in[2,3]$.
\end{itemize}
Moreover, the~Zariski decomposition of the~divisor $L_u$ exists on the~3-fold $W_2$ for each $u\in[3,5]$,
and the~Zariski decomposition exists on $W_3$ for $u\in[5,10]$.
Let us denote by $P(u)$ its positive part, and let us denote by $N(u)$ its negative part.
Then $P(u)=L_u-N(u)$, where
$$
N(u)=\left\{\aligned
&\frac{u-3}{4}T_3\ \text{ for }\ u\in[3,5],\\
&\frac{u-3}{4}T_3\ \text{ for }\ u\in[5,7],\\
&\frac{u-7}{3}T_2+\frac{u-4}{3}T_3\ \text{ for }\ u\in[7,8],\\
&\frac{u-8}{2}T_1+\frac{u-7}{3}T_2+\frac{u-4}{3}T_3\ \text{ for }\ u\in[8,10].
\endaligned
\right.
$$
Here, the~divisor $L_u-\frac{u-3}{4}T_3$ is nef on $W_2$ for $u\in[3,5]$,
and it is nef on $W_3$ for $u\in [5,7]$.

Now, let us consider a common partial toric resolution $\widetilde{W}$ of the~toric 3-folds $W_0$, $W_1$, $W_2$ and $W_3$.
Namely,~let $\widetilde{W}$ be the~toric 3-fold defined by the~simplicial fan $\Sigma_{\widetilde{W}}$ in $\mathbb{R}^3$ given by the~following~data:
\begin{itemize}
\item the~list of primitive generators of rays of $\Sigma_{\widetilde{W}}$ is
\begin{align*}
v_0&=(2,4,-1), & v_1&=(1,0,0), & v_2&=(0,0,1), & v_3&=(0,1,0), & v_4&=(0,0,-1), & v_5&=(0,-1,1), \\
v_6&=(-1,0,0), & v_7&=(0,4,-1), & v_8&=(1,2,0), & v_9&=(2,3,0), & v_{10}&=(2,0,3);
\end{align*}

\item the~list of maximal cones of $\Sigma_{\widetilde{W}}$ is
\begin{align*}
[0, 1,4]&, & [0,1,9]&, & [0,3,7]&, & [0,3,8]&, & [ 0,4,7]&, & [0,8,9]&, &  [1,4,5]&, & [1,5,10]&, &[1,9,10]&, \\
[2,3,6]&, & [2,3,8]&, & [2,5,6]&, & [2,5,10]&, & [2,8,10]&, & [3,6,7]&, & [4,5,6]&, & [4,6,7] &, & [8,9,10]&.
\end{align*}
\end{itemize}
The fan $\Sigma_{\widetilde{W}}$ can be diagramed as follows:
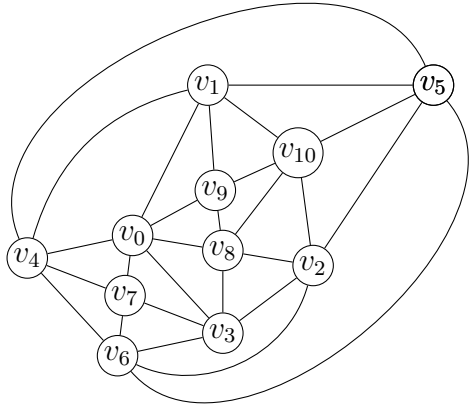
\begin{center}
\begin{tikzpicture}[node distance=1cm,auto]
\node[state, inner sep=1pt,minimum size=0pt] (q_1) at (1,2) {$v_1$};
\node[state,inner sep=1pt,minimum size=0pt] (q_2) at (2.4,-.4) {$v_2$};
\node[state,inner sep=1pt,minimum size=0pt] (q_3) at (1.2,-1.3) {$v_3$};
\node[state,inner sep=1pt,minimum size=0pt] (q_4) at (-1.4,-.3) {$v_4$};
\node[state,inner sep=1pt,minimum size=0pt] (q_5) at (4,2) {$v_5$};
\node[state,inner sep=1pt,minimum size=0pt] (q_6) at (-.2,-1.6) {$v_6$};
\node[state,inner sep=1pt,minimum size=0pt] (q_7) at (-.1,-.8) {$v_{7}$};
\node[state,inner sep=1pt,minimum size=0pt] (q_8) at (1.2,-.2) {$v_{8}$};
\node[state, inner sep=1pt,minimum size=0pt] (q_0) {$v_0$};
\node[state,inner sep=1pt,minimum size=0pt] (q_9) at (1.1,.6) {$v_{9}$};
\node[state,inner sep=1pt,minimum size=0pt] (q_10) at (2.2,1.1) {$v_{10}$};
\node[state,inner sep=1pt,minimum size=0pt] (q_5) at (4,2) {$v_5$};
\path[-] (q_0) edge node[swap] {} (q_1);
\path[-] (q_0) edge node[swap] {} (q_3);
\path[-] (q_0) edge node[swap] {} (q_4);
\path[-] (q_1) edge node[swap] {} (q_10);
\path[-] (q_8) edge node[swap] {} (q_3);
\path[-] (q_1) edge node[swap] {} (q_5);
\path[-] (q_4) edge node[swap] {} (q_7);
\path[-] (q_3) edge node[swap] {} (q_7);
\path[-] (q_0) edge node[swap] {} (q_7);
\path[-] (q_6) edge node[swap] {} (q_7);
\path[-] (q_4) edge node[swap] {} (q_6);
\path[-] (q_3) edge node[swap] {} (q_6);
\path[-] (q_2) edge node[swap] {} (q_3);
\path[-] (q_2) edge node[swap] {} (q_5);
\path[-] (q_0) edge node[swap] {} (q_8);
\path[-] (q_2) edge node[swap] {} (q_8);
\path[-] (q_10) edge node[swap] {} (q_9);
\path[-] (q_0) edge node[swap] {} (q_9);
\path[-] (q_1) edge node[swap] {} (q_9);
\path[-] (q_8) edge node[swap] {} (q_9);
\path[-] (q_8) edge node[swap] {} (q_10);
\path[-] (q_2) edge node[swap] {} (q_10);
\path[-] (q_5) edge node[swap] {} (q_10);
\path[-, bend right = 30] (q_1) edge node {} (q_4);
\path[-, bend left = 90] (q_4) edge node {} (q_5);
\path[-, bend left = 50] (q_2) edge node {} (q_6);
\path[-, bend left = 90] (q_5) edge node {} (q_6);
\end{tikzpicture}
\end{center}
Then there exists the~following toric commutative diagram
$$
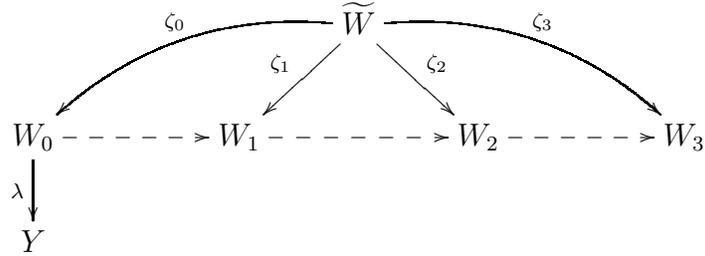
\xymatrix{
&&&\widetilde{W}\ar@/_1.4pc/@{->}[dlll]_{\zeta_0}\ar@{->}[dr]^{\zeta_2}\ar@{->}[dl]_{\zeta_1}\ar@/^1.4pc/@{->}[rrrd]^{\zeta_3}&&&\\%
W_0\ar@{-->}[rr]\ar@{->}[d]_{\lambda}&&W_1\ar@{-->}[rr]&&W_2\ar@{-->}[rr]&&W_3\\
Y&&&&&&}
$$
where $\zeta_0$, $\zeta_1$, $\zeta_2$, $\zeta_3$ are toric birational morphisms.

Let $\widetilde{T}_i$ be the~torus invariant divisor on $\widetilde{W}$ corresponding to the~ray $v_i$ in the~fan $\Sigma_{\widetilde{W}}$.
Then
$$
\begin{array}{ll}
\zeta^*_0(T_0)=\widetilde{T}_0, &
\zeta^*_0(T_1)=\widetilde{T}_1+\widetilde{T}_8+2\widetilde{T}_9+2\widetilde{T}_{10},\\
\zeta^*_0(T_2)=\widetilde{T}_2+3\widetilde{T}_{10}, &
\zeta^*_0(T_3)=\widetilde{T}_3+4\widetilde{T}_7+2\widetilde{T}_8+3\widetilde{T}_{9},\\
\zeta^*_1(T_0)=\widetilde{T}_0
+\frac{1}{2}\widetilde{T}_8+\frac{3}{4}\widetilde{T}_{9}, &
\zeta^*_1(T_1)=\widetilde{T}_1+\frac{1}{2}\widetilde{T}_9+2\widetilde{T}_{10},\\
\zeta^*_1(T_2)=\widetilde{T}_2+\frac{1}{2}\widetilde{T}_8
+\frac{3}{4}\widetilde{T}_{9}+3\widetilde{T}_{10}, &
\zeta^*_1(T_3)=\widetilde{T}_3+4\widetilde{T}_7,\\
\zeta^*_2(T_0)=\widetilde{T}_0+\widetilde{T}_7
+\frac{1}{2}\widetilde{T}_8+\frac{3}{4}\widetilde{T}_{9}, &
\zeta^*_2(T_1)=\widetilde{T}_1+\frac{1}{2}\widetilde{T}_9+2\widetilde{T}_{10},\\
\zeta^*_2(T_2)=\widetilde{T}_2+\frac{1}{2}\widetilde{T}_8
+\frac{3}{4}\widetilde{T}_{9}+3\widetilde{T}_{10}, &
\zeta^*_2(T_3)=\widetilde{T}_3,\\
\zeta^*_3(T_0)=\widetilde{T}_0+\widetilde{T}_7
+\frac{1}{2}\widetilde{T}_8+\widetilde{T}_9+\widetilde{T}_{10}, &
\zeta^*_3(T_1)=\widetilde{T}_1,\\
\zeta^*_3(T_2)=\widetilde{T}_2+\frac{1}{2}\widetilde{T}_{8}, &
\zeta^*_3(T_3)=\widetilde{T}_3.
\end{array}
$$

On the~3-fold $\widetilde{W}$, the~Zariski decomposition of the~divisor $\zeta_0^*(L_u)$ does exist for every $u\in[0,10]$.
Let~$P_{\widetilde{W}}(u)$ be its positive part, and let $N_{\widetilde{W}}(u)$ be its negative part.
We can express them as linear combinations of the~torus invariant divisors.
These expressions are presented in Table~\ref{table:Part3-ZDu}.

Fix the~quotient homomorphism $\mathbb{Z}^3\to\mathbb{Z}^3/\mathbb{Z}v_0\cong \mathbb{Z}^2$
such that $v_1\mapsto (1,0)$ and $v_3\mapsto (0,1)$.
Then  $\Sigma_{\widetilde{W}}$ is mapped to the~fan in $\mathbb{R}^2$
whose rays are generated by the~following vectors:
$$
w_1=(1,0),   w_2=(2,3),    w_3=(1,2),    w_4=(0,1),  w_5=(-1,0), w_6=(-1,-2).
$$
This two-dimensional fan defines the~surface $\widetilde{T}_0$.
Let $\zeta=\zeta_0|_{\widetilde{T}_0}\colon\widetilde{T}_0\to T_0$ be the~restriction~map.
Then $\zeta$ is described by a~map from the~fan of the~toric surface~$\widetilde{T}_0$ to the~fan of the~surface~$T_0$,
which can be illustrated by the~following toric picture:
 
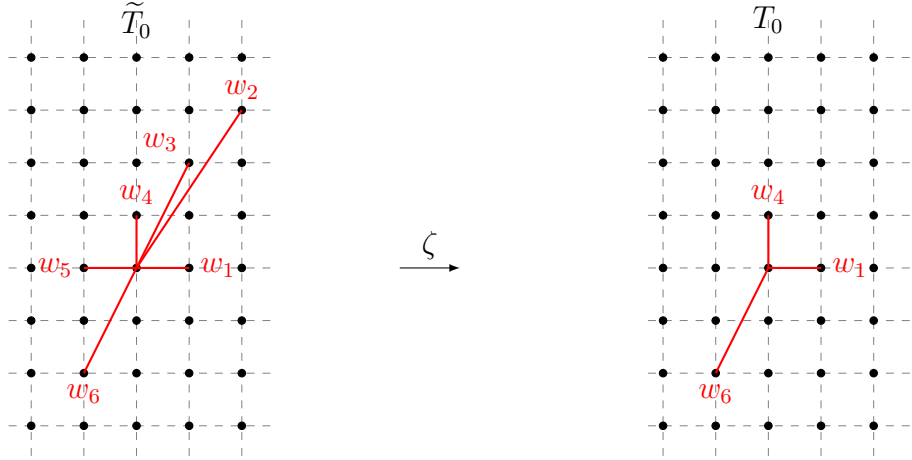
\begin{center}
\begin{tikzpicture}[>=latex]
    \begin{scope}
        \draw[style=help lines,dashed] (2.5,1.7) grid[step=.7cm] (6,7.5); 
    \foreach \x in {4,5,...,8}{                           
        \foreach \y in {3,4,...,10}{                       
        \node[draw,circle,inner sep=1pt,fill] at (.7*\x,.7*\y) {}; 
 }
}
\draw[thick, red,-] (4.2,4.2) -- (5.6,6.3)  node [above ] {$w_2$};
\draw[thick, red,-] (4.2,4.2) -- (4.2,4.9)  node [above ] {$w_4$};
\draw[thick, red,-] (4.2,4.2) -- (4.9,5.6)  node [above left] {$w_3$};
\draw[thick, red,-] (4.2,4.2) -- (4.9,4.2)  node [right] {$w_1$};
\draw[thick, red,-] (4.2,4.2) -- (3.5,2.8) node [below] {$w_6$};
\draw[thick, red,-] (4.2,4.2) -- (3.5,4.2) node [left] {$w_5$};
\draw[style=help lines,dashed] (11,1.7) grid[step=.7cm] (14.5,7.5);
\foreach \x in {16,17,...,20}{
        \foreach \y in {3,4,...,10}{
                \node[draw,circle,inner sep=1pt,fill] at (.7*\x,.7*\y) {};
        }
    }
\draw[thick, red,-] (12.6,4.2) -- (12.6,4.9)  node [above ] {$w_4$};
\draw[thick, red,-] (12.6,4.2) -- (13.3,4.2)  node [right] {$w_1$};
\draw[thick, red,-] (12.6,4.2) -- (11.9,2.8) node [below] {$w_6$};
\draw[->] (7.7,4.2) -- (8.5,4.2) node [above left] {};
     \node (O) at (8.1,4.5) {$\zeta$};
        \node (O) at (4.2,7.5) {$\widetilde{T}_0$};
           \node (O) at (12.6,7.5) {$T_0$};
\end{scope}
\end{tikzpicture}
\end{center}
It contracts the~curves of the~rays $w_5$, $w_3$, $w_2$ to points on the~surface $T_0$.

Let $\alpha_1,\ldots,\alpha_6$ be the~torus invariant curves in $\widetilde{T}_0$ defined by $w_1,\ldots,w_6$, respectively. Then
$$
\widetilde{T}_1\big|_{\widetilde{T}_0}=\alpha_1, \widetilde{T}_3\big|_{\widetilde{T}_0}=\alpha_4, \widetilde{T}_4\big|_{\widetilde{T}_0}=\frac{1}{2}\alpha_6, \widetilde{T}_7\big|_{\widetilde{T}_0}=\frac{1}{2}\alpha_5, \widetilde{T}_8\big|_{\widetilde{T}_0}=\alpha_3, \widetilde{T}_9\big|_{\widetilde{T}_0}=\alpha_2, \widetilde{T}_0\big|_{\widetilde{T}_0}=-\frac{1}{2}(\alpha_1+2\alpha_2+\alpha_3).
$$
Set $\overline{\alpha}_1=\zeta(\alpha_1)$, $\overline{\alpha}_4=\zeta(\alpha_4)$, $\overline{\alpha}_6=\zeta(\alpha_6)$.
Then $\overline{\alpha}_1=\{x=0\}$,  $\overline{\alpha}_4=\{y=0\}$, $\overline{\alpha}_6=n_G=\{z=0\}$.
Set $Q_{14}=\overline{\alpha}_1\cap \overline{\alpha}_4$, $Q_{16}=\overline{\alpha}_1\cap \overline{\alpha}_6$, $Q_{46}=\overline{\alpha}_4\cap \overline{\alpha}_6$.
Then $Q_{16}$ is the~singular point of the~surface~$G$.
Note that the~curve $R_G$ meets $\overline{\alpha}_1$ transversally at $Q_{14}$,
it meets the~curve $\overline{\alpha}_4$ transversally at two distinct points (one of them is $Q_{14}$),
and $R_G$ meets the~curve $\overline{\alpha}_6$ transversally at a single point, which is different from $Q_{16}$ and $Q_{46}$.
Let $Q_4$ be the~point in $R_G\cap\overline{\alpha}_4$ that is different from $Q_{14}$,
and let $Q_6$ be the~intersection point $R_G\cap \overline{\alpha}_6$.

Arguing as in Section~\ref{subsection:3-4-step-2},
we obtain the~following intersection matrix:
$$
A:=\left(\alpha_i\alpha_j\right)=
\begin{pmatrix}[1.4]
  -\frac{1}{6} & \frac{1}{3}& 0 &0 &0 &\frac{1}{2}\\
   \frac{1}{3} & -\frac{2}{3}& 1&0 &0& 0\\
    0 & 1& -2 &1&0&0\\
    0& 0& 1 &-2&1&0\\
     0 & 0& 0& 1 &-\frac{1}{2}&\frac{1}{2}\\
     \frac{1}{2} & 0&0&0& \frac{1}{2} &0\\
\end{pmatrix}
$$

Now, set $\widetilde{P}(u)=P_{\widetilde{W}}(u)|_{\widetilde{T}_0}$ and $\widetilde{N}(u)=N_{\widetilde{W}}(u)|_{\widetilde{T}_0}$.
We can express $\widetilde{P}(u)$ and $\widetilde{N}(u)$ as linear combinations of the~curves $\alpha_1$, $\alpha_2$, $\alpha_3$, $\alpha_4$, $\alpha_5$, $\alpha_6$.
These expressions are presented in Table~\ref{table:Part3-widetildePNu}.

Let $Q$ be a point in the~surface $G=T_0$, let $C$ be a smooth curve in $G$ that passes through~$Q$,
and let $\widetilde{C}$ be its proper transform on $\widetilde{T}_0$.
For every $u\in[0,10]$, let
$$
t(u)=\sup\Big\{v\in \mathbb R_{\geqslant 0} \ \big|\ \text{$\widetilde{P}(u)-vC$ is pseudoeffective}\Big\}.
$$
For every $v\in[0,t(u)]$, let $P(u,v)$ be the~positive part of the~Zariski decomposition of $\widetilde{P}(u)-vC$,
and let $N(u,v)$ be its negative part. Set
$$
S_L\big(W^{G}_{\bullet,\bullet};C\big)=\frac{3}{L^3}\int\limits_0^{10}\big(\widetilde{P}(u)\big)^2\mathrm{ord}_C\big(\widetilde{N}(u)\big)du+\frac{3}{L^3}\int\limits_0^{10}\int\limits_0^{t(u)}\big(P(u,v)\big)^2dvdu.
$$
Now, we write $\zeta^*(C)=\widetilde{C}+\Sigma$ for an effective $\mathbb{R}$-divisor $\Sigma$ on the~surface $\widetilde{T}_0$.
For every $u\in[0,10]$, write $\widetilde{N}(u)=d(u)C+N^\prime(u)$,
where $d(u)=\mathrm{ord}_C(\widetilde{N}(u))$, and $N^\prime(u)$ is an effective divisor~on~$\widetilde{T}_0$.~Set
$$
S\big(W_{\bullet, \bullet,\bullet}^{G,C};Q\big)=\frac{3}{L^3}
\int\limits_0^{10}\int\limits_0^{t(u)}\big(P(u,v)\cdot\widetilde{C}\big)^2dvdu+F_Q\big(W_{\bullet, \bullet,\bullet}^{G,C}\big)
$$
for
$$
F_Q\big(W_{\bullet, \bullet,\bullet}^{G,C}\big)=\frac{6}{L^3}\int\limits_0^{10}\int\limits_0^{t(u)}\big(P(u,v)\cdot\widetilde{C}\big)\cdot\mathrm{ord}_Q\Big(\big(N^\prime(u)+N(u,v)-(v+d(u))\Sigma\big)\big|_{\widetilde{C}}\Big)dvdu,
$$
where we consider $Q$ as a point in $\widetilde{C}$ using the~isomorphism $\widetilde{C}\cong C$ induced by $\zeta$.

If $C\not\subset\mathrm{Supp}(\Delta_G)$, we have
$(K_{G}+C+\Delta_G)\vert_{C}\sim_{\mathbb{R}}K_C+\Delta_C$, where $\Delta_C$ is an~effective divisor known as the~different.
If $C\subset\mathrm{Supp}(\Delta_G)$, we still can define the~different $\Delta_C$ using
$$
\big(K_{G}+C+\Delta_G-\mathrm{ord}_C(\Delta_G)C\big)\big\vert_{C}\sim_{\mathbb{R}}K_C+\Delta_C.
$$
The different $\Delta_C$ can be computed locally near any point in $C$.
Now, arguing as in Section~\ref{subsection:3-4-step-2}, we see
that to exclude the~case ($\mathbb{A}_3$),
it is enough to show that for every point $Q\in G$, there exists a smooth irreducible curve $C\subset G$ passing through $Q$
such that
\begin{equation}
\label{equation:3-4-A3-case-1}
S_L\big(W^{G}_{\bullet,\bullet};C\big)\leqslant A_{G,\Delta_G}(C)
\end{equation}
and
\begin{equation}
\label{equation:3-4-A3-case-2}
S(W_{\bullet,\bullet,\bullet}^{G,C};Q)\leqslant A_{C,\Delta_C}(Q).
\end{equation}
Let us do this in the~rest of this section, which would complete the~proof of Theorem~\ref{theorem:3-4-K-stable}.

Let $Q$ be a~point in $G=T_0\cong\mathbb{P}(1,2,1)$.
Let us choose the~curve $C$ as follows.
If $Q\in\overline{\alpha}_1\cup\overline{\alpha}_4\cup\overline{\alpha}_6$,
we let $C$ be a curve among $\overline{\alpha}_1$, $\overline{\alpha}_4$, $\overline{\alpha}_6$ that contains $Q$.
If $Q\not\in\overline{\alpha}_1\cup\overline{\alpha}_4\cup\overline{\alpha}_6$,
then there is a unique curve $\overline{\alpha}_0\in|\mathcal{O}_G(1)|$ that contains~$Q$.
In~this case, we let $C=\overline{\alpha}_0$, and we denote by $\alpha_0$ the~proper transform of the~curve $\overline{\alpha}_0$ on the~surface~$\widetilde{T}_0$.
Then $\Sigma$ and $\Delta_C$ can be described as follows:
\begin{itemize}
\item[($\overline{\alpha}_1$)] if $C=\overline{\alpha}_1$, then $\Sigma=2\alpha_2+\alpha_3$ and $\Delta_C=\frac{1}{2}Q_{16}+\frac{1}{2}Q_{14}$,
\item[($\overline{\alpha}_4$)] if $C=\overline{\alpha}_4$, then $\Sigma=3\alpha_2+2\alpha_3+2\alpha_5$ and $\Delta_C=\frac{1}{2}Q_{14}+\frac{1}{2}Q_4$,
\item[($\overline{\alpha}_6$)] if $C=\overline{\alpha}_6$, then $\Sigma=\alpha_5$ and $\Delta_C=\frac{3}{4}Q_{16}+\frac{1}{2}Q_6$,
\item[($\overline{\alpha}_0$)] if $C=\overline{\alpha}_0$, then $\Sigma=0$ and $\Delta_C=\Delta_{G}\vert_{C}+\frac{3}{4}Q_{16}$.
\end{itemize}
In the~last case, we have $\mathrm{ord}_Q(\Delta_C)\leqslant\frac{1}{2}$, because $\overline{\alpha}_0$ and $R_G$ meet transversally.

In each possible case, we compute $t(u)$ in Table~\ref{table:t-u-2}.

For each $u\in[0,10]$ and $v\in[0,t(u)]$, we can express
both divisors $P(u,v)$ and $N(u,v)$ as linear combinations of the~curves
$\alpha_1$, $\alpha_2$, $\alpha_3$, $\alpha_4$, $\alpha_5$, $\alpha_6$.
They are listed in Tables~\ref{table:Part3-ZDuv1},~\ref{table:Part3-ZDuv4},~\ref{table:Part3-ZDuv6},~\ref{table:Part3-ZDuv0}.

Now, arguing as in Section~\ref{subsection:3-4-step-2}, we compute
$$
S_L\big(W^{G}_{\bullet,\bullet};C\big)=
\left\{\aligned
&\frac{1}{2}\ \text{ if }\ C=\overline{\alpha}_1, \\
&\frac{7}{9}\ \text{ if }  C=\overline{\alpha}_4, \\
&\frac{2}{9}\ \text{ if }  C=\overline{\alpha}_6, \\
&\frac{3}{16}\ \text{ if }  C=\overline{\alpha}_0.
\endaligned
\right.
$$
This gives \eqref{equation:3-4-A3-case-1}. Note that $A_{G,\Delta_G}(\overline{\alpha}_6)=\frac{1}{2}$.

If $Q\in\overline{\alpha}_1\setminus\{Q_{14}\}$, let $C=\overline{\alpha}_1$, then $S_L(W^{G,\overline{\alpha}_1}_{\bullet,\bullet,\bullet};Q)=\frac{1}{9}$.
If $Q\in\overline{\alpha}_4\setminus\{Q_{46}\}$, let $C=\overline{\alpha}_4$, then
$$
S_L\big(W^{G, \overline{\alpha}_4}_{\bullet,\bullet,\bullet};Q\big)=
\left\{\aligned
&\frac{1}{2}\ \text{ if }\ Q=Q_{14},\\
&\frac{3}{16}\ \text{ if }\ Q\ne Q_{14}.
\endaligned
\right.
$$
If $Q\in\overline{\alpha}_6\setminus\{Q_{16}\}$, we let $C=\overline{\alpha}_6$, which gives
$$
S_L\big(W^{G, \overline{\alpha}_6}_{\bullet,\bullet,\bullet};Q\big)=
\left\{\aligned
&\frac{7}{9}\ \text{ if }\ Q=Q_{46},\\
&\frac{2}{9}\ \text{ if }\ Q\ne Q_{46}.
\endaligned
\right.
$$
If $Q\not\in\overline{\alpha}_1\cup\overline{\alpha}_4\cup\overline{\alpha}_6$, we let $C=\overline{\alpha}_0$, which gives
$S_L(W^{G, \overline{\alpha}_0}_{\bullet,\bullet,\bullet};Q)=\frac{1729}{6912}$.
In each case we get \eqref{equation:3-4-A3-case-2}. This excludes the~case ($\mathbb{A}_3$), and completes the~proof of Theorem~\ref{theorem:3-4-K-stable}.


\appendix

\section{Tables}
\label{section:tables}

\begin{longtable}{|c|c|cccccccc|}
\caption{Zariski decomposition of the~divisor $\zeta_0^*(L_u)$}
\label{table:ZDu}
\endfirsthead
\endhead
\hline
$u$
& $P_{\widetilde{W}}(u)$ \& $N_{\widetilde{W}}(u)$
&\ \  $\widetilde{T}_0$ \ \
&\ \  $\widetilde{T}_1$ \ \
&\ \  $\widetilde{T}_2$ \ \
&\ \  $\widetilde{T}_3$ \ \
&\ \  $\widetilde{T}_7$ \ \
&\ \  $\widetilde{T}_8$ \ \
&\ \  $\widetilde{T}_9$ \ \
&\ \  $\widetilde{T}_{10}$ \ \
\\
\hline
\hline
$[0,1]$&
 $P_{\widetilde{W}}(u)$ &
 $7-u$&$1$&$1$&$2$&$6$&$7$&$5$&$3$ \\
&
 $N_{\widetilde{W}}(u)$ &
 $0$&$0$&$0$&$0$&$0$&$0$&$0$&$0$ \\
  \hline
 $[1,2]$&
 $P_{\widetilde{W}}(u)$ &
 $7-u$&$1$&$1$&$2$&$7-u$&$8-u$&$\frac{17-2u}{3}$&$3$ \\
&
 $N_{\widetilde{W}}(u)$ &
 $0$&$0$&$0$&$0$&$u-1$&$u-1$&$\frac{2}{3}(u-1)$&$0$ \\
 \hline
 $[2,4]$&
 $P_{\widetilde{W}}(u)$ &
 $7-u$&$1$&$1$&$\frac{8-u}{3}$&$7-u$&$8-u$&$\frac{17-2u}{3}$&$3$ \\
&
 $N_{\widetilde{W}}(u)$ &
 $0$&$0$&$0$&$\frac{u-2}{3}$&$u-1$&$u-1$&$\frac{2}{3}(u-1)$&$0$ \\
 \hline
 $[4,5]$&
 $P_{\widetilde{W}}(u)$ &
 $7-u$&$1$&$1$&$\frac{8-u}{3}$&$7-u$&$8-u$&$7-u$&$7-u$ \\
&
 $N_{\widetilde{W}}(u)$ &
  $0$&$0$&$0$&$\frac{u-2}{3}$&$u-1$&$u-1$&$u-2$&$u-4$ \\
 \hline
 $[5,6]$&
 $P_{\widetilde{W}}(u)$ &
 $7-u$&$1$&$\frac{7-u}{2}$&$\frac{7-u}{2}$&$7-u$&$\frac{3}{2}(7-u)$&$7-u$&$7-u$ \\
&
  $N_{\widetilde{W}}(u)$ &
  $0$&$0$&$\frac{u-5}{2}$&$\frac{u-3}{2}$&$u-1$&$\frac{3u-7}{2}$&$u-2$&$u-4$ \\
 \hline
 $[6,7]$&
 $P_{\widetilde{W}}(u)$ &
 $7-u$&$7-u$&$\frac{7-u}{2}$&$\frac{7-u}{2}$&$7-u$&$\frac{3}{2}(7-u)$&$7-u$&$7-u$ \\
&
  $N_{\widetilde{W}}(u)$ &
  $0$&$u-6$&$\frac{u-5}{2}$&$\frac{u-3}{2}$&$u-1$&$\frac{3u-7}{2}$&$u-2$&$u-4$ \\
 \hline
\end{longtable}

\begin{longtable}{|c|c|cccccc|}
\caption{Expressions for $\widetilde{P}(u)$ and $\widetilde{N}(u)$}
\label{table:widetildePNu}
\endfirsthead
\endhead
\hline
$u$ &
&\ \  $\alpha_1$ \ \
&\ \  $\alpha_2$ \ \
&\ \  $\alpha_3$ \ \
&\ \  $\alpha_4$ \ \
&\ \  $\alpha_5$ \ \
&\ \  $\alpha_6$ \ \
\\
\hline
\hline
 $[0,1]$&
 $\widetilde{P}(u)$ &
 $u-6$&$u-2$&$u$&$2$&$6$&$0$ \\
&
 $\widetilde{N}(u)$  &
 $0$&$0$&$0$&$0$&$0$&$0$ \\
  \hline
 $[1,2]$&
 $\widetilde{P}(u)$  &
 $u-6$&$\frac{u-4}{3}$&$1$&$2$&$7-u$&$0$\\
&
 $\widetilde{N}(u)$  &
 $0$&$\frac{2}{3}(u-1)$&$u-1$&$0$&$u-1$&$0$ \\
 \hline
 $[2,4]$&
 $\widetilde{P}(u)$  &
 $u-6$&$\frac{u-4}{3}$&$1$&$\frac{8-u}{3}$&$7-u$&$0$ \\
&
 $\widetilde{N}(u)$  &
 $0$&$\frac{2}{3}(u-1)$&$u-1$&$\frac{u-2}{3}$&$u-1$&$0$ \\
 \hline
 $[4,5]$&
 $\widetilde{P}(u)$  &
 $u-6$&$0$&$1$&$\frac{8-u}{3}$&$7-u$&$0$ \\
&
 $\widetilde{N}(u)$  &
  $0$&$u-2$&$u-1$&$\frac{u-2}{3}$&$u-1$&$0$ \\
 \hline
 $[5,6]$&
 $\widetilde{P}(u)$  &
 $u-6$&$0$&$\frac{7-u}{2}$&$\frac{7-u}{2}$&$7-u$&$0$ \\
&
 $\widetilde{N}(u)$  &
  $0$&$u-2$&$\frac{3u-7}{2}$&$\frac{u-3}{2}$&$u-1$&$0$ \\
 \hline
 $[6,7]$&
 $\widetilde{P}(u)$  &
 $0$&$0$&$\frac{7-u}{2}$&$\frac{7-u}{2}$&$7-u$&$0$ \\
&
 $\widetilde{N}(u)$  &
  $u-6$&$u-2$&$\frac{3u-7}{2}$&$\frac{u-3}{2}$&$u-1$&$0$ \\
 \hline

\end{longtable}

\begin{longtable}{|c|c|c|c|c|c|c|}
\caption{Values of $t(u)$}
\label{table:t-u-1}
\endfirsthead
\endhead
  \hline
  \diagbox{$C$}{$u$}  & $[0,1]$  & $[1,2]$  & $[2,4]$ & $[4,5]$ & $[5,6]$ & $[6,7]$ \\
  \hline
  \hline
  $\overline{\alpha}_1$  & $u$ & $1$ & $1$  & $1$  & $1$  & $7-u$  \\
  \hline
  $\overline{\alpha}_4$  & $\frac{u}{3}$  &  $\frac{u}{3}$ &  $\frac{2}{3}$ & $\frac{2}{3}$  & $\frac{9-u}{6}$   & $\frac{7-u}{2}$   \\
  \hline
  $\overline{\alpha}_6$  &  $u$ & $1$  & $1$  & $1$  & $\frac{7-u}{2}$  & $\frac{7-u}{2}$  \\
  \hline
  $\overline{\alpha}_0$  &  $u$ & $1$  & $1$  & $\frac{7-u}{3}$  & $\frac{7-u}{3}$  & $\frac{7-u}{3}$  \\
  \hline
\end{longtable}

\begin{longtable}{|c|c|c|cccccc|}
\caption{\mbox{Expressions for $P(u,v)$ and $N(u,v)$ in the~case $C=\overline{\alpha}_1$}}
\label{table:ZDuv1}
\endfirsthead
\endhead
\hline
$u$& $v$& $P(u,v)$ \& $N(u,v)$
& $\alpha_1$
& $\alpha_2$
& $\alpha_3$
& $\alpha_4$
& $\alpha_5$
& $\alpha_6$
\\
\hline
\hline
 $[0,1]$& $[0,u]$&
 $P(u,v)$ &
 $u-6-v$&$u-2-v$&$u-v$&$2$&$6$&$0$ \\
& &
 $N(u,v)$  &
 $0$&$v$&$v$&$0$&$0$&$0$ \\
  \hline\hline
$[1,2]$& $[0,u-1]$&
 $P(u,v)$ &
 $u-6-v$&$\frac{u-4-v}{3}$&$1$&$2$&$7-u$&$0$\\
& &
 $N(u,v)$  &
 $0$&$\frac{v}{3}$&$0$&$0$&$0$&$0$ \\
 \hline
$[1,2]$& $[u-1,1]$&
 $P(u,v)$ &
 $u-6-v$&$u-2-v$&$u-v$&$2$&$7-u$&$0$\\
& &
 $N(u,v)$  &
 $0$&$\frac{3v-2u+2}{3}$&$v-u+1$&$0$&$0$&$0$ \\
 \hline\hline
$[2,4]$& $[0,1]$&
  $P(u,v)$ &
 $u-6-v$&$\frac{u-4-v}{3}$&$1$&$\frac{8-u}{3}$&$7-u$&$0$ \\
& &
 $N(u,v)$  &
 $0$&$\frac{v}{3}$&$0$&$0$&$0$&$0$ \\
 \hline\hline
$[4,5]$& $[0,u-4]$&
  $P(u,v)$ &
   $u-6-v$&$0$&$1$&$\frac{8-u}{3}$&$7-u$&$0$ \\
& &
$N(u,v)$  &
  $0$&$0$&$0$&$0$&$0$&$0$ \\
 \hline
$[4,5]$& $[u-4,1]$&
  $P(u,v)$ &
   $u-6-v$&$\frac{u-4-v}{3}$&$1$&$\frac{8-u}{3}$&$7-u$&$0$ \\
& &
$N(u,v)$  &
  $0$&$\frac{v+4-u}{3}$&$0$&$0$&$0$&$0$ \\
 \hline
\hline
$[5,6]$& $[0,1]$&
  $P(u,v)$ &
 $u-6-v$&$0$&$\frac{7-u}{2}$&$\frac{7-u}{2}$&$7-u$&$0$ \\
& &
 $N(u,v)$  &
  $0$&$0$&$0$&$0$&$0$&$0$ \\
 \hline\hline
$[6,7]$& $[0,7-u]$&
 $P(u,v)$ &
  $-v$&$0$&$\frac{7-u}{2}$&$\frac{7-u}{2}$&$7-u$&$0$ \\
& &
 $N(u,v)$  &
  $0$&$0$&$0$&$0$&$0$&$0$ \\
 \hline
\end{longtable}

\begin{longtable}{|c|c|c|cccccc|}
\caption{\mbox{Expressions for $P(u,v)$ and $N(u,v)$ in the~case $C=\overline{\alpha}_4$}}
\label{table:ZDuv4}
\endfirsthead
\endhead
\hline
$u$& $v$& $P(u,v)$ \& $N(u,v)$
& $\alpha_1$
& $\alpha_2$
& $\alpha_3$
& $\alpha_4$
& $\alpha_5$
& $\alpha_6$
\\
\hline
\hline
 $[0,1]$& $[0,\frac{u}{3}]$&
 $P(u,v)$ &
 $u-6$&$u-2-2v$&$u-3v$&$2-v$&$6-3v$&$0$ \\
& &
 $N(u,v)$  &
 $0$&$2v$&$3v$&$0$&$3v$&$0$ \\
  \hline\hline
$[1,2]$& $[0,\frac{u-1}{3}]$&
 $P(u,v)$ &
 $u-6$&$\frac{u-4}{3}$&$1$&$2-v$&$7-u$&$0$\\
& &
 $N(u,v)$  &
 $0$&$0$&$0$&$0$&$0$&$0$ \\
 \hline
$[1,2]$& $[\frac{u-1}{3},\frac{u}{3}]$&
 $P(u,v)$ &
 $u-6$&$u-2-2v$&$u-3v$&$2-v$&$6-3v$&$0$\\
& &
 $N(u,v)$  &
 $0$&$\frac{6v-2u+2}{3}$&$3v-u+1$&$0$&$3v-u+1$&$0$ \\
 \hline\hline
$[2,4]$& $[0,\frac{1}{3}]$&
  $P(u,v)$ &
 $u-6$&$\frac{u-4}{3}$&$1$&$\frac{8-u-3v}{3}$&$7-u$&$0$ \\
& &
 $N(u,v)$  &
 $0$&$0$&$0$&$0$&$0$&$0$ \\
 \hline
 $[2,4]$& $[\frac{1}{3},\frac{2}{3}]$&
  $P(u,v)$ &
 $u-6$&$\frac{u-2-6v}{3}$&$2-3v$&$\frac{8-u-3v}{3}$&$8-u-3v$&$0$ \\
& &
 $N(u,v)$  &
 $0$&$\frac{6v-2}{3}$&$3v-1$&$0$&$3v-1$&$0$ \\
 \hline\hline
$[4,5]$& $[0,\frac{5-u}{3}]$&
  $P(u,v)$ &
   $u-6$&$0$&$1$&$\frac{8-u-3v}{3}$&$7-u$&$0$ \\
& &
$N(u,v)$  &
  $0$&$0$&$0$&$0$&$0$&$0$ \\
 \hline
$[4,5]$& $[\frac{5-u}{3},\frac{1}{3}]$&
  $P(u,v)$ &
   $u-6$&$0$&$\frac{8-u-3v}{3}$&$\frac{8-u-3v}{3}$&$7-u$&$0$ \\
& &
$N(u,v)$  &
  $0$&$0$&$\frac{3v+u-5}{3}$&$0$&$0$&$0$ \\
  \hline
$[4,5]$& $[\frac{1}{3},\frac{u-2}{6}]$&
  $P(u,v)$ &
   $u-6$&$0$&$\frac{8-u-3v}{3}$&$\frac{8-u-3v}{3}$&$8-u-3v$&$0$ \\
& &
$N(u,v)$  &
  $0$&$0$&$\frac{3v+u-5}{3}$&$0$&$3v-1$&$0$ \\
   \hline
$[4,5]$& $[\frac{u-2}{6},\frac{2}{3}]$&
  $P(u,v)$ &
   $u-6$&$\frac{u-2-6v}{3}$&$2-3v$&$\frac{8-u-3v}{3}$&$8-u-3v$&$0$ \\
& &
$N(u,v)$  &
  $0$&$\frac{6v+2-u}{3}$&$3v-1$&$0$&$3v-1$&$0$ \\

 \hline
\hline
$[5,6]$& $[0,\frac{7-u}{6}]$&
  $P(u,v)$ &
 $u-6$&$0$&$\frac{7-u-2v}{2}$&$\frac{7-u-2v}{2}$&$7-u$&$0$ \\
& &
 $N(u,v)$  &
  $0$&$0$&$v$&$0$&$0$&$0$ \\
  \hline
  $[5,6]$& $[\frac{7-u}{6},\frac{1}{2}]$&
  $P(u,v)$ &
 $u-6$&$0$&$\frac{7-u-2v}{2}$&$\frac{7-u-2v}{2}$&$\frac{21-3u-6v}{2}$&$0$ \\
& &
 $N(u,v)$  &
  $0$&$0$&$v$&$0$&$\frac{6v+u-7}{2}$&$0$ \\
   \hline
  $[5,6]$& $[\frac{1}{2},\frac{9-u}{6}]$&
  $P(u,v)$ &
 $u-6$&$1-2v$&$\frac{9-u-6v}{2}$&$\frac{7-u-2v}{2}$&$\frac{21-3u-6v}{2}$&$0$ \\
& &
 $N(u,v)$  &
  $0$&$2v-1$&$3v-1$&$0$&$\frac{6v+u-7}{2}$&$0$ \\
 \hline\hline
$[6,7]$& $[0,\frac{7-u}{6}]$&
 $P(u,v)$ &
  $0$&$0$&$\frac{7-u-2v}{2}$&$\frac{7-u-2v}{2}$&$7-u$&$0$ \\
& &
 $N(u,v)$  &
  $0$&$0$&$v$&$0$&$0$&$0$ \\
  \hline
$[6,7]$& $[\frac{7-u}{6},\frac{7-u}{2}]$&
 $P(u,v)$ &
  $0$&$0$&$\frac{7-u-2v}{2}$&$\frac{7-u-2v}{2}$&$\frac{21-3u-6v}{2}$&$0$ \\
& &
 $N(u,v)$  &
  $0$&$0$&$v$&$0$&$\frac{6v+u-7}{2}$&$0$ \\
 \hline
\end{longtable}

\begin{longtable}{|c|c|c|cccccc|}
\caption{\mbox{Expressions for $P(u,v)$ and $N(u,v)$ in the~case $C=\overline{\alpha}_6$}}
\label{table:ZDuv6}
\endfirsthead
\endhead
\hline
$u$& $v$& $P(u,v)$ \& $N(u,v)$
& $\alpha_1$
& $\alpha_2$
& $\alpha_3$
& $\alpha_4$
& $\alpha_5$
& $\alpha_6$
\\
\hline
\hline
 $[0,1]$& $[0,u]$&
 $P(u,v)$ &
 $u-6$&$u-2$&$u$&$2$&$6-v$&$-v$ \\
& &
 $N(u,v)$  &
 $0$&$0$&$0$&$0$&$v$&$0$ \\
  \hline\hline
$[1,2]$& $[0,u-1]$&
 $P(u,v)$ &
 $u-6$&$\frac{u-4}{3}$&$1$&$2$&$7-u$&$-v$\\
& &
 $N(u,v)$  &
 $0$&$0$&$0$&$0$&$0$&$0$ \\
 \hline
$[1,2]$& $[u-1,1]$&
 $P(u,v)$ &
 $u-6$&$\frac{u-4}{3}$&$1$&$2$&$6-v$&$-v$\\
& &
 $N(u,v)$  &
 $0$&$0$&$0$&$0$&$v-u+1$&$0$ \\
 \hline\hline
$[2,4]$& $[0,1]$&
  $P(u,v)$ &
 $u-6$&$\frac{u-4}{3}$&$1$&$\frac{8-u}{3}$&$7-u$&$-v$ \\
& &
 $N(u,v)$  &
 $0$&$0$&$0$&$0$&$0$&$0$ \\
 \hline\hline
$[4,5]$& $[0,\frac{6-u}{2}]$&
  $P(u,v)$ &
   $u-6$&$0$&$1$&$\frac{8-u}{3}$&$7-u$&$-v$ \\
& &
$N(u,v)$  &
  $0$&$0$&$0$&$0$&$0$&$0$ \\
 \hline
$[4,5]$& $[\frac{6-u}{2},1]$&
  $P(u,v)$ &
   $-2v$&$0$&$1$&$\frac{8-u}{3}$&$7-u$&$-v$ \\
& &
$N(u,v)$  &
  $2v-6+u$&$0$&$0$&$0$&$0$&$0$ \\
 \hline
\hline
$[5,6]$& $[0,\frac{6-u}{2}]$&
  $P(u,v)$ &
 $u-6$&$0$&$\frac{7-u}{2}$&$\frac{7-u}{2}$&$7-u$&$-v$ \\
& &
 $N(u,v)$  &
  $0$&$0$&$0$&$0$&$0$&$0$ \\
  \hline
$[5,6]$& $[\frac{6-u}{2},\frac{7-u}{2}]$&
  $P(u,v)$ &
 $-2v$&$0$&$\frac{7-u}{2}$&$\frac{7-u}{2}$&$7-u$&$-v$ \\
& &
 $N(u,v)$  &
  $2v+u-6$&$0$&$0$&$0$&$0$&$0$ \\

 \hline\hline
$[6,7]$& $[0,\frac{7-u}{2}]$&
 $P(u,v)$ &
  $-2v$&$0$&$\frac{7-u}{2}$&$\frac{7-u}{2}$&$7-u$&$-v$ \\
& &
 $N(u,v)$  &
  $2v$&$0$&$0$&$0$&$0$&$0$ \\
\hline
\end{longtable}

\begin{longtable}{|c|c|c|cccccc|}
\caption{\mbox{Expressions for $P(u,v)$ and $N(u,v)$ in the~case $C=\overline{\alpha}_0$}}
\label{table:ZDuv0}
\endfirsthead
\endhead
\hline
$u$& $v$& $P(u,v)$ \& $N(u,v)$
& $\alpha_1$
& $\alpha_2$
& $\alpha_3$
& $\alpha_4$
& $\alpha_5$
& $\alpha_6$
\\
\hline
\hline
 $[0,1]$& $[0,u]$&
 $P(u,v)$ &
 $u-6-v$&$u-2-v$&$u-v$&$2$&$6$&$0$ \\
& &
 $N(u,v)$  &
 $0$&$0$&$0$&$0$&$0$&$0$ \\
  \hline\hline
$[1,2]$& $[0,2-u]$&
 $P(u,v)$ &
 $u-6-v$&$\frac{u-4-3v}{3}$&$1-v$&$2$&$7-u$&$0$\\
& &
 $N(u,v)$  &
 $0$&$0$&$0$&$0$&$0$&$0$ \\
 \hline
$[1,2]$& $[2-u,1]$&
 $P(u,v)$ &
 $u-6-v$&$\frac{u-4-3v}{3}$&$1-v$&$\frac{8-u-v}{3}$&$7-u$&$0$\\
& &
 $N(u,v)$  &
 $0$&$0$&$0$&$\frac{v+u-2}{3}$&$0$&$0$ \\
 \hline\hline
$[2,4]$& $[0,1]$&
  $P(u,v)$ &
 $u-6-v$&$\frac{u-4-3v}{3}$&$1-v$&$\frac{8-u-v}{3}$&$7-u$&$0$ \\
& &
 $N(u,v)$  &
 $0$&$0$&$0$&$\frac{v}{3}$&$0$&$0$ \\
 \hline\hline
$[4,5]$& $[0,5-u]$&
  $P(u,v)$ &
   $u-6-v$&$-v$&$1-v$&$\frac{8-u-v}{3}$&$7-u$&$0$ \\
& &
$N(u,v)$  &
  $0$&$0$&$0$&$\frac{v}{3}$&$0$&$0$ \\
 \hline
$[4,5]$& $[5-u,\frac{6-u}{2}]$&
  $P(u,v)$ &
   $u-6-v$&$-v$&$\frac{7-u-3v}{2}$&$\frac{7-u-v}{2}$&$7-u$&$0$ \\
& &
$N(u,v)$  &
  $0$&$0$&$\frac{u+v-5}{2}$&$\frac{u+3v-5}{6}$&$0$&$0$ \\
 \hline
$[4,5]$& $[\frac{6-u}{2},\frac{7-u}{3}]$&
  $P(u,v)$ &
   $-3v$&$-v$&$\frac{7-u-3v}{2}$&$\frac{7-u-v}{2}$&$7-u$&$0$ \\
& &
$N(u,v)$  &
  $u-6+2v$&$0$&$\frac{u+v-5}{2}$&$\frac{u+3v-5}{6}$&$0$&$0$ \\
 \hline
\hline
$[5,6]$& $[0,\frac{6-u}{2}]$&
  $P(u,v)$ &
 $u-6-v$&$-v$&$\frac{7-u-3v}{2}$&$\frac{7-u-v}{2}$&$7-u$&$0$ \\
& &
 $N(u,v)$  &
  $0$&$0$&$\frac{v}{2}$&$\frac{v}{2}$&$0$&$0$ \\
  \hline
$[5,6]$& $[\frac{6-u}{2},\frac{7-u}{3}]$&
  $P(u,v)$ &
 $-3v$&$-v$&$\frac{7-u-3v}{2}$&$\frac{7-u-v}{2}$&$7-u$&$0$ \\
& &
 $N(u,v)$  &
  $u-6+2v$&$0$&$\frac{v}{2}$&$\frac{v}{2}$&$0$&$0$ \\
 \hline\hline
$[6,7]$& $[0,\frac{7-u}{3}]$&
 $P(u,v)$ &
  $-3v$&$-v$&$\frac{7-u-3v}{2}$&$\frac{7-u-v}{2}$&$7-u$&$0$ \\
& &
 $N(u,v)$  &
  $2v$&$0$&$\frac{v}{2}$&$\frac{v}{2}$&$0$&$0$ \\
 \hline
\end{longtable}

\begin{longtable}{|c|c|cccccccc|}
\caption{Zariski decomposition of the~divisor $\zeta_0^*(L_u)$}
\label{table:Part3-ZDu}
\endfirsthead
\endhead
\hline
$u$& $P_{\widetilde{W}}(u)$ \& $N_{\widetilde{W}}(u)$
&\ \  $\widetilde{T}_0$ \ \
&\ \  $\widetilde{T}_1$ \ \
&\ \  $\widetilde{T}_2$ \ \
&\ \  $\widetilde{T}_3$ \ \
&\ \  $\widetilde{T}_7$ \ \
&\ \  $\widetilde{T}_8$ \ \
&\ \  $\widetilde{T}_9$ \ \
&\ \  $\widetilde{T}_{10}$ \ \
\\
\hline
\hline
 $[0,1]$&
 $P_{\widetilde{W}}(u)$ &
 $10-u$&$1$&$1$&$2$&$8$&$5$&$8$&$5$ \\
&
 $N_{\widetilde{W}}(u)$ &
 $0$&$0$&$0$&$0$&$0$&$0$&$0$&$0$ \\
  \hline
 $[1,2]$&
 $P_{\widetilde{W}}(u)$ &
 $10-u$&$1$&$1$&$2$&$8$&$\frac{11-u}{2}$&$\frac{35-3u}{4}$&$5$ \\
&
 $N_{\widetilde{W}}(u)$ &
 $0$&$0$&$0$&$0$&$0$&$\frac{u-1}{2}$&$\frac{3(u-1)}{4}$&$0$ \\
  \hline
 $[2,3]$&
 $P_{\widetilde{W}}(u)$ &
 $10-u$&$1$&$1$&$2$&$10-u$&$\frac{11-u}{2}$&$\frac{35-3u}{4}$&$5$ \\
&
 $N_{\widetilde{W}}(u)$ &
 $0$&$0$&$0$&$0$&$u-2$&$\frac{u-1}{2}$&$\frac{3(u-1)}{4}$&$0$ \\
  \hline
 $[3,5]$&
 $P_{\widetilde{W}}(u)$ &
 $10-u$&$1$&$1$&$\frac{11-u}{4}$&$10-u$&$\frac{11-u}{2}$&$\frac{35-3u}{4}$&$5$ \\
&
 $N_{\widetilde{W}}(u)$ &
 $0$&$0$&$0$&$\frac{u-3}{4}$&$u-2$&$\frac{u-1}{2}$&$\frac{3(u-1)}{4}$&$0$ \\
  \hline
 $[5,7]$&
 $P_{\widetilde{W}}(u)$ &
 $10-u$&$1$&$1$&$\frac{11-u}{4}$&$10-u$&$\frac{11-u}{2}$&$10-u$&$10-u$ \\
&
 $N_{\widetilde{W}}(u)$ &
 $0$&$0$&$0$&$\frac{u-3}{4}$&$u-2$&$\frac{u-1}{2}$&$u-2$&$u-5$ \\
  \hline
 $[7,8]$&
 $P_{\widetilde{W}}(u)$ &
 $10-u$&$1$&$\frac{10-u}{3}$&$\frac{10-u}{3}$&$10-u$&
 $\frac{2(10-u)}{3}$&$10-u$&$10-u$ \\
&
 $N_{\widetilde{W}}(u)$ &
 $0$&$0$&$\frac{u-7}{3}$&$\frac{u-4}{3}$&$u-2$&$\frac{2u-5}{3}$&$u-2$&$u-5$ \\
  \hline
 $[8,10]$&
 $P_{\widetilde{W}}(u)$ &
 $10-u$&$\frac{10-u}{2}$&$\frac{10-u}{3}$&$\frac{10-u}{3}$&$10-u$&
 $\frac{2(10-u)}{3}$&$10-u$&$10-u$ \\
&
 $N_{\widetilde{W}}(u)$ &
 $0$&$\frac{u-8}{2}$&$\frac{u-7}{3}$&$\frac{u-4}{3}$&$u-2$&$\frac{2u-5}{3}$&$u-2$&$u-5$
 \\
 \hline
\end{longtable}

\begin{longtable}{|c|c|cccccc|}
\caption{Expressions for $\widetilde{P}(u)$ and $\widetilde{N}(u)$}
\label{table:Part3-widetildePNu}
\endfirsthead
\endhead
\hline
$u$& $\widetilde{P}(u)$ \& $\widetilde{N}(u)$
&\ \  $\alpha_1$ \ \
&\ \  $\alpha_2$ \ \
&\ \  $\alpha_3$ \ \
&\ \  $\alpha_4$ \ \
&\ \  $\alpha_5$ \ \
&\ \  $\alpha_6$ \ \
\\
\hline
\hline
 $[0,1]$&
 $\widetilde{P}(u)$ &
 $\frac{u-8}{2}$&$u-2$&$\frac{u}{2}$&$2$&$4$&$0$ \\
&
 $\widetilde{N}(u)$  &
 $0$&$0$&$0$&$0$&$0$&$0$ \\
  \hline
 $[1,2]$&
 $\widetilde{P}(u)$  &
 $\frac{u-8}{2}$&$\frac{u-5}{4}$&$\frac{1}{2}$&$2$&$4$&$0$\\
&
 $\widetilde{N}(u)$  &
 $0$&$\frac{3(u-1)}{4}$&$\frac{u-1}{2}$&$0$&$0$&$0$ \\
  \hline
 $[2,3]$&
 $\widetilde{P}(u)$  &
 $\frac{u-8}{2}$&$\frac{u-5}{4}$&$\frac{1}{2}$&$2$&$\frac{10-u}{2}$&$0$\\
&
 $\widetilde{N}(u)$  &
 $0$&$\frac{3(u-1)}{4}$&$\frac{u-1}{2}$&$0$&$\frac{u-2}{2}$&$0$ \\
  \hline
 $[3,5]$&
 $\widetilde{P}(u)$  &
 $\frac{u-8}{2}$&$\frac{u-5}{4}$&$\frac{1}{2}$&$\frac{11-u}{4}$&$\frac{10-u}{2}$&$0$\\
&
 $\widetilde{N}(u)$  &
 $0$&$\frac{3(u-1)}{4}$&$\frac{u-1}{2}$&$\frac{u-3}{4}$&$\frac{u-2}{2}$&$0$ \\
  \hline
 $[5,7]$&
 $\widetilde{P}(u)$  &
 $\frac{u-8}{2}$&$0$&$\frac{1}{2}$&$\frac{11-u}{4}$&$\frac{10-u}{2}$&$0$\\
&
 $\widetilde{N}(u)$  &
 $0$&$u-2$&$\frac{u-1}{2}$&$\frac{u-3}{4}$&$\frac{u-2}{2}$&$0$ \\
  \hline
 $[7,8]$&
 $\widetilde{P}(u)$  &
 $\frac{u-8}{2}$&$0$&$\frac{10-u}{6}$&$\frac{10-u}{3}$&$\frac{10-u}{2}$&$0$\\
&
 $\widetilde{N}(u)$  &
 $0$&$u-2$&$\frac{2u-5}{3}$&$\frac{u-4}{3}$&$\frac{u-2}{2}$&$0$ \\
  \hline
 $[8,10]$&
 $\widetilde{P}(u)$  &
 $0$&$0$&$\frac{10-u}{6}$&$\frac{10-u}{3}$&$\frac{10-u}{2}$&$0$\\
&
 $\widetilde{N}(u)$  &
 $\frac{u-8}{2}$&$u-2$&$\frac{2u-5}{3}$&$\frac{u-4}{3}$&$\frac{u-2}{2}$&$0$ \\
\hline
\end{longtable}

\begin{longtable}{|c|c|c|c|c|c|c|c|c|}
\caption{Values of $t(u)$}
\label{table:t-u-2}
\endfirsthead
\endhead
  \hline
  \diagbox{$C$}{$u$}     &  $[0,1]$      & $[1,2]$       & $[2,3]$         & $[3,5]$       & [5,6]   & $[6,7]$        & $[7,8]$ & $[8,10]$  \\
  \hline
  \hline
  $\overline{\alpha}_1$  & $\frac{u}{2}$ & $\frac{u}{2}$ & $1$             & $1$            & $1$      & $1$            & $1$ & $\frac{10-u}{2}$ \\
  \hline
  $\overline{\alpha}_4$  & $\frac{u}{4}$ & $\frac{u}{4}$ & $\frac{u}{4}$   & $\frac{3}{4}$  &  $\frac{3}{4}$   & $\frac{3}{4}$  & $\frac{16-u}{12}$ & $\frac{10-u}{3}$ \\
  \hline
  $\overline{\alpha}_6$  & $\frac{u}{2}$ & $\frac{1}{2}$ &  $\frac{1}{2}$  & $\frac{1}{2}$  &  $\frac{1}{2}$   & $\frac{1}{2}$  & $\frac{10-u}{6}$ & $\frac{10-u}{6}$ \\
  \hline
  $\overline{\alpha}_0$  & $\frac{u}{2}$ & $\frac{1}{2}$ &  $\frac{1}{2}$  & $\frac{1}{2}$  &  $\frac{1}{2}$    & $\frac{10-u}{8}$  & $\frac{10-u}{8}$ & $\frac{10-u}{8}$ \\
  \hline
\end{longtable}

\begin{longtable}{|c|c|c|cccccc|}
\caption{\mbox{Expressions for $P(u,v)$ and $N(u,v)$ in the~case $C=\overline{\alpha}_1$}}
\label{table:Part3-ZDuv1}
\endfirsthead
\endhead
\hline
$u$& $v$& $P(u,v)$ \& $N(u,v)$
& $\alpha_1$
& $\alpha_2$
& $\alpha_3$
& $\alpha_4$
& $\alpha_5$
& $\alpha_6$
\\
\hline
\hline
 $[0,1]$& $[0,\frac{u}{2}]$&
 $P(u,v)$ &
 $\frac{u-8}{2}-v$&$u-2-2v$&$\frac{u}{2}-v$&$2$&$4$&$0$ \\
& &
 $N(u,v)$  &
 $0$&$2v$&$v$&$0$&$0$&$0$ \\
  \hline\hline
$[1,2]$& $[0,\frac{u-1}{2}]$&
 $P(u,v)$ &
 $\frac{u-8}{2}-v$&$\frac{u-5-2v}{4}$&$\frac{1}{2}$&$2$&$4$&$0$\\
& &
 $N(u,v)$  &
 $0$&$\frac{v}{2}$&$0$&$0$&$0$&$0$ \\
 \hline
$[1,2]$& $[\frac{u-1}{2},\frac{u}{2}]$&
 $P(u,v)$ &
 $\frac{u-8}{2}-v$&$u-2-2v$&$\frac{u-2v}{2}$&$2$&$4$&$0$\\
& &
 $N(u,v)$  &
 $0$&$\frac{3-3u+8v}{4}$&$\frac{2v-u+1}{2}$&$0$&$0$&$0$ \\
  \hline\hline
$[2,3]$& $[0,\frac{u-1}{2}]$&
 $P(u,v)$ &
 $\frac{u-8}{2}-v$&$\frac{u-5-2v}{4}$&$\frac{1}{2}$&$2$&$\frac{10-u}{2}$&$0$\\
& &
 $N(u,v)$  &
 $0$&$\frac{v}{2}$&$0$&$0$&$0$&$0$ \\
 \hline
$[2,3]$& $[\frac{u-1}{2},1]$&
 $P(u,v)$ &
 $\frac{u-8}{2}-v$&$u-2-2v$&$\frac{u-2v}{2}$&$2$&$\frac{10-u}{2}$&$0$\\
& &
 $N(u,v)$  &
 $0$&$\frac{3-3u+8v}{4}$&$\frac{2v-u+1}{2}$&$0$&$0$&$0$ \\
 \hline\hline
$[3,5]$& $[0,1]$&
  $P(u,v)$ &
 $\frac{u-8}{2}-v$&$\frac{u-5-2v}{4}$&$\frac{1}{2}$&$\frac{11-u}{4}$&
 $\frac{10-u}{2}$&$0$ \\
& &
 $N(u,v)$  &
 $0$&$\frac{v}{2}$&$0$&$0$&$0$&$0$ \\
 \hline\hline
$[5,7]$& $[0,\frac{u-5}{2}]$&
  $P(u,v)$ &
   $\frac{u-8}{2}-v$&$0$&$\frac{1}{2}$&$\frac{11-u}{4}$&$\frac{10-u}{2}$&$0$ \\
& &
$N(u,v)$  &
  $0$&$0$&$0$&$0$&$0$&$0$ \\
 \hline
$[5,7]$& $[\frac{u-5}{2},1]$&
  $P(u,v)$ &
  $\frac{u-8}{2}-v$&$\frac{u-5-2v}{4}$&$\frac{1}{2}$&$\frac{11-u}{4}$
  &$\frac{10-u}{2}$&$0$ \\
& &
$N(u,v)$  &
  $0$&$\frac{2v-u+5}{4}$&$0$&$0$&$0$&$0$ \\
 \hline
 \hline
$[7,8]$& $[0,1]$&
  $P(u,v)$ &
 $\frac{u-8}{2}-v$&$0$&$\frac{10-u}{6}$&$\frac{10-u}{3}$&$\frac{10-u}{2}$&$0$ \\
& &
 $N(u,v)$  &
  $0$&$0$&$0$&$0$&$0$&$0$ \\
 \hline\hline
$[8,10]$& $[0,\frac{10-u}{2}]$&
 $P(u,v)$ &
  $-v$&$0$&$\frac{10-u}{6}$&$\frac{10-u}{3}$&$\frac{10-u}{2}$&$0$ \\
& &
 $N(u,v)$  &
  $0$&$0$&$0$&$0$&$0$&$0$ \\
 \hline

\end{longtable}

\begin{longtable}{|c|c|c|cccccc|}
\caption{\mbox{Expressions for $P(u,v)$ and $N(u,v)$ in the~case $C=\overline{\alpha}_4$}}
\label{table:Part3-ZDuv4}
\endfirsthead
\endhead
\hline
$u$& $v$& $P(u,v)$ \& $N(u,v)$
& $\alpha_1$
& $\alpha_2$
& $\alpha_3$
& $\alpha_4$
& $\alpha_5$
& $\alpha_6$
\\
\hline
\hline
 $[0,1]$& $[0,\frac{u}{4}]$&
 $P(u,v)$ &
 $\frac{u-8}{2}$&$u-2-3v$&$\frac{u}{2}-2v$&$2-v$&$4-2v$&$0$ \\
& &
 $N(u,v)$  &
 $0$&$3v$&$2v$&$0$&$2v$&$0$ \\
  \hline\hline
$[1,2]$& $[0,\frac{u-1}{4}]$&
 $P(u,v)$ &
 $\frac{u-8}{2}$&$\frac{u-5}{4}$&$\frac{1}{2}$&$2-v$&$4-2v$&$0$\\
& &
 $N(u,v)$  &
 $0$&$0$&$0$&$0$&$2v$&$0$ \\
 \hline
$[1,2]$& $[\frac{u-1}{4},\frac{u}{4}]$&
 $P(u,v)$ &
 $\frac{u-8}{2}$&$u-2-3v$&$\frac{u}{2}-2v$&$2-v$&$4-2v$&$0$\\
& &
 $N(u,v)$  &
 $0$&$\frac{3(4v-u+1)}{4}$&$\frac{4v-u+1}{2}$&$0$&$2v$&$0$ \\
  \hline\hline
$[2,3]$& $[0,\frac{u-2}{4}]$&
 $P(u,v)$ &
 $\frac{u-8}{2}$&$\frac{u-5}{4}$&$\frac{1}{2}$&$2-v$&$\frac{10-u}{2}$&$0$\\
& &
 $N(u,v)$  &
 $0$&$0$&$0$&$0$&$0$&$0$ \\
 \hline
$[2,3]$& $[\frac{u-2}{4},\frac{u-1}{4}]$&
 $P(u,v)$ &
 $\frac{u-8}{2}$&$\frac{u-5}{4}$&$\frac{1}{2}$&$2-v$&$4-2v$&$0$\\
& &
 $N(u,v)$  &
 $0$&$0$&$0$&$0$&$\frac{4v-u+2}{2}$&$0$ \\
 \hline
$[2,3]$& $[\frac{u-1}{4},\frac{u}{4}]$&
 $P(u,v)$ &
 $\frac{u-8}{2}$&$u-2-3v$&$\frac{u}{2}-2v$&$2-v$&$4-2v$&$0$\\
& &
 $N(u,v)$  &
 $0$&$\frac{3(4v-u+1)}{4}$&$\frac{4v-u+1}{2}$&$0$&$\frac{4v-u+2}{2}$&$0$ \\
 \hline\hline
$[3,5]$& $[0,\frac{1}{4}]$&
  $P(u,v)$ &
 $\frac{u-8}{2}$&$\frac{u-5}{4}$&$\frac{1}{2}$&$\frac{11-u-4v}{4}$
 &$\frac{10-u}{2}$&$0$ \\
& &
 $N(u,v)$  &
 $0$&$0$&$0$&$0$&$0$&$0$ \\
 \hline
 $[3,5]$& $[\frac{1}{4},\frac{1}{2}]$&
  $P(u,v)$ &
 $\frac{u-8}{2}$&$\frac{u-5}{4}$&$\frac{1}{2}$&$\frac{11-u-4v}{4}$&
 $\frac{11-u-4v}{2}$&$0$ \\
& &
 $N(u,v)$  &
 $0$&$0$&$0$&$0$&$\frac{4v-1}{2}$&$0$ \\
 \hline
 $[3,5]$& $[\frac{1}{2},\frac{3}{4}]$&
  $P(u,v)$ &
 $\frac{u-8}{2}$&$\frac{u+1-12v}{4}$&$\frac{3-4v}{2}$&$\frac{11-u-4v}{4}$&
 $\frac{11-u-4v}{2}$&$0$ \\
& &
 $N(u,v)$  &
 $0$&$\frac{3(2v-1)}{2}$&$2v-1$&$0$&$\frac{4v-1}{2}$&$0$ \\
 \hline\hline
$[5,6]$& $[0,\frac{1}{4}]$&
  $P(u,v)$ &
   $\frac{u-8}{2}$&$0$&$\frac{1}{2}$&$\frac{11-u-4v}{4}$&$\frac{10-u}{2}$&$0$ \\
& &
$N(u,v)$  &
  $0$&$0$&$0$&$0$&$0$&$0$ \\
 \hline
$[5,6]$& $[\frac{1}{4},\frac{7-u}{4}]$&
  $P(u,v)$ &
   $\frac{u-8}{2}$&$0$&$\frac{1}{2}$&$\frac{11-u-4v}{4}$&$\frac{11-u-4v}{2}$&$0$ \\
& &
$N(u,v)$  &
  $0$&$0$&$0$&$0$&$\frac{4v-1}{2}$&$0$ \\
 \hline
$[5,6]$& $[\frac{7-u}{4},\frac{1+u}{12}]$&
  $P(u,v)$ &
   $\frac{u-8}{2}$&$0$&$\frac{11-u-4v}{8}$&$\frac{11-u-4v}{4}$
   &$\frac{11-u-4v}{2}$&$0$ \\
& &
$N(u,v)$  &
  $0$&$0$&$\frac{4v+u-7}{8}$&$0$&$\frac{4v-1}{2}$&$0$ \\
 \hline
$[5,6]$& $[\frac{1+u}{12},\frac{3}{4}]$&
  $P(u,v)$ &
   $\frac{u-8}{2}$&$\frac{1+u-12v}{4}$&$\frac{3-4v}{2}$&$\frac{11-u-4v}{4}$
   &$\frac{11-u-4v}{2}$&$0$ \\
& &
$N(u,v)$  &
  $0$&$\frac{12v-u-1}{4}$&$2v-1$&$0$&$\frac{4v-1}{2}$&$0$ \\
 \hline\hline
$[6,7]$& $[0,\frac{7-u}{4}]$&
  $P(u,v)$ &
   $\frac{u-8}{2}$&$0$&$\frac{1}{2}$&$\frac{11-u-4v}{4}$&$\frac{10-u}{2}$&$0$ \\
& &
$N(u,v)$  &
  $0$&$0$&$0$&$0$&$0$&$0$ \\
 \hline
$[6,7]$& $[\frac{7-u}{4},\frac{1}{4}]$&
  $P(u,v)$ &
   $\frac{u-8}{2}$&$0$&$\frac{11-u-4v}{8}$&$\frac{11-u-4v}{4}$&
   $\frac{10-u}{2}$&$0$ \\
& &
$N(u,v)$  &
  $0$&$0$&$\frac{4v+u-7}{8}$&$0$&$0$&$0$ \\
 \hline
$[6,7]$& $[\frac{1}{4},\frac{1+u}{12}]$&
  $P(u,v)$ &
   $\frac{u-8}{2}$&$0$&$\frac{11-u-4v}{8}$&$\frac{11-u-4v}{4}$
   &$\frac{11-u-4v}{2}$&$0$ \\
& &
$N(u,v)$  &
  $0$&$0$&$\frac{4v+u-7}{8}$&$0$&$\frac{4v-1}{2}$&$0$ \\
 \hline
$[6,7]$& $[\frac{1+u}{12},\frac{3}{4}]$&
  $P(u,v)$ &
   $\frac{u-8}{2}$&$\frac{1+u-12v}{4}$&$\frac{3-4v}{2}$&$\frac{11-u-4v}{4}$
   &$\frac{11-u-4v}{2}$&$0$ \\
& &
$N(u,v)$  &
  $0$&$\frac{12v-u-1}{4}$&$2v-1$&$0$&$\frac{4v-1}{2}$&$0$ \\
 \hline\hline
$[7,8]$& $[0,\frac{10-u}{12}]$&
  $P(u,v)$ &
 $\frac{u-8}{2}$&$0$&$\frac{10-u-3v}{6}$&$\frac{10-u-3v}{3}$&$\frac{10-u}{2}$&$0$ \\
& &
 $N(u,v)$  &
  $0$&$0$&$\frac{v}{2}$&$0$&$0$&$0$ \\
  \hline
  $[7,8]$& $[\frac{10-u}{12},\frac{2}{3}]$&
  $P(u,v)$ &
 $\frac{u-8}{2}$&$0$&$\frac{10-u-3v}{6}$&$\frac{10-u-3v}{3}$&
 $\frac{2(10-u-3v)}{3}$&$0$ \\
& &
 $N(u,v)$  &
  $0$&$0$&$\frac{v}{2}$&$0$&$\frac{12v+u-10}{6}$&$0$ \\
   \hline
  $[7,8]$& $[\frac{2}{3},\frac{16-u}{12}]$&
  $P(u,v)$ &
 $\frac{u-8}{2}$&$2-3v$&$\frac{16-u-12v}{6}$&$\frac{10-u-3v}{3}$&
 $\frac{2(10-u-3v)}{3}$&$0$ \\
& &
 $N(u,v)$  &
  $0$&$3v-2$&$2v-1$&$0$&$\frac{12v+u-10}{6}$&$0$ \\
 \hline\hline
$[8,10]$& $[0,\frac{10-u}{12}]$&
 $P(u,v)$ &
  $0$&$0$&$\frac{10-u-3v}{6}$&$\frac{10-u-3v}{3}$&$\frac{10-u}{2}$&$0$ \\
& &
 $N(u,v)$  &
  $0$&$0$&$\frac{v}{2}$&$0$&$0$&$0$ \\
  \hline
$[8,10]$& $[\frac{10-u}{12},\frac{10-u}{3}]$&
 $P(u,v)$ &
  $0$&$0$&$\frac{10-u-3v}{6}$&$\frac{10-u-3v}{3}$&
  $\frac{2(10-u-3v)}{3}$&$0$ \\
& &
 $N(u,v)$  &
  $0$&$0$&$\frac{v}{2}$&$0$&$\frac{12v+u-10}{6}$&$0$ \\
 \hline
 \end{longtable}

\begin{longtable}{|c|c|c|cccccc|}
\caption{\mbox{Expressions for $P(u,v)$ and $N(u,v)$ in the~case $C=\overline{\alpha}_6$}}
\label{table:Part3-ZDuv6}
\endfirsthead
\endhead
\hline
$u$& $v$& $P(u,v)$ \& $N(u,v)$
& $\alpha_1$
& $\alpha_2$
& $\alpha_3$
& $\alpha_4$
& $\alpha_5$
& $\alpha_6$
\\
\hline
\hline
 $[0,1]$& $[0,\frac{u}{2}]$&
 $P(u,v)$ &
 $\frac{u-8}{2}$&$u-2$&$\frac{u}{2}$&$2$&$4-v$&$-v$ \\
& &
 $N(u,v)$  &
 $0$&$0$&$0$&$0$&$v$&$0$ \\
\hline
\hline
 $[1,2]$& $[0,\frac{1}{2}]$&
 $P(u,v)$ &
 $\frac{u-8}{2}$&$\frac{u-5}{4}$&$\frac{1}{2}$&$2$&$4-v$&$-v$ \\
& &
 $N(u,v)$  &
 $0$&$0$&$0$&$0$&$v$&$0$ \\
  \hline\hline
$[2,3]$& $[0,\frac{u-2}{2}]$&
 $P(u,v)$ &
 $\frac{u-8}{2}$&$\frac{u-5}{4}$&$\frac{1}{2}$&$2$&$\frac{10-u}{2}$&$-v$\\
& &
 $N(u,v)$  &
 $0$&$0$&$0$&$0$&$0$&$0$ \\
 \hline
$[2,3]$& $[\frac{u-2}{2},\frac{1}{2}]$&
 $P(u,v)$ &
 $\frac{u-8}{2}$&$\frac{u-5}{4}$&$\frac{1}{2}$&$2$&
 $4-v$&$-v$\\
& &
 $N(u,v)$  &
 $0$&$0$&$0$&$0$&$\frac{2v+2-u}{2}$&$0$ \\
 \hline\hline
$[3,5]$& $[0,\frac{1}{2}]$&
 $P(u,v)$ &
 $\frac{u-8}{2}$&$\frac{u-5}{4}$&$\frac{1}{2}$&$\frac{11-u}{4}$&
 $\frac{10-u}{2}$&$-v$\\
& &
 $N(u,v)$  &
 $0$&$0$&$0$&$0$&$0$&$0$ \\
 \hline\hline
$[5,7]$& $[0,\frac{8-u}{6}]$&
 $P(u,v)$ &
 $\frac{u-8}{2}$&$0$&$\frac{1}{2}$&$\frac{11-u}{4}$&
 $\frac{10-u}{2}$&$-v$\\
& &
 $N(u,v)$  &
 $0$&$0$&$0$&$0$&$0$&$0$ \\
\hline
$[5,7]$& $[\frac{8-u}{6},\frac{1}{2}]$&
 $P(u,v)$ &
 $-3v$&$0$&$\frac{1}{2}$&$\frac{11-u}{4}$&
 $\frac{10-u}{2}$&$-v$\\
& &
 $N(u,v)$  &
 $\frac{6v+u-8}{2}$&$0$&$0$&$0$&$0$&$0$ \\
 \hline\hline
$[7,8]$& $[0,\frac{8-u}{6}]$&
 $P(u,v)$ &
 $\frac{u-8}{2}$&$0$&$\frac{10-u}{6}$&$\frac{10-u}{3}$&
 $\frac{10-u}{2}$&$-v$\\
& &
 $N(u,v)$  &
 $0$&$0$&$0$&$0$&$0$&$0$ \\
\hline
$[7,8]$& $[\frac{8-u}{6},\frac{10-u}{6}]$&
 $P(u,v)$ &
 $-3v$&$0$&$\frac{10-u}{6}$&$\frac{10-u}{3}$&
 $\frac{10-u}{2}$&$-v$\\
& &
 $N(u,v)$  &
 $\frac{6v+u-8}{2}$&$0$&$0$&$0$&$0$&$0$ \\
 \hline\hline
$[8,10]$& $[0,\frac{10-u}{6}]$&
 $P(u,v)$ &
 $-3v$&$0$&$\frac{10-u}{6}$&$\frac{10-u}{3}$&
 $\frac{10-u}{2}$&$-v$\\
& &
 $N(u,v)$  &
 $3v$&$0$&$0$&$0$&$0$&$0$ \\
 \hline

 \end{longtable}

\begin{longtable}{|c|c|c|cccccc|}
\caption{\mbox{Expressions for $P(u,v)$ and $N(u,v)$ in the~case $C=\overline{\alpha}_0$}}
\label{table:Part3-ZDuv0}
\endfirsthead
\endhead
\hline
$u$& $v$& $P(u,v)$ \& $N(u,v)$
& $\alpha_1$
& $\alpha_2$
& $\alpha_3$
& $\alpha_4$
& $\alpha_5$
& $\alpha_6$
\\
\hline
\hline
 $[0,1]$& $[0,\frac{u}{2}]$&
 $P(u,v)$ &
 $\frac{u-8}{2}-v$&$u-2-2v$&$\frac{u}{2}-v$&$2$&$4$&$0$ \\
& &
 $N(u,v)$  &
 $0$&$0$&$0$&$0$&$0$&$0$ \\
  \hline\hline
 $[1,2]$& $[0,\frac{1}{2}]$&
 $P(u,v)$ &
 $\frac{u-8}{2}-v$&$\frac{u-5-8v}{4}$&$\frac{1}{2}-v$&$2$&$4$&$0$ \\
& &
 $N(u,v)$  &
 $0$&$0$&$0$&$0$&$0$&$0$ \\
\hline\hline
 $[2,3]$& $[0,\frac{3-u}{2}]$&
 $P(u,v)$ &
 $\frac{u-8}{2}-v$&$\frac{u-5-8v}{4}$&$\frac{1}{2}-v$&$2$&$\frac{10-u}{2}$&$0$ \\
& &
 $N(u,v)$  &
 $0$&$0$&$0$&$0$&$0$&$0$ \\
\hline
 $[2,3]$& $[\frac{3-u}{2},\frac{1}{2}]$&
 $P(u,v)$ &
 $\frac{u-8}{2}-v$&$\frac{u-5-8v}{4}$&$\frac{1}{2}-v$&$\frac{11-u-2v}{4}$
 &$\frac{10-u}{2}$&$0$ \\
& &
 $N(u,v)$  &
 $0$&$0$&$0$&$\frac{2v+u-3}{4}$&$0$&$0$ \\
\hline\hline
 $[3,5]$& $[0,\frac{1}{2}]$&
 $P(u,v)$ &
 $\frac{u-8}{2}-v$&$\frac{u-5-8v}{4}$&$\frac{1}{2}-v$&$\frac{11-u-2v}{4}$
 &$\frac{10-u}{2}$&$0$ \\
& &
 $N(u,v)$  &
 $0$&$0$&$0$&$\frac{v}{2}$&$0$&$0$ \\
\hline\hline
 $[5,6]$& $[0,\frac{8-u}{6}]$&
 $P(u,v)$ &
 $\frac{u-8}{2}-v$&$-2v$&$\frac{1}{2}-v$&$\frac{11-u-2v}{4}$
 &$\frac{10-u}{2}$&$0$ \\
& &
 $N(u,v)$  &
 $0$&$0$&$0$&$\frac{v}{2}$&$0$&$0$ \\
\hline
 $[5,6]$& $[\frac{8-u}{6},\frac{1}{2}]$&
 $P(u,v)$ &
 $-4v$&$-2v$&$\frac{1}{2}-v$&$\frac{11-u-2v}{4}$
 &$\frac{10-u}{2}$&$0$ \\
& &
 $N(u,v)$  &
 $\frac{6v+u-8}{2}$&$0$&$0$&$\frac{v}{2}$&$0$&$0$ \\
\hline\hline
 $[6,\frac{13}{2}]$& $[0,\frac{8-u}{6}]$&
 $P(u,v)$ &
 $\frac{u-8}{2}-v$&$-2v$&$\frac{1}{2}-v$&$\frac{11-u-2v}{4}$
 &$\frac{10-u}{2}$&$0$ \\
& &
 $N(u,v)$  &
 $0$&$0$&$0$&$\frac{v}{2}$&$0$&$0$ \\
\hline
 $[6,\frac{13}{2}]$& $[\frac{8-u}{6},\frac{7-u}{2}]$&
 $P(u,v)$ &
 $-4v$&$-2v$&$\frac{1}{2}-v$&$\frac{11-u-2v}{4}$
 &$\frac{10-u}{2}$&$0$ \\
& &
 $N(u,v)$  &
 $\frac{6v+u-8}{2}$&$0$&$0$&$\frac{v}{2}$&$0$&$0$ \\
\hline
 $[6,\frac{13}{2}]$& $[\frac{7-u}{2},\frac{10-u}{8}]$&
 $P(u,v)$ &
 $-4v$&$-2v$&$\frac{10-u-8v}{6}$&$\frac{10-u-2v}{3}$
 &$\frac{10-u}{2}$&$0$ \\
& &
 $N(u,v)$  &
 $\frac{6v+u-8}{2}$&$0$&$\frac{2v+u-7}{6}$&$\frac{8v+u-7}{12}$&$0$&$0$ \\
\hline\hline
 $[\frac{13}{2},7]$& $[0,\frac{7-u}{2}]$&
 $P(u,v)$ &
 $\frac{u-8}{2}-v$&$-2v$&$\frac{1}{2}-v$&$\frac{11-u-2v}{4}$
 &$\frac{10-u}{2}$&$0$ \\
& &
 $N(u,v)$  &
 $0$&$0$&$0$&$\frac{v}{2}$&$0$&$0$ \\
\hline
 $[\frac{13}{2},7]$& $[\frac{7-u}{2},\frac{8-u}{6}]$&
 $P(u,v)$ &
 $\frac{u-8}{2}-v$&$-2v$&$\frac{10-u-8v}{6}$&$\frac{10-u-2v}{3}$
 &$\frac{10-u}{2}$&$0$ \\
& &
 $N(u,v)$  &
 $0$&$0$&$\frac{2v+u-7}{6}$&$\frac{8v+u-7}{12}$&$0$&$0$ \\
\hline
 $[\frac{13}{2},7]$& $[\frac{8-u}{6},\frac{10-u}{8}]$&
 $P(u,v)$ &
 $-4v$&$-2v$&$\frac{10-u-8v}{6}$&$\frac{10-u-2v}{3}$
 &$\frac{10-u}{2}$&$0$ \\
& &
 $N(u,v)$  &
 $\frac{6v+u-8}{2}$&$0$&$\frac{2v+u-7}{6}$&$\frac{8v+u-7}{12}$&$0$&$0$ \\
\hline\hline
 $[7,8]$& $[0,\frac{8-u}{6}]$&
 $P(u,v)$ &
 $\frac{u-8}{2}-v$&$-2v$&$\frac{10-u-8v}{6}$&$\frac{10-u-2v}{3}$
 &$\frac{10-u}{2}$&$0$ \\
& &
 $N(u,v)$  &
 $0$&$0$&$\frac{v}{3}$&$\frac{2v}{3}$&$0$&$0$ \\
\hline
 $[7,8]$& $[\frac{8-u}{6},\frac{10-u}{8}]$&
 $P(u,v)$ &
 $-4v$&$-2v$&$\frac{10-u-8v}{6}$&$\frac{10-u-2v}{3}$
 &$\frac{10-u}{2}$&$0$ \\
& &
 $N(u,v)$  &
 $\frac{6v+u-8}{2}$&$0$&$\frac{v}{3}$&$\frac{2v}{3}$&$0$&$0$ \\
\hline\hline
 $[8,10]$& $[0,\frac{10-u}{8}]$&
 $P(u,v)$ &
 $-4v$&$-2v$&$\frac{10-u-8v}{6}$&$\frac{10-u-2v}{3}$
 &$\frac{10-u}{2}$&$0$ \\
& &
 $N(u,v)$  &
 $3v$&$0$&$\frac{v}{3}$&$\frac{2v}{3}$&$0$&$0$ \\
 \hline
\end{longtable}

\end{document}